\documentclass[10pt,a4paper]{article}
\usepackage[svgnames]{xcolor}
\usepackage{mdframed}
\usepackage[utf8]{inputenc}
\usepackage{amsmath, amsthm, amssymb, bbm, bm}
\usepackage{stmaryrd}
\usepackage{tikz-cd}
\usepackage{enumerate}
\usepackage{appendix}

\usepackage[hidelinks]{hyperref}
\usepackage{color}

\definecolor{mygray}{gray}{0.75}

\newtheorem{tvrz}{Proposition}[section]
\newtheorem{lemma}[tvrz]{Lemma}
\newtheorem{theorem}[tvrz]{Theorem}
\newtheorem{cor}[tvrz]{Corollary}
\theoremstyle{definition}
\newtheorem{definice}[tvrz]{Definition}
\theoremstyle{remark}
\newtheorem{rem}[tvrz]{Remark}
\theoremstyle{definition}
\newtheorem{mdexample}[tvrz]{Example}
\newenvironment{example}%
{\begin{mdframed}[topline=false, rightline=false, bottomline=false, linewidth=0.2em, linecolor=mygray, innerleftmargin=0.5em, innerrightmargin=0,leftmargin=-0.7em]\begin{mdexample}}%
{\end{mdexample}\end{mdframed}}

\def\^{\wedge}
\def\<{\langle}
\def\>{\rangle}
\def\M{\mathcal{M}}

\def\N{\mathbb{N}}

\def\cN{\mathcal{N}}

\def\cS{\mathcal{S}}
\def\X{\mathfrak{X}}
\def\g{\mathfrak{g}}
\def\sp{\mathfrak{sp}}

\def\h{\mathfrak{h}}

\def\frP{\mathfrak{P}}
\def\frQ{\mathfrak{Q}}
\def\frF{\mathfrak{F}}
\def\frI{\mathfrak{I}}

\def\R{\mathbb{R}}

\def\C{\mathcal{C}}

\def\E{\mathcal{E}}

\def\G{\mathcal{G}}

\def\H{\mathcal{H}}
\def\Z{\mathbb{Z}}

\def\fA{\mathbf{A}}
\def\ol{\overline}

\def\frJ{\mathfrak{J}}

\def\frP{\mathfrak{P}}

\def\fB{\mathbf{B}}
\def\fT{\mathbf{T}}

\def\fg{\mathbf{g}}

\def\frM{\mathfrak{M}}

\def\fM{\mathbf{M}}

\def\bbA{\mathbb{A}}

\def\bbz{\mathbbm{z}}
\def\bby{\mathbbm{y}}
\def\bbx{\mathbbm{x}}

\def\bbu{\mathbbm{u}}
\def\bbv{\mathbbm{v}}

\def\bDelta{\blacktriangle}

\def\dm{\mathrm{m}}
\def\cD{\nabla}

\def\1{\mathbbm{1}}

\def\fG{\mathbf{G}}
\def\fF{\mathbf{F}}

\def\f1{\mathbf{1}}

\def\~{\widetilde}

\def\gl{\mathfrak{gl}}
\def\ao{\mathfrak{o}}
\def\sfm{\mathsf{m}}
\def\sfe{\mathsf{e}}
\def\sfi{\mathsf{i}}

\def\fbeta{\bm{\beta}}

\newcommand\ul[1]{\underline{#1}}
\newcommand\dia[1]{{#1}_{\diamond}}
\def\sfj{\mathsf{j}}
\newcommand\mind[2]{\lfloor \substack{#1 \\ #2} \rfloor}

\DeclareMathOperator{\gdim}{gdim}

\newcommand{\modu}[1]{\text{ (mod }#1)}

\DeclareMathOperator{\Set}{\mathbf{Set}}

\DeclareMathOperator{\gVect}{\mathbf{gVec}}

\DeclareMathOperator{\gMan}{\mathbf{gMan}}

\DeclareMathOperator{\Op}{\mathbf{Op}}

\DeclareMathOperator{\im}{im}

\DeclareMathOperator{\Sym}{Sym}

\DeclareMathOperator{\GL}{GL}

\DeclareMathOperator{\gO}{O}
\DeclareMathOperator{\gSp}{Sp}

\DeclareMathOperator{\Ad}{Ad}

\DeclareMathOperator{\op}{op}

\DeclareMathOperator{\Aut}{Aut}
\DeclareMathOperator{\Lin}{Lin}

\newcommand{\cH}[1]{{\color{SaddleBrown}{#1}}}
\begin{document}
\begin{flushright}
\today
% preprint number (if any)
\end{flushright}
\vspace{0.7cm}
\begin{center}
 %\vskip1cm

\baselineskip=13pt {\Large \bf{Three Examples of Graded Lie Groups}\\}
 \vskip0.5cm
 {\large{Jan Vysoký$^{1}$}}\\
 \vskip0.6cm
$^{1}$\textit{Faculty of Nuclear Sciences and Physical Engineering, Czech Technical University in Prague\\ Břehová 7, 115 19 Prague 1, Czech Republic, jan.vysoky@fjfi.cvut.cz}\\
\vskip0.3cm
\end{center}

\begin{abstract}
Lie theory is, beyond any doubt, an absolutely essential part of differential geometry. It is therefore necessary to seek its generalization to 
$\Z$-graded geometry. In particular, it is vital to construct non-trivial and explicit examples of graded Lie groups and their corresponding graded Lie algebras.

Three fundamental families of graded Lie groups are developed in this paper: the general linear group associated with any graded vector space, the graded orthogonal group associated with a graded vector space equipped with a metric, and the graded symplectic group associated with a graded vector space equipped with a symplectic form. We provide both a direct geometric construction and a functor-of-points perspective. It is shown that their corresponding Lie algebras are isomorphic to the anticipated subalgebras of the graded Lie algebra of linear endomorphisms. Isomorphisms of graded Lie groups induced by linear isomorphisms, as well as possible applications, are also discussed.
\end{abstract}


%\tableofcontents
\section*{Introduction}
\addcontentsline{toc}{section}{Introduction}
In supergeometry, Lie supergroups are of fundamental importance due to their utilization in supersymmetry. First mentioned already in \cite{BerzeinLeites1975}, there are basically three equivalent approaches to those. F. Berezin defines Lie supergroups in \cite{berezin1987superanalysis} as group objects in the category of supermanifolds. Equivalently, they can be viewed as super Hopf algebras on the Sweedler dual of the superalgebra of global functions. This viewpoint is due to B. Kostant in \cite{kostant1977graded}. Finally, they can be viewed as certain Lie group actions on Lie superalgebras, so called super Harish-Chandra pairs, see \cite{deligne1999notes} and \cite{koszul1982graded}. To understand the correspondence of the three viewpoints, see Chapter 7 of \cite{carmeli2011mathematical}. The most prominent example, the general linear supergroup $\GL(m|n)$, appears in all the above references. On the other hand, finite-dimensional simple Lie superalgebras were classified by V. Kac \cite{Kac1977, kac1977sketch}. Two classical series of this classification are formed by \textit{orthosymplectic Lie superalgebras} $\mathfrak{osp}(m|2n)$. The corresponding \textit{orthosymplectic Lie supergroups} $\text{OSp}(m|2n)$ were introduced in \cite{rittenberg2005guide}. Note that in the above literature, Lie supergroups are usually defined as groups of certain matrices valued in some auxiliary Grassmann algebra. In modern language, this means that they are defined by their functor of points evaluated on supermanifolds of the form $\R^{0|q}$. 

$\Z$-graded manifolds form a natural generalization of supergeometry. Recently $\Z$-graded manifolds with local coordinates of arbitrary degrees \cite{Vysoky:2022gm, kotov2024category, fairon2017introduction} were introduced, thus extending the theory of non-negatively graded (or just $\N$-graded) manifolds \cite{Kontsevich:1997vb, severa2001some, Voronov:2019mav} and \cite{mehta2006supergroupoids,2011RvMaP..23..669C}. In the following, \textit{graded} will always mean \textit{$\Z$-graded}. It is only natural to examine Lie theory in this category. For $\N$-graded manifolds, this was done in \cite{jubin2019differential} mostly using the language of Harish-Chandra pairs, and later extended to more general setting in \cite{kotov2023various}. In the master thesis of R. Šmolka \cite{Smolka2023}, graded Lie groups and elements of graded principal bundle theory were introduced. In particular, he constructed the pivotal example of the \textit{general linear group} associated with a graded vector space $\R^{(n_{j})}$. Note that this graded manifold \textit{always} requires one to use both positive and negative coordinate degrees, hence it never fits into the $\N$-graded setting. In the related category of $\Z_{2}^{n}$-graded manifolds, Lie groups were considered in \cite{bruce2020schwarz, bruce2021linear, bruce2025principal}. In particular, the appropriate general linear group was constructed using its functor of points. 

The realization that none of the above references contain a $\Z$-graded generalization of the orthosymplectic group $\text{OSp}(m|2n)$ was a main motivation for writing this paper. Our intention is to assign to any graded vector space $V$ equipped with a degree $\ell$ metric $g$ a certain graded Lie subgroup $\gO(V,g)$ of the \textit{general linear group} $\GL(V)$. For this, we need to find a more abstract and coordinate-independent construction of the general linear group. It turns out that this can be done in a rather straightforward way resembling the ordinary Lie theory. We only utilize the functor of points to avoid the explicit construction of the inverse. Since $\gO(V,g)$ is assigned in a canonical way to a metric $g$, we call it simply the \textit{graded orthogonal group}. It turned out that the construction can be kept very geometric without any digressions into abstract nonsense or graded Hopf algebras and Frechét topologies. In fact, it closely resembles the standard construction of $\gO(V,g)$ as a regular level set of the map $A \mapsto (g^{-1}A^{T}g)A$ corresponding to the value $\1_{V}$. The functor of points perspective is then offered as a bonus observation, though. We believe that the results of this paper can be utilized as valuable examples of not only graded Lie groups but of graded manifolds in general. It is also shown how the \textit{graded symplectic group} $\gSp(V,\omega)$ associated canonically with a symplectic from $\omega$ is constructed. We are aware of the fact that a graded version of the special linear group $\text{SL}(V)$ is still missing. This is because at the moment, we do not possess a suitable definition of a $\Z$-graded Berezinian. We plan to address this in the future. 

We make a conscious decision to \textit{not include} any section on the theory of $\Z$-graded manifolds. For basic notions, we refer the reader to \S 2 of our previous paper \cite{vsmolka2025threefold}. For a detailed introduction, see the review \cite{Vysoky:2022gm}. Non-elementary theory required in this paper includes mostly statements about submanifolds and their construction via pullbacks. Those can be found in \S 7 of the above reference. For introduction to graded Lie group theory, see the master thesis \cite{Smolka2023} of R. Šmolka. 

The paper is organized as follows. In Section \ref{sec_LA}, we introduce linear algebra of graded vector spaces equipped with a degree $\ell$ metric $g$. In particular, we show that a graded vector space $\gl(V)$ of all linear endomorphisms of $V$ decomposes as a direct sum of symmetric and skew-symmetric maps with respect to $g$, respectively. 

Section \ref{sec_GLG} begins by recalling the notion of graded Lie groups. We define a graded Lie subgroup and prove that it again forms a graded Lie group. It is then shown that its associated Lie algebra can be canonically identified with a subalgebra of the Lie algebra associated to the ``ambient'' graded Lie group. We finish this section by observing that group axioms can be checked on the set level by evaluating the functor of points. 

In Section \ref{sec_diamond}, we recall how every finite-dimensional graded real vector space $V$ can be viewed as a graded manifold which we denote as $\dia{V}$. For the purposes of this paper, it is convenient to view this construction as a functor $\diamond: \gVect \rightarrow \gMan^{\infty}$. We show how the tangent space to $\dia{V}$ at each point can be canonically identified with $V$ and utilize this to show that subspaces can be viewed as closed embedded submanifolds. We discuss how graded smooth maps into $\dia{V}$ are always uniquely determined by pullbacks of global coordinate functions. It is shown how bilinear maps can be naturally promoted to graded smooth maps of the corresponding graded smooth manifolds. We conclude this section by examining the functor of points associated with $\dia{V}$. 

Section \ref{sec_generallinear} is devoted to the construction of the most important example of a graded Lie group, namely the general linear group $\GL(V)$. It is a modified and perhaps more streamlined version of the original construction in \cite{Smolka2023}. The graded manifold $\GL(V)$ is defined as an open submanifold of $\dia{\gl(V)}$, where the underlying manifold $\gl(V)_{0}$ is restricted to the open subset $\GL(V_{\bullet})$ of invertible degree zero automorphisms of $V$. The group multiplication $\mu$ is obtained as a restriction of the graded smooth map assigned to the bilinear composition of graded linear maps, the group unit $e$ is defined to correspond to the point $\1_{V} \in \GL(V_{\bullet})$. In \cite{Smolka2023}, there is an explicit and rather complicated coordinate formula for the inverse mapping $\iota$. 

Admittedly, pure laziness led us to work around this issue by observing that the functor of points associated with $\GL(V)$ is naturally isomorphic to a functor assigning to each graded manifold $\cS$ a group of automorphisms $\Aut(\frM(\cS))$ of a certain graded module $\frM(\cS)$. This allows us to borrow the inverse from this group and use the Yoneda lemma to construct $\iota$. In fact, we can also immediately use this viewpoint to prove that $\GL(V)$ is a graded Lie group. Finally, we show that $\GL(V)$ deserves its name by finding its associated graded Lie algebra to be canonically isomorphic to $\gl(V)$ equipped with the graded commutator. 

The main novel subject of this paper, the \textit{graded orthogonal group} $\gO(V,g)$ associated with a graded vector space $V$ equipped with a degree $\ell$ metric $g$, is constructed in Section \ref{sec_GOVg}. We closely mimic the construction from ordinary differential geometry, albeit sometimes a bit in disguise. We divide the procedure into several steps:
\begin{enumerate}[(1)]
\item Construct an embedded submanifold $\Sym^{\times}(V,g)$ of both $\GL(V)$ and $\dia{\Sym(V,g)}$, a graded version of the set of invertible automorphisms of $V$ symmetric with respect to $g$. 
\item Find a graded smooth map $\tau^{\times}: \GL(V) \rightarrow \GL(V)$, a graded version of the map $A \mapsto g^{-1}A^{T}g$, and prove that it is an anti-homomorphism of graded Lie groups. 
\item Show that a graded smooth map $\varphi^{\times} = \mu \circ (\tau^{\times}, \1_{\GL(V)})$ lifts to a unique graded smooth map $\varphi: \GL(V) \rightarrow \Sym^{\times}(V,g)$. This is just a graded version of the map $A \mapsto (g^{-1}A^{T}g)A$. 
\item The unit $e: \{ \ast \} \rightarrow \GL(V)$ lifts to a map $e^{\times}: \{ \ast \} \rightarrow \Sym^{\times}(V,g)$. This just means that $\1_{V}$ is both invertible and symmetric with respect to $g$. We show that $\varphi$ is transversal to $e^{\times}$, or equivalently that $\1_{V} \in \Sym^{\times}_{0}(V,g)$ is a regular value of $\varphi$. 
\item Construct $\gO(V,g)$ as a regular level set submanifold of $\varphi$. This is a graded version of the orthogonality condition $(g^{-1} A^{T} g)A = \1_{V}$. Finally, prove that $\gO(V,g)$ is a graded Lie group and its associated Lie algebra can be identified with the subalgebra $\ao(V,g)$ of $\gl(V)$ consisting of endomorphisms of $V$ skew-symmetric with respect to $g$. 
\end{enumerate}
All these statements are proved in four parts of Section \ref{sec_technical}. This is the core mathematical part of the paper and can become a bit technical. To not obscure the overall picture, we have decided to move it into the separate section. 

In Section \ref{sec_OGFOP}, we provide an alternative perspective on the graded orthogonal group, namely the one of its functor of points. First, we show that the graded module $\frM(\cS)$ can be equipped with a $\C^{\infty}_{\cS}(S)$-bilinear form $\<\cdot,\cdot\>_{g}$ induced by the metric $g$ on $V$. This allows us to define a subgroup $\gO(\frM(\cS),g)$ of $\Aut(\frM(\cS))$ consisting of automorphisms preserving $\<\cdot,\cdot\>_{g}$. We then show that the functor of points associated with $\gO(V,g)$ is naturally isomorphic to the functor $\cS \mapsto \gO(\frM(\cS),g)$. 

Section \ref{sec_gradedsymplectic} is devoted to showing that a degree $\ell$ metric $g$ can be replaced by a degree $\ell$ \textit{symplectic form} $\omega$. The whole construction of Section \ref{sec_GOVg} can be repeated without any hiccups to obtain a \textit{graded symplectic group} $\gSp(V,\omega)$. Its Lie algebra is identified with the subalgebra $\sp(V,\omega)$ of $\gl(V)$ consisting of endomorphisms of $V$ skew-symmetric with respect to $\omega$. 

In Section \ref{sec_isomorphisms}, we investigate isomorphisms of graded vector spaces and graded smooth maps they induce between the respective graded Lie groups. One can consider isomorphisms of any degree. In the general linear case, every isomorphism induces a canonical isomorphism of general linear groups. This allows one to view the assignment $V \mapsto \GL(V)$ as a functor from a certain groupoid. As an interesting example, we show that the canonical degree shifting operator $\delta[m]: V[m] \rightarrow V$ induces an isomorphism of $\GL(V[m])$ and $\GL(V)$. It is only natural to assume that for a pair of graded vector spaces $(V,g)$ and $(W,g')$ equipped with metrics $g$ and $g'$, graded linear isometries induce isomorphisms of the respective graded orthogonal groups. However, this is only true for \textit{even degree} isomorphisms. For non-zero odd degrees, expected strange things happen. In particular, odd degree isomorphisms $M: V \rightarrow W$ can only relate metrics to symplectic forms and vice versa. Consequently, they induce graded Lie group isomorphisms e.g. of $\gO(V,g)$ and $\gSp(W,\omega')$. This is illustrated on three examples.

Finally, in Section \ref{sec_examples}, we discuss more explicit details and possible applications for graded Lie groups appearing in this paper. In Subsection \ref{subsec_standard} we show that to a graded vector space $V$ and a degree $\ell$ metric $g$, there is a convenient total basis for $V$, in which the metric $g$ takes a very simple ``standard form''. It turns out that that it is convenient to write $\ell = -2k + \epsilon$, where $k = - \lfloor \frac{\ell}{2} \rfloor$ and $\epsilon = \ell \modu{2}$. The standard form of $g$ is then discussed based on the values of $\epsilon \in \{0,1\}$ and $k \in \Z$. In Subsection \ref{subsec_OVG}, we show that for degree $\ell$ orthogonal graded vector bundles $(\E,\<\cdot,\cdot\>)$ over $\M$, defined in \cite{vysoky2022graded}, for any $m \in M$ there is always a convenient local frame for $\E$ over some $U \in \Op_{m}(M)$, in which $\<\cdot,\cdot\>$ takes a simple form. This generalizes the previous subsection. It then turns out that the group of automorphisms of $\E|_{U}$ preserving $\<\cdot,\cdot\>$ can then be identified with the group $\gMan^{\infty}(\M|_{U}, \GL(K))$, where $K \in \gVect$ is the typical fiber of $\E$. In Subsection \ref{subsec_ULG}, we use results of Subsection \ref{subsec_standard} to more carefully analyze underlying Lie groups of graded orthogonal groups. Finally, in Subsection \ref{subsec_repre}, we define left actions and representations of graded Lie groups and find the relation of these two notions. 

\section*{Acknowledgments}
I would like to thank Branislav Jurčo, Filip Moučka and Rudolf Šmolka for helpful discussions.

The research of was supported by grant GAČR 24-10031K. The author is also grateful for a financial support from MŠMT under grant no. RVO 14000. 

\section{Linear algebra} \label{sec_LA}
A graded vector space $V$ is a sequence $V = (V_{j})_{j \in \Z}$, where $V_{j}$ is an ordinary vector space for each $j \in \Z$. For the purposes of this paper, we will assume that all graded vector spaces are real and finite-dimensional, that is $\sum_{j \in \Z} \dim(V_{j}) < \infty$. We will denote the respective category as $\gVect$. By a graded dimension of $V \in \gVect$, we mean the sequence $\gdim(V) = ( \dim(V_{j}) )_{j \in \Z}$.

For any $V,W \in \gVect$, one writes $\Lin(V,W)$ for the set of degree zero graded linear maps and $\ul{\Lin}(V,W)$ for the graded vector space of graded linear maps of an arbitrary degree. $\gVect$ is monoidal with respect to the tensor product $\otimes_{\R}$ and $\ul{\Lin}$ is the corresponding internal hom-functor. 

The dual graded vector space is defined by $V^{\ast} := \ul{\Lin}(V,\R)$, where $\R$ is viewed as a trivially graded vector space. Note that $\gdim(V^{\ast}) = \neg \gdim(V)$, where we introduce the $\neg$ operator by
\begin{equation}
\neg (r_{j})_{j \in \Z} := (r_{-j})_{j \in \Z}
\end{equation}
on any sequence of objects $(r_{j})_{j \in \Z}$ labeled by integers. 

We write $\gl(V) := \ul{\Lin}(V,V)$. Let $[\cdot,\cdot]$ denote the graded commutator, that is 
\begin{equation} \label{eq_gcommutator}
[A,B] := A B - (-1)^{|A||B|} B A,
\end{equation} 
for all $A,B \in \gl(V)$. It is easy to see that $(\gl(V), [\cdot,\cdot])$ forms a graded Lie algebra of degree zero. Note that we omit any symbols for the composition of graded linear maps.

\begin{definice} \label{def_metric}
Let $V \in \gVect$. By a \textbf{symmetric bilinear form on $V$ of degree $\ell$}, we mean a graded bilinear map $g: V \times V \rightarrow \R$ satisfying 
\begin{enumerate}[(i)]
\item $|g(v,w)| = |v| + |w| + \ell$;
\item $g(v,w) = (-1)^{(|v| + \ell)(|w| + \ell)} g(w,v)$. 
\end{enumerate}
The grading on $\R$ is assumed to be trivial, so $g(v,w) = 0$ whenever $|v| + |w| + \ell \neq 0$. We say that a symmetric bilinear form $g$ of degree $\ell$ is a \textbf{degree $\ell$ metric on $V$}, if $g_{\flat} \in \ul{\Lin}(V,V^{\ast})$ defined for all $v,w \in V$ by the formula
\begin{equation} \label{eq_gflattog}
[g_{\flat}(v)](w) := (-1)^{(|v|+1)\ell} g(v,w) 
\end{equation}
is an isomorphism of graded vector spaces. Note that $|g_{\flat}| = \ell$. 
\end{definice}
\begin{rem} \label{rem_metric}
There are some remarks in order. 
\begin{enumerate}[(i)]
\item The graded symmetry of $g$ can be written in terms of $g_{\flat}$ as 
\begin{equation} \label{eq_gradedsymmetrygflat}
[g_{\flat}(v)](w) = (-1)^{|v||w| + \ell} [g_{\flat}(w)](v). 
\end{equation}
\item We will henceforth drop the $\flat$ subscript and use both notions interchangeably. 
\item By saying that $g: V \rightarrow V^{\ast}$ is an isomorphism of graded vector spaces, we mean that each its component $g_{j}: V_{j} \rightarrow (V^{\ast})_{j + \ell} = (V_{-(j+\ell)})^{\ast}$ is an isomorphism. In particular, if $(r_{j})_{j \in \Z} = \gdim(V)$, a degree $\ell$ metric on $V$ exists only if $r_{j} = r_{-(j+\ell)}$ for every $j \in \Z$. 
\item We will write $g^{-1} \in \ul{\Lin}(V^{\ast},V)$ for the inverse. Note that $|g^{-1}| = -\ell$ and $(g^{-1})_{j} = (g_{j-\ell})^{-1}$ for each $j \in \Z$.  
\end{enumerate}
\end{rem}
Recall that to any $A \in \gl(V)$, there is an induced transpose map $A^{T} \in \gl(V^{\ast})$ defined for all $\xi \in V^{\ast}$ and $v \in V$ by the formula
\begin{equation} \label{eq_transpose}
[A^{T}(\xi)](v) := (-1)^{|A||\xi|} \xi( A(v)).
\end{equation}
Observe that $|A^{T}| = |A|$. Note that for all $A,B \in \gl(V)$, one has 
\begin{equation} \label{eq_comptranspose}
(A B)^{T} = (-1)^{|A||B|} B^{T} A^{T}. 
\end{equation}
\begin{definice} \label{def_symantisym}
Let $g$ be a degree $\ell$ metric on $V \in \gVect$. We say that $A \in \gl(V)$ is \textbf{symmetric with respect to $g$}, if 
\begin{equation}
g A = (-1)^{\ell |A|} A^{T} g.
\end{equation}
Such maps form a subspace of $\gl(V)$ which we denote as $\Sym(V,g)$. We say that $A \in \gl(V)$ is \textbf{skew-symmetric with respect to $g$}, if 
\begin{equation}
g A = - (-1)^{\ell |A|} A^{T} g. 
\end{equation}
Such maps form a subspace of $\gl(V)$ which we denote as $\ao(V,g)$. 
\end{definice}
\begin{tvrz} \label{tvrz_taumap}
Let $g$ be a degree $\ell$ metric on $V \in \gVect$. Define $\tau: \gl(V) \rightarrow \gl(V)$ by 
\begin{equation} \label{eq_taumap}
\tau(A) := (-1)^{\ell|A|} g^{-1} A^{T} g.
\end{equation}
Then $\tau \in \gl( \gl(V))_{0}$, $\tau^{2} = \1_{\gl(V)}$, and $\Sym(V,g)$ and $\ao(V,g)$ are $\pm 1$ eigenspaces of $\tau$. 
\end{tvrz}
\begin{proof}
Clearly $\tau$ is graded linear of degree zero. The symmetry of $g$ and definitions imply 
\begin{equation}
[\tau(A)]^{T} = (-1)^{\ell|A|} gA g^{-1} .
\end{equation}
Consequently, one has 
\begin{equation}
\tau^{2}(A) = (-1)^{ \ell |\tau(A)|} g^{-1} [\tau(A)]^{T} g = (-1)^{\ell|A|} g^{-1} ( (-1)^{\ell|A|} g A g^{-1}) g = A. 
\end{equation}
The rest of the claims is obvious. 
\end{proof}
\begin{tvrz} \label{tvrz_Symortdecom} Let $g$ be a degree $\ell$ metric on $V \in \gVect$. 
\begin{enumerate}[(i)]
\item There is a direct sum decomposition $\gl(V) = \Sym(V,g) \oplus \ao(V,g)$;
\item $\ao(V,g)$ forms a graded Lie subalgebra of $(\gl(V),[\cdot,\cdot])$. 
\item The projector $p: \gl(V) \rightarrow \gl(V)$ onto the subspace $\ao(V,g)$ is given by the formula
\begin{equation} \label{eq_projectoronoVg}
p = \frac{1}{2}( \1_{\gl(V)} - \tau). 
\end{equation}
\end{enumerate}
\end{tvrz}
\begin{proof}
Straightforward verification. 
\end{proof}

\section{Graded Lie groups and subgroups} \label{sec_GLG}
A \textbf{graded Lie group} $(\G,\mu,\iota,e)$ is a group object in the category of graded manifolds $\gMan^{\infty}$. In other words, it is a graded manifold $\G$ together with the following graded smooth maps:
\begin{enumerate}[(i)]
\item a \textbf{multiplication} $\mu: \G \times \G \rightarrow \G$;
\item an \textbf{inverse}  $\iota: \G \rightarrow \G$; 
\item a \textbf{unit} $e: \{ \ast \} \rightarrow \G$, where $\{ \ast \} \in \gMan^{\infty}$ is a terminal object consisting of a one-point (graded) manifold. 
\end{enumerate}
They fit into the following commutative diagrams:
\begin{enumerate}[({g}1)]
\item The \textbf{associativity} diagram:
\begin{equation} \label{eq_associativity}
\begin{tikzcd}
\G \times (\G \times \G) \arrow{d} \arrow{r}{\1_{\G} \times \mu} & \G \times \G \arrow{dd}{\mu} \\
(\G \times \G) \arrow{d}{\mu \times \1_{\G}} \times \G &  \\
\G \times \G \arrow{r}{\mu} & \G
\end{tikzcd},
\end{equation}
where the unmarked arrow corresponds to the canonical graded diffeomorphism. 
\item The \textbf{unitarity} diagrams:
\begin{equation} \label{eq_unitarity}
\begin{tikzcd}
\G \arrow{rd}{\1_{\G}} \arrow{r}{(e_{\G},\1_{\G})} &[2em] \G \times \G \arrow{d}{\mu} \\
& \G 
\end{tikzcd}, \; \; 
\begin{tikzcd}
\G \arrow{rd}{\1_{\G}} \arrow{r}{(\1_{\G},e_{\G})} &[2em] \G \times \G \arrow{d}{\mu} \\
& \G 
\end{tikzcd},
\end{equation}
where for any $\M \in \gMan^{\infty}$, we write $e_{\M}$ for the composition $\begin{tikzcd}[column sep=small] \M \arrow{r} & \{ \ast \} \arrow{r}{e} &  \G \end{tikzcd}$. 
\item The \textbf{inversion} diagrams:
\begin{equation} \label{eq_inversion}
\begin{tikzcd}
\G \arrow{rd}{e_{\G}} \arrow{r}{(\1_{\G},\iota)} &[2em] \G \times \G \arrow{d}{\mu} \\
& \G 
\end{tikzcd}, \; \; 
\begin{tikzcd}
\G \arrow{rd}{e_{\G}} \arrow{r}{(\iota,\1_{\G})} &[2em] \G \times \G \arrow{d}{\mu} \\
& \G 
\end{tikzcd}.
\end{equation}
\end{enumerate}
It can be shown that $\iota$ and $e$, if they exist and fit into the respective diagrams, are unique. It is also clear that the underlying manifold $G$ together with the underlying smooth maps $\ul{\mu}: G \times G \rightarrow G$, $\ul{\iota}: G \rightarrow G$, and $\ul{e}: \{ \ast \} \rightarrow G$ form an ordinary Lie group. The choice of $e$ is equivalent to the choice of a single point $\ul{e}(\ast) \in G$ which we will often denote simply as $e$. 

Let us now define graded Lie subgroups. To keep things simple, we only consider those which form closed embedded submanifolds. This avoids some unnecessary technical difficulties. See \S 7 of \cite{Vysoky:2022gm} for details and further discussion. 
\begin{definice} 
Let $(\H,j)$ be a closed embedded submanifold of $\G$. We say that $(\H,j)$ is a \textbf{graded Lie subgroup} of $\G$, if there exist graded smooth maps $\mu': \H \times \H \rightarrow \H$ and $\iota': \H \rightarrow \H$ fitting into the diagrams
\begin{equation} \label{eq_subgroup}
\begin{tikzcd}
\H \times \H \arrow[dashed]{r}{\mu'} \arrow{d}{j \times j}& \H \arrow{d}{j}\\
\G \times \G \arrow{r}{\mu} & \G 
\end{tikzcd}, \; \; 
\begin{tikzcd}
\H \arrow{r}{\iota'} \arrow{d}{j} & \H \arrow{d}{j} \\
\G \arrow{r}{\iota} & \G 
\end{tikzcd}.
\end{equation}
\end{definice} 
Note that if such $\mu'$ and $\iota'$ exist, they are unique. This is because embeddings are monomorphisms in the category $\gMan^{\infty}$ of graded manifolds, see Proposition 7.28 in \cite{Vysoky:2022gm}. The basic observation is the following one:
\begin{tvrz} \label{tvrz_subgroupisgroup}
Let $(\H,j)$ be a graded Lie subgroup of $\G$. Then there exists $e': \{ \ast \} \rightarrow \H$ fitting into the diagram
\begin{equation} \label{eq_subgroupunit}
\begin{tikzcd}
\{ \ast \} \arrow[dashed]{r}{e'} \arrow{rd}{e} & \H \arrow{d}{j} \\
& \G
\end{tikzcd}
\end{equation}
Moreover, $(\H, \mu', \iota', e')$ forms a graded Lie group and $j: \H \rightarrow \G$ becomes a graded Lie group morphism. 
\end{tvrz}
\begin{proof}
In order to find $e'$, since $(\H,j)$ is a closed embedded submanifold, it suffices to argue that $e^{\ast}( \ker(j^{\ast})) = 0$, where $e^{\ast}$ and $j^{\ast}$ are pullbacks of global functions.  See Proposition 7.32 in \cite{Vysoky:2022gm}. Let $t: \H \rightarrow \{ \ast \}$ denote the unique arrow into the terminal object. Note that $t^{\ast}$ is always injective, so it suffices to prove that $0 = t^{\ast}( e^{\ast}(\ker(j^{\ast})) = (e \circ t)^{\ast}(\ker(j^{\ast})) = (e_{\G} \circ j)^{\ast}( \ker(j^{\ast}))$. Observe that $e \circ t = e_{\G} \circ j$, since both arrows factor through the terminal object and $e$. By utilizing (\ref{eq_inversion}) and (\ref{eq_subgroup}), we get
\begin{equation}
e_{\G} \circ j = \mu \circ (\1_{\g},\iota) \circ j = \mu \circ (j \times j) \circ (\1_{\H},\iota') = j \circ \mu' \circ (\1_{\H},\iota'). 
\end{equation}
Consequently, one has $(e_{\G} \circ j)^{\ast} = (\mu' \circ (\1_{\H},\iota'))^{\ast} \circ j^{\ast}$, so obviously $(e_{\G} \circ j)^{\ast}(\ker(j^{\ast})) = 0$. There thus exists a unique $e': \{ \ast \} \rightarrow \H$ fitting into (\ref{eq_subgroupunit}). 

Next, use (\ref{eq_associativity}) and definitions to form a diagram
\begin{equation}
\begin{tikzcd}
\H \times (\H \times \H) \arrow{dd} \arrow{rd}{j \times (j \times j)} \arrow{rrr}{\1_{\H} \times \mu'} &[2em] &[1.5em] & \H \times \H \arrow{ld}{j \times j} \arrow{dddd}{\mu'} \\
& \G \times (\G \times \G) \arrow{d} \arrow{r}{\1_{\G} \times \mu} & \G \times \G \arrow{dd}{\mu} & \\
(\H \times \H) \times \H \arrow{dd}{\mu' \times \1_{\H}} \arrow{r}{(j \times j) \times j} & (\G \times \G) \times \G \arrow{d}{\mu \times \1_{\G}} & & \\
& \G \times \G \arrow{r}{\mu} & \G & \\
\H \times \H \arrow{ur}{j \times j} \arrow{rrr}{\mu'} & & & \H \arrow{ul}{j}
\end{tikzcd},
\end{equation}
where the ``side squares'' and the middle diagram commute. This shows that the associativity diagram for $\H$ and $\mu'$ composed with $j: \H \rightarrow \G$ commutes. But $j$ is a monomorphism in $\gMan^{\infty}$ and the associativity diagram itself commutes. Other axioms are verified using the same arguments. $j$ being a graded Lie group morphism is the first diagram in (\ref{eq_subgroup}). 
\end{proof}

Next, let us discuss associated graded Lie algebras. Let $\G$ be a graded Lie group. Recall that a vector field $X$ on $\G$ is called \textbf{left-invariant}, if the induced vector field $1 \otimes X$ on $\G \times \G$ is $\mu$-related to $X$. More explicitly, one has 
\begin{equation}
(1 \otimes X)( \mu^{\ast}(f)) = \mu^{\ast}(X(f)) 
\end{equation} 
for any $f \in \C^{\infty}_{\G}(G)$. It is a straightforward check that left-invariant vector fields form a graded Lie subalgebra $\X^{L}_{\G}(G)$ of the graded Lie algebra $\X_{\G}(G)$ of global vector fields on $\G$. There is a canonical graded vector space isomorphism $\X^{L}_{\G}(G) \cong T_{e}\G$ given by $X \mapsto X|_{e}$. Conversely, the unique left-invariant vector field corresponding to $x \in T_{e}\G$ is denoted as $x^{L} \in \X_{\G}^{L}(G)$ and it is given by the formula
\begin{equation} \label{eq_leftinvariantexplicit}
x^{L} = (\1_{\G},e_{\G})^{\ast} \circ (1 \otimes \ol{x}) \circ \mu^{\ast},
\end{equation}
where $\ol{x} \in \X_{\G}(G)$ is an arbitrary vector field satisfying $\ol{x}|_{e} = x$. One writes $\g := T_{e}\G$. A graded Lie algebra bracket on $\g$ is then defined by the formula 
\begin{equation} \label{eq_Liebracket}
[x,y]_{\g}^{L} = [x^{L},y^{L}]|_{e},
\end{equation}
for all $x,y \in \g$. $(\g,[\cdot,\cdot]_{\g})$ is called the \textbf{graded Lie algebra associated with $\G$}. We refer the interested reader to Chapter 3 in \cite{Smolka2023} for details. 

Now, suppose that $(\H,j)$ is a graded Lie subgroup of $\G$. Let $\h$ be the graded Lie algebra associated with $\H$. What is its relation to $\g$? 

\begin{lemma} \label{lem_smfdVFextension}
Let $(\H,j)$ be a closed embedded submanifold of a graded manifold $\G$. Then to any $X \in \X_{\H}(H)$, there exists $Y \in \X_{\G}(G)$, such that $X \sim_{j} Y$. 
\end{lemma}
\begin{proof}
For any embedded submanifold $(\H,j)$, the tangent map $Tj: T\H \rightarrow T\G$ induces a graded vector bundle morphism $\hat{T}j: T\H \rightarrow T\G_{\H}$, where $T\G_{\H} := j^{\ast}(T\G)$. See Proposition 7.12 in \cite{Vysoky:2022gm}. It now follows from Proposition 7.5 and  Proposition 7.11 in \cite{Vysoky:2022gm} that there exists $Y \in \X_{\H}(H)$ such that $(\hat{T}j)(X) = j^{\ast}(Y)$. It is not difficult to see that this property is equivalent to $X \sim_{j} Y$. 
\end{proof}

\begin{tvrz} \label{tvrz_subalgebra}
Let $(\H,j)$ be a graded Lie subgroup of $\G$. Let $e'$ be the unit of $\H$. 

Then $\sfj := T_{e'}j: \h \rightarrow \g$ is a monomorphism of the associated graded Lie algebras. In other words, $\h$ can be identified with a graded Lie subalgebra of $\g$. 
\end{tvrz}
\begin{proof}
It suffices to argue that for any $x \in \h$, the left-invariant vector field $x^{L} \in \X_{\H}(H)$ is $j$-related to $\sfj(x)^{L} \in \X_{\G}(G)$. Indeed, the usual property of graded commutators then implies 
\begin{equation}
[x^{L},y^{L}] \sim_{j} [\sfj(x)^{L},\sfj(y)^{L}],
\end{equation}
for all $x,y \in \h$. Evaluating this at $e' \in H$ gives the required property
\begin{equation}
\sfj([x,y]_{\h}) = [\sfj(x),\sfj(y)]_{\g}.
\end{equation}
Note that $\sfj$ is always injective as $j$ is an embedding. Recall that the action of $\sfj(x)^{L} \in \X_{\G}(G)$ can be also written as a composition
\begin{equation} \label{eq_sfjxLjinak}
\sfj(x)^{L} = (\1_{\G},e_{\G})^{\ast} \circ (1 \otimes Y) \circ \mu^{\ast},
\end{equation}
where $Y \in \X_{\G}(G)$ is an \textit{arbitrary} vector field on $\G$ satisfying $Y|_{e} = \sfj(x)$. See Theorem 3.16 in \cite{Smolka2023}. By Lemma \ref{lem_smfdVFextension}, we can choose it to satisfy $x^{L} \sim_{j} Y$. Note that then indeed
\begin{equation}
Y|_{e} = (T_{e'}j)(x^{L}|_{e'}) = \sfj(x). 
\end{equation}
To prove that $x^{L} \sim_{j} \sfj(x)^{L}$, we must verify that $j^{\ast} \circ \sfj(x)^{L} = x^{L} \circ j^{\ast}$. Using (\ref{eq_sfjxLjinak}), we get
\begin{equation}
j^{\ast} \circ \sfj(x)^{L} = j^{\ast} \circ (\1_{\G},e_{\G})^{\ast} \circ (1 \otimes Y) \circ \mu^{\ast} = (\1_{\H},e'_{\H})^{\ast} \circ (j \times j)^{\ast} \circ (1 \otimes Y) \circ \mu^{\ast},
\end{equation}
where we have used $(\1_{\G},e_{\G}) \circ j = (j \times j) \circ (\1_{\H},e'_{\H})$ following from (\ref{eq_subgroupunit}). Next, the fact that $x^{L} \sim_{j} Y$ implies that $1 \otimes x^{L} \sim_{j \times j} 1 \otimes Y$, so we can write
\begin{equation}
\begin{split}
(\1_{\H},e'_{\H})^{\ast} \circ (j \times j)^{\ast} \circ (1 \otimes Y) \circ \mu^{\ast} = & \ (\1_{\H},e'_{\H})^{\ast} \circ (1 \otimes x^{L}) \circ (j \times j)^{\ast} \circ \mu^{\ast} \\
= & \ \big( (\1_{\H},e'_{\H})^{\ast} \circ (1 \otimes x^{L}) \circ \mu'^{\ast} \big) \circ j^{\ast} \\
= & \ x^{L} \circ j^{\ast}.
\end{split}
\end{equation}
We have utilized the first diagram in (\ref{eq_subgroup}) and the analogue of (\ref{eq_sfjxLjinak}) for $x^{L}$. 
\end{proof}
Instead of verifying the axioms (g1) - (g3) directly, it is often useful to use the following:
\begin{tvrz} \label{tvrz_gradedLGintermsofsets}
Let $\G$ be a graded manifold together with graded smooth maps  $\mu: \G \times \G \rightarrow \G$, $\iota: \G \rightarrow \G$ and $e: \{ \ast \} \rightarrow \G$. For every graded manifold $\cS$, form a set $\frP(\cS) := \gMan^{\infty}(\cS,\G)$ and the following maps: 
\begin{enumerate}[(i)]
\item A map $\sfm_{\cS}: \frP(\cS) \times \frP(\cS) \rightarrow \frP(\cS)$ given by the formula 
\begin{equation}
\dm_{\cS}(\phi,\phi') := \mu \circ (\phi,\phi'),
\end{equation}
for all $\phi,\phi' \in \frP(\cS)$. 
\item A map $\sfi_{\cS}: \frP(\cS) \rightarrow \frP(\cS)$ given by the formula 
\begin{equation}
\sfi_{\cS}(\phi) := \iota \circ \phi,
\end{equation}
for every $\phi \in \frP(\cS)$. 
\item A mapping $\sfe_{\cS}: \{ \ast \} \rightarrow \frP(\cS)$ from the one-point-set, defined by $\sfe_{\cS}(\ast) = e_{\cS}$, see under (\ref{eq_unitarity}). 
\end{enumerate}
Then $(\G,\mu,\iota,e)$ is a graded Lie group, iff $(\frP(\cS), \sfm_{\cS}, \sfi_{\cS}, \sfe_{\cS})$ is a group for all $\cS \in \gMan^{\infty}$. 
\end{tvrz}
\begin{proof}
This is the standard property of group objects, see e.g. Chapter III \S 6 of \cite{mac1998categories}. 
\end{proof}
\section{Diamond functor, bilinear maps} \label{sec_diamond}
Recall that to any finite-dimensional graded vector space, there is an associated graded manifold. In fact, this assignment is very well-behaved.
\begin{tvrz} \label{tvrz_diamant}
To any $V \in \gVect$, there is an associated graded manifold $\dia{V} \in \gMan^{\infty}$. Its underlying manifold is the zero component $V_{0}$ with the standard smooth structure, and 
\begin{equation}
\gdim(\dia{V}) = \neg \gdim(V). 
\end{equation}
Let $V,W \in \gVect$. To any $A \in \Lin(V,W)$, there is an associated graded smooth map $\dia{A}: \dia{V} \rightarrow \dia{W}$, such that $\ul{\dia{A}} = A_{0}$. This assignment defines a functor $\diamond: \gVect \rightarrow \gMan^{\infty}$. 
\end{tvrz}
\begin{proof}
For a precise definition of the sheaf $\C^{\infty}_{\dia{V}}$, we refer the reader to Example 3.23 in \cite{Vysoky:2022gm}. Let us only recall the bare minimum required in this paper. To any total basis $( t_{\lambda})_{\lambda=1}^{n}$ for $V$, there are associated global coordinates $( \bbz^{\lambda} )_{\lambda=1}^{n}$ on $\dia{V}$ which transform as vectors of the respective dual basis. More precisely, let $(t'_{\lambda})_{\lambda = 1}^{n}$ be another total basis. We \textit{always} assume it ordered so that $|t'_{\lambda}| = |t_{\lambda}|$. Let us write
\begin{equation} \label{eq_twototalbasis}
t'_{\lambda} = \fB_{\lambda}{}^{\kappa} t_{\kappa},
\end{equation}
where $\fB_{\lambda}{}^{\kappa} \in \R$ satisfy $|\fB_{\lambda}{}^{\kappa}| = |t_{\lambda}| - |t_{\kappa}|$. For this to make sense, we view $\R$ as a (trivially) graded algebra. In particular, this makes sure that $\fB_{\lambda}{}^{\kappa} = 0$ whenever $|t_{\lambda}| \neq |t_{\kappa}|$. The induced coordinates $( \bbz^{\lambda})_{\lambda=1}^{n}$ and $(\bbz'^{\lambda})_{\lambda=1}^{n}$ then transform as 
\begin{equation} \label{eq_diaVectcoordtansform}
\bbz^{\lambda} = (-1)^{|t_{\lambda}|(|t_{\kappa}| - |t_{\lambda}|)} \fB_{\kappa}{}^{\lambda} \bbz'^{\kappa}.
\end{equation}
Strictly speaking, the sign in this formula is unnecessary. It is however useful to keep it there as a guiding principle - as if $\fB_{\kappa}{}^{\lambda}$ could have a non-zero degree. Finally, let $A \in \Lin(V,W)$. Let $(t_{\lambda})_{\lambda=1}^{n}$ and $(s_{\sigma})_{\sigma=1}^{m}$ be a total basis for $V$ and $W$, respectively. Then one can write
\begin{equation} \label{eq_fAmatrix}
A(t_{\lambda}) = \fA_{\lambda}{}^{\sigma} s_{\sigma},
\end{equation}
where $\fA_{\lambda}{}^{\sigma} \in \R$ satisfy $|\fA_{\lambda}{}^{\sigma}| = |t_{\lambda}| - |s_{\sigma}|$. The transpose $A^{T}: W^{\ast} \rightarrow V^{\ast}$ of $A$ is defined similarly to (\ref{eq_transpose}) and its expression using dual bases reads
\begin{equation}
A^{T}(s^{\sigma}) = (-1)^{|s_{\sigma}|( |t_{\lambda}| - |s_{\sigma}|)} \fA_{\lambda}{}^{\sigma} t^{\lambda}. 
\end{equation}
The induced graded smooth map $\dia{A}: \dia{V} \rightarrow \dia{W}$ is then uniquely determined by pullbacks of corresponding coordinate functions:
\begin{equation} \label{eq_diaApullback}
\dia{A}^{\ast}( \bbu^{\sigma}) = (-1)^{|s_{\sigma}|( |t_{\lambda}| - |s_{\sigma}|)} \fA_{\lambda}{}^{\sigma} \bbz^{\lambda}
\end{equation}
Since those transform as vectors of dual bases, this is a priori independent of any choices. 
\end{proof}

We will also need a way how to handle tangent spaces and induced tangent maps. 

\begin{tvrz} \label{tvrz_tangentVSidentification}
Let $V \in \gVect$. For any $v \in V_{0}$, there is a canonical isomorphism
\begin{equation} \label{eq_tangentVSidentification}
V \cong T_{v}\dia{V}. 
\end{equation}
Moreover, let $A \in \Lin(V,W)$. Then the induced graded linear map $T_{v}\dia{A}: T_{v}\dia{V} \rightarrow T_{A_{0}(v)} \dia{W}$ fits into the commutative diagram
\begin{equation}
\begin{tikzcd}
T_{v}\dia{V} \arrow{r}{T_{v}\dia{A}} &[2em] T_{A_{0}(v)} \dia{W} \\
V \arrow{u} \arrow{r}{A} & W \arrow{u}
\end{tikzcd},
\end{equation}
where the vertical arrows are the isomorphisms (\ref{eq_tangentVSidentification}). In other words, this statement allows us to canonically identify $T_{v}\dia{A}$ with $A$. 
\end{tvrz}
\begin{proof}
Let $(t_{\lambda})_{\lambda = 1}^{n}$ be a total basis for $V$. Let $(\bbz^{\lambda})_{\lambda=1}^{n}$ be the induced global coordinates on $\dia{V}$. Then the isomorphism (\ref{eq_tangentVSidentification}) is determined by 
\begin{equation}
t_{\lambda} \mapsto (-1)^{|t_{\lambda}|} \frac{\partial}{\partial \bbz^{\lambda}}|_{v},
\end{equation}
for each $\lambda \in \{1,\dots,n\}$. Recall that the coordinate vector fields are uniquely determined by their action on coordinate functions, see Remark 4.13 in \cite{Vysoky:2022gm}:
\begin{equation}
\frac{\partial}{\partial \bbz^{\lambda}}( \bbz^{\kappa}) = \delta^{\kappa}_{\lambda}. 
\end{equation}

The extra sign may seem a bit odd, yet it makes all definitions compatible with coordinate transformations discussed in the proof of Proposition \ref{tvrz_diamant}. The rest is just a straightforward verification. 
\end{proof}

\begin{tvrz} \label{tvrz_subspacesubmanifold}
Let $V \in \gVect$ and let $P \subseteq V$ be a vector subspace. Let $i: P \rightarrow V$ be the respective inclusion. Then $(\dia{P},\dia{i})$ is a closed embedded submanifold of $\dia{V}$. 
\end{tvrz}
\begin{proof}
For each $p \in P_{0}$, the tangent map $T_{p}\dia{i}$ can be identified with $i$ in the sense of Proposition \ref{tvrz_tangentVSidentification}, so it is injective. Hence $\dia{i}$ is an immersion. The underlying map $\ul{\dia{i}} \equiv i_{0}: P_{0} \rightarrow V_{0}$ is obviously a closed embedding, hence $(\dia{P},\dia{i})$ is a closed embedded submanifold of $\dia{V}$. 
\end{proof}

One often utilizes the following statement.
\begin{tvrz} \label{tvrz_onpullback}
Let $V \in \gVect$. Let $\cS \in \gMan^{\infty}$ be an arbitrary graded manifold. Let $(t_{\lambda})_{\lambda=1}^{n}$ be a total basis for $V$ inducing the coordinates $(\bbz^{\lambda})_{\lambda=1}^{n}$ on $\dia{V}$. 

Then every graded smooth map $\varphi: \cS \rightarrow \dia{V}$ determines and is uniquely determined by a collection of functions $\{ f^{\lambda} \}_{\lambda=1}^{n} \subseteq \C^{\infty}_{\cS}(S)$, such that $|f^{\lambda}| = |t^{\lambda}|$, and the relation to $\varphi$ is 
\begin{equation}
f^{\lambda} = \varphi^{\ast}(\bbz^{\lambda}). 
\end{equation}
\end{tvrz}
\begin{proof}
This is a variant of Theorem 3.29 in \cite{Vysoky:2022gm}.
\end{proof}

Now, let $V,W \in \gVect$. Recall that there is a canonical graded bilinear map $\alpha: V \times W \rightarrow V \otimes_{\R} W$, given by $\alpha(v,w) = v \otimes w$. We will now argue that it can be promoted to a certain graded smooth map. This will help us to ``geometrize'' general bilinear maps in what follows. 

\begin{tvrz}  \label{tvrz_diaalpha}
Let $V,W \in \gVect$, and $\alpha: V \times W \rightarrow V \otimes_{\R} W$ as above.

Then there is a canonical graded smooth map $\dia{\alpha}: \dia{V} \times \dia{W} \rightarrow \dia{(V \otimes_{\R} W)}$. Its underlying map $\ul{\dia{\alpha}}: V_{0} \times W_{0} \rightarrow (V \otimes_{\R} W)_{0}$ is given by $\ul{\dia{\alpha}}(v,w) = v \otimes w$ for all $v,w \in V_{0}$. 
\end{tvrz}
\begin{proof}
Let $(t_{\lambda})_{\lambda=1}^{n}$ and $(s_{\sigma})_{\sigma=1}^{m}$ be a total basis for $V$ and $W$, respectively. Let $(\bbz^{\lambda})_{\lambda=1}^{n}$ and $(\bbu^{\sigma})_{\sigma=1}^{m}$ be the induced coordinates on $\dia{V}$ and $\dia{W}$, respectively. The product graded manifold $\dia{V} \times \dia{W}$ is then equipped with coordinates $(\bbz^{1}, \dots, \bbz^{n} , \bbu^{1}, \dots, \bbu^{m})$. 

On the other hand, the two bases above can be used to provide a total basis for $V \otimes_{\R} W$ consisting of vectors $r_{\lambda,\sigma} := t_{\lambda} \otimes s_{\sigma}$, $\lambda \in \{1,\dots,n\}$ and $\sigma \in \{1, \dots, m \}$. Let $( \bbv^{\lambda,\sigma} )$ be the corresponding coordinates on $\dia{(V \otimes_{\R} W)}$. One declares
\begin{equation} \label{eq_alphapullback}
\dia{\alpha}^{\ast}( \bbv^{\lambda,\sigma}) := (-1)^{|t_{\lambda}||s_{\sigma}|} \bbz^{\lambda} \bbu^{\sigma} = \bbu^{\sigma} \bbz^{\lambda}. 
\end{equation}
It is now straightforward check that this is independent of the chosen bases and that the underlying map has the required form. 
\end{proof}
\begin{rem}
Note that the underlying manifold of $\dia{(V \otimes_{\R} W)}$ is not $V_{0} \otimes_{\R} W_{0}$, but a direct sum
\begin{equation} \label{eq_VotimesWunderling}
(V \otimes_{\R} W)_{0} = \bigoplus_{k + \ell = 0} V_{k} \otimes_{\R} V_{\ell}. 
\end{equation} 
The underlying smooth map of $\dia{\alpha}$ is thus a canonical bilinear map $V_{0} \times W_{0} \rightarrow V_{0} \otimes_{\R} W_{0}$ followed by the inclusion into (\ref{eq_VotimesWunderling}). 
\end{rem}

\begin{cor} \label{cor_bilineargrad}
Let $V,W,X \in \gVect$. Let $\beta: V \times W \rightarrow X$ be a degree zero bilinear map. 

Then there is a canonical graded smooth map $\dia{\beta}: \dia{V} \times \dia{W} \rightarrow \dia{X}$. Its underlying map is given by $\ul{\dia{\beta}}(v,w) = \beta(v,w)$ for all $v,w \in V_{0}$. 
\end{cor}
\begin{proof}
By the universal property of the tensor product, there is a unique graded linear map $\beta': V \otimes_{\R} W \rightarrow X$ of degree zero, such that $\beta = \beta' \circ \alpha$. Let us define $\dia{\beta}$ as a composition
\begin{equation} \label{eq_diabeta}
\dia{\beta} := \dia{\beta}' \circ \dia{\alpha},
\end{equation}
where $\dia{\beta}': \dia{(V \otimes_{\R} W)} \rightarrow \dia{X}$ is just the arrow map of $\diamond$ applied on $\beta'$.
\end{proof}

\begin{example}
Let us examine how the map $\dia{\beta}$ looks explicitly. Fix a total basis $(t_{\lambda})_{\lambda=1}^{n}$ for $V$, a total basis $(s_{\sigma})_{\sigma=1}^{m}$ for $W$, and a total basis $(x_{\rho})_{\rho=1}^{q}$ for $X$. We thus have the corresponding coordinates $(\bbz^{\lambda})_{\lambda=1}^{n}$ for $\dia{V}$, $(\bbu^{\sigma})_{\sigma=1}^{m}$ for $\dia{W}$, and $(\bbx^{\rho})_{\rho=1}^{q}$ for $\dia{X}$. We will use the same symbols to denote the induced coordinates on $\dia{V} \times \dia{W}$. Write
\begin{equation}
\beta(t_{\lambda},s_{\sigma}) = \fbeta_{\lambda \sigma}{}^{\rho} x_{\rho},
\end{equation}
where $\fbeta_{\lambda \sigma}{}^{\rho} \in \R$ satisfy $|\fbeta_{\lambda \sigma}{}^{\rho}| = |t_{\lambda}| + |s_{\sigma}| - |x_{\rho}|$. The corresponding graded linear map $\beta'$ is thus 
\begin{equation}
\beta'(r_{\lambda,\sigma}) = \fbeta_{\lambda \sigma}{}^{\rho} x_{\rho},
\end{equation}
where $r_{\lambda,\sigma} := t_{\lambda} \otimes s_{\sigma}$ is the induced total basis of $V \otimes_{\R} W$. Let $(\bbv^{\lambda,\sigma})$ be the induced coordinates on $\dia{(V \otimes_{\R} W)}$. It follows from (\ref{eq_diaApullback}) that 
\begin{equation}
\dia{\beta}'^{\ast}(\bbx^{\rho}) = (-1)^{|x_{\rho}|(|t_{\lambda}| + |s_{\sigma}| - |x_{\rho}|)} \fbeta_{\lambda \sigma}{}^{\rho} \bbv^{\lambda,\sigma}. 
\end{equation}
Finally, using the definition (\ref{eq_diabeta}) and (\ref{eq_alphapullback}), we find that $\dia{\beta}^{\ast}$ has the explicit form
\begin{equation} \label{eq_diabetaastexplicit}
\dia{\beta}^{\ast}(\bbx^{\rho}) = (-1)^{|x_{\rho}|(|t_{\lambda}| + |s_{\sigma}| - |x_{\rho}|)} \fbeta_{\lambda \sigma}{}^{\rho} \bbu^{\sigma} \bbz^{\lambda}. 
\end{equation}
\end{example}

Now, let $\cS$ be a general graded manifold. We can consider a free $\C^{\infty}_{\cS}(S)$-module
\begin{equation} \label{eq_frMcS}
\frM(\cS) := \C^{\infty}_{\cS}(S) \otimes_{\R} V. 
\end{equation}
See \S 4.1 of \cite{vsmolka2025threefold} for details. It turns out that the degree zero component of this module has an important relation to $\dia{V}$.
\begin{tvrz}
The assignment $\cS \mapsto \frM(\cS)_{0}$, where $\frM(\cS)$ is defined by (\ref{eq_frMcS}), defines a functor $\frM_{0}: (\gMan^{\infty})^{\op} \rightarrow \Set$. 
This functor is naturally isomorphic to the functor of points $\frQ$ associated with the graded manifold $\dia{V}$. 
\end{tvrz}

\begin{proof}
If $\varphi: \cN \rightarrow \cS$ is a graded smooth map, the associated arrow map is 
\begin{equation}
\frM_{0}(\varphi) := \varphi^{\ast} \otimes \1_{V}: \frM(\cS)_{0} \rightarrow \frM(\cN)_{0}.
\end{equation}
It is an easy exercise to check that this makes $\frM_{0}$ into a functor. Next, recall that the functor of points $\frQ$ associated with $\dia{V}$ is just the hom-functor $\frQ := \gMan^{\infty}(-,\dia{V})$. We thus have to construct a bijection 
\begin{equation}
\Xi_{\cS}: \frQ(\cS) \rightarrow \frM_{0}(\cS)
\end{equation}
natural in $\cS$. Let $(t_{\lambda})_{\lambda = 1}^{n}$ be a fixed total basis for $V$ inducing the coordinates $(\bbz^{\lambda})_{\lambda=1}^{n}$ for $\dia{V}$. Let $\Phi_{\lambda} := (-1)^{|t_{\lambda}|} 1 \otimes t_{\lambda} \in \frM(\cS)$. It follows that $( \Phi_{\lambda})_{\lambda=1}^{n}$ is a frame for the $\C^{\infty}_{\cS}(S)$-module $\frM(\cS)$. For any $\phi \in \frQ(\cS)$, define 
\begin{equation} \label{eq_XiS}
\Xi_{\cS}(\phi) := \phi^{\ast}(\bbz^{\lambda}) \cdot \Phi_{\lambda}. 
\end{equation}
Note that the right-hand side is of degree zero, hence an element of $\frM(\cS)_{0}$. It follows from Proposition \ref{tvrz_onpullback} that this is a bijection. The naturality in $\cS$ is easy to check. Note that (\ref{eq_XiS}) is independent of the used basis $(t_{\lambda})_{\lambda=1}^{n}$. This follows from (\ref{eq_diaVectcoordtansform}) and we leave it as an exercise. 
\end{proof}

\section{General linear group} \label{sec_generallinear}
In this section, let us revisit the basic example of a graded Lie group. It is based on \S 2.4.1 of \cite{Smolka2023}, offering perhaps some new perspective and justifying certain choices made there. Let us start by constructing the graded manifold $\GL(V)$. 

\begin{tvrz}
Let $V \in \gVect$. Then the set 
\begin{equation}
\GL(V_{\bullet}) = \{ A \in \gl(V)_{0} \mid A \text{ is invertible} \}
\end{equation}is open in $\gl(V)_{0}$. We can thus consider a graded manifold
\begin{equation}
\GL(V) := \dia{\gl(V)}|_{\GL(V_{\bullet})}. 
\end{equation}
Note that $\gdim(\GL(V)) = \neg \gdim( \gl(V))$. 
\end{tvrz}
\begin{proof}
Since $V$ is finite-dimensional, one has 
\begin{equation}
\gl(V)_{0} \cong \bigoplus_{k \in I} \gl(V_{k}),
\end{equation}
where $ I = \{ k \in \Z \mid V_{k} \neq \{0\} \}$ is finite. The topology on $\gl(V)_{0}$ then corresponds to the product topology on $\prod_{k \in I} \gl(V_{k})$ and $\GL(V_{\bullet})$ corresponds to the open subset $\prod_{k \in I} \GL(V_{k})$. 
\end{proof}
We are now ready to produce the multiplication on $\GL(V)$. Let $\beta: \gl(V) \times \gl(V) \rightarrow \gl(V)$ be the degree zero bilinear map defined as 
\begin{equation} \label{eq_betaglVmulti}
\beta(A,B) := AB,
\end{equation}
for all $A,B \in \gl(V)$. By Corollary \ref{cor_bilineargrad}, there is an induced graded smooth map 
\begin{equation}
\dia{\beta}: \dia{\gl(V)} \times \dia{\gl(V)} \rightarrow \dia{\gl(V)}.
\end{equation}
The underlying smooth map has the property $\ul{\dia{\beta}}( \GL(V_{\bullet}) \times \GL(V_{\bullet})) \subseteq \GL(V_{\bullet})$. This implies that $\dia{\beta}$ restricts to a graded smooth map 
\begin{equation}
\mu: \GL(V) \times \GL(V) \rightarrow \GL(V),
\end{equation}
such that its underlying smooth map $\ul{\mu}: \GL(V_{\bullet}) \times \GL(V_{\bullet}) \rightarrow \GL(V_{\bullet})$ is given by $\ul{\mu}(A,B) = AB$. This is the multiplication on $\GL(V)$. Equivalently, if $k: \GL(V) \rightarrow \dia{\gl(V)}$ denotes the embedding, $\mu$ is the unique graded smooth map fitting into the commutative diagram
\begin{equation} \label{eq_muviaksandbeta}
\begin{tikzcd}
\GL(V) \times \GL(V) \arrow[dashed]{r}{\mu} \arrow{d}{k \times k} & \GL(V) \arrow{d}{k} \\
\dia{\gl(V)} \times \dia{\gl(V)} \arrow{r}{\dia{\beta}} & \dia{\gl(V)}
\end{tikzcd}.
\end{equation}

Before proceeding further, let us introduce coordinates on $\dia{\gl(V)}$, and thus also on $\GL(V)$. Let $(t_{\lambda})_{\lambda=1}^{n}$ be a fixed total basis for $V$. For each $\lambda,\kappa \in \{1,\dots,n\}$, define $\Delta_{\lambda}{}^{\kappa} \in \gl(V)$ by 
\begin{equation} 
\Delta_{\lambda}{}^{\kappa}( t_{\nu}) := \delta^{\kappa}_{\nu} \delta^{\mu}_{\lambda} t_{\mu}. 
\end{equation}
It maps $t_{\kappa}$ to $t_{\lambda}$ and all remaining basis vectors to zero, hence $|\Delta_{\lambda}{}^{\kappa}| = |t_{\lambda}| - |t_{\kappa}|$. We call this the \textbf{standard basis} for $\gl(V)$ corresponding to $(t_{\lambda})_{\lambda=1}^{n}$. It is defined so that if $A \in \gl(V)$ is given by $A(t_{\lambda}) = \fA_{\lambda}{}^{\kappa} t_{\kappa}$, then $A = \fA_{\lambda}{}^{\kappa} \Delta_{\kappa}{}^{\lambda}$.  Note that one assumes that $|\fA_{\lambda}{}^{\kappa}| = |t_{\lambda}| - |t_{\kappa}| + |A|$. Let $\cD^{\lambda}{}_{\kappa} \in \gl(V)^{\ast}$ denote the corresponding dual basis, that is we define
\begin{equation} \label{eq_dualstandard}
\cD^{\lambda}{}_{\kappa}( \Delta_{\nu}{}^{\mu}) := \delta^{\lambda}_{\nu} \delta^{\mu}_{\kappa}. 
\end{equation}
Finally, let us write $(\bby^{\lambda}{}_{\kappa})$ for the induced coordinates on $\dia{\gl(V)}$. Observe that one has 
\begin{equation}
|\bby^{\lambda}{}_{\kappa}| = |\cD^{\lambda}{}_{\kappa}| = - |\Delta_{\lambda}{}^{\kappa}| = |t_{\kappa}| - |t_{\lambda}|,
\end{equation}
for all $\lambda,\kappa \in \{1,\dots,n\}$. Since $\GL(V)$ is just a restriction of $\dia{\gl(V)}$ to the open subset $\GL(V_{\bullet})$, we will use the same coordinates there. 
\begin{rem}
Suppose that $(t'_{\lambda})_{\lambda=1}^{n}$ is another total basis of $V$, related to the original one as in (\ref{eq_twototalbasis}). The associated standard basis $\Delta'{}_{\lambda}{}^{\kappa}$ is related to the original one as 
\begin{equation}
\Delta'{}_{\lambda}{}^{\kappa} = (-1)^{(|t_{\lambda}| - |t_{\kappa}|)(|t_{\mu}| - |t_{\kappa}|)} (\fB^{-1})_{\mu}{}^{\kappa} \fB_{\lambda}{}^{\nu} \Delta_{\nu}{}^{\mu},
\end{equation}
where $t_{\mu} = (\fB^{-1})_{\mu}{}^{\kappa} t'_{\kappa}$. The induced coordinates on $\dia{\gl(V)}$ are then related as
\begin{equation} \label{eq_coordinatesonGLVrelation}
(-1)^{|t_{\mu}| - |t_{\nu}|} \bby^{\nu}{}_{\mu} = (-1)^{|t_{\kappa}| - |t_{\lambda}|} (\fB^{-1})_{\mu}{}^{\kappa} \bby'^{\lambda}{}_{\kappa} \fB_{\lambda}{}^{\nu}. 
\end{equation}
This is a straightforward exercise. 
\end{rem}
\begin{tvrz}
Let $(\bby^{\lambda}{}_{\kappa})$ be the coordinates on $\GL(V)$ introduced above. Let $(\bbz^{\lambda}{}_{\kappa}, \bbu^{\lambda}{}_{\kappa})$ denote the induced coordinates on the product $\GL(V) \times \GL(V)$. Then the pullback $\mu^{\ast}$ has the explicit form
\begin{equation} \label{eq_muastexplicit}
\mu^{\ast}( \bby^{\lambda}{}_{\kappa}) = \bbu^{\nu}{}_{\kappa} \bbz^{\lambda}{}_{\nu}. 
\end{equation}
\end{tvrz} 
\begin{proof}
It suffices to calculate the expression for $\dia{\beta}^{\ast}$, where $\beta$ is defined by (\ref{eq_betaglVmulti}). One has 
\begin{equation}
\beta( \Delta_{\rho}{}^{\sigma}, \Delta_{\nu}{}^{\mu}) = \Delta_{\rho}{}^{\sigma} \Delta_{\nu}{}^{\mu} = \delta^{\sigma}_{\nu} \Delta_{\rho}{}^{\mu} = \delta^{\lambda}_{\rho} \delta^{\sigma}_{\nu}  \delta^{\mu}_{\kappa} \Delta_{\lambda}{}^{\kappa}.
\end{equation}
The rest is just plugging into the formula (\ref{eq_diabetaastexplicit}). 
\end{proof}

It is trivial to guess the unit $e: \{ \ast \} \rightarrow \GL(V)$. Indeed, it corresponds to a single point in $\GL(V_{\bullet})$ which is to serve as a unit for the ordinary Lie group $\GL(V_{\bullet})$ with a multiplication $\ul{\mu}(A,B) = AB$. The natural choice is therefore $e := \1_{V} \in \GL(V_{\bullet})$. 

To define the inverse map $\iota: \GL(V) \rightarrow \GL(V)$ without too much work, we will now make a little detour, useful on its own. 
Let $\Aut(\frM(\cS))$ be the set of \textit{degree zero} $\C^{\infty}_{\cS}(S)$-module automorphisms of $\frM(\cS)$ defined by (\ref{eq_frMcS}). It turns out that it has a crucial relation to $\GL(V)$. 
\begin{tvrz} \label{tvrz_GLVFOP}
The assignment $\cS \mapsto \Aut(\frM(\cS))$ defines a functor
\begin{equation}
\frF: (\gMan^{\infty})^{\op} \rightarrow \Set
\end{equation}
 naturally isomorphic to the functor of points $\frP$ associated with $\GL(V)$. 
\end{tvrz}
\begin{proof}
Let $\varphi: \cN \rightarrow \cS$ be a graded smooth map. We must construct a set map 
\begin{equation}
\frF(\varphi): \frF(\cS) \rightarrow \frF(\cN). 
\end{equation}
Choose a total basis $(t_{\lambda})_{\lambda=1}^{n}$ for $V$. Let us define $\Phi_{\lambda} := (-1)^{|t_{\lambda}|} 1 \otimes t_{\lambda}$. It follows that $(\Phi_{\lambda})_{\lambda=1}^{n}$ forms a frame for $\frM(\cS)$. We will use the same symbol regardless of a particular $\cS$. For any $F \in \frF(\cS)$ we can use the frame to write
\begin{equation} \label{eq_Fframedecomposition}
F(\Phi_{\lambda}) = \fF^{\kappa}{}_{\lambda} \cdot \Phi_{\kappa},
\end{equation}
for unique $\fF^{\kappa}{}_{\lambda} \in \C^{\infty}_{\cS}(S)$ of degree $|\fF^{\kappa}{}_{\lambda}| = |t_{\lambda}| - |t_{\kappa}|$. Let $\ul{\fF}$ be the $n \times n$ matrix defined as 
\begin{equation}
\ul{\fF}^{\kappa}{}_{\lambda} = \ul{\fF^{\kappa}{}_{\lambda}}.
\end{equation}
$\ul{f} \in C^{\infty}_{S}(S)$ denotes the body of the function $f \in \C^{\infty}_{\cS}(S)$. It follows that $F$ is invertible, iff $\det{\ul{\fF}}$ is everywhere non-zero. See Appendix A of \cite{vsmolka2025threefold} for the proof of this statement. We now propose $G := [\frF(\varphi)](F)$ to be given by the formula
\begin{equation} \label{eq_rfFfunctormap}
G(\Phi_{\lambda}) := \fG^{\kappa}{}_{\lambda} \cdot \Phi_{\kappa}, \text{ where } \fG^{\kappa}{}_{\lambda} := \varphi^{\ast}( \fF^{\kappa}{}_{\lambda}). 
\end{equation}
It is easy to see that $\det{\ul{\fG}} = \det{\ul{\fF}} \circ \ul{\phi}$ is everywhere non-zero, proving that $G \in \frF(\cN)$. It is straightforward to verify that $\frF$ defines a functor. Recall that the functor of points $\frP$ associated with $\GL(V)$ is just the contravariant hom-functor corresponding to $\GL(V)$, that is 
\begin{equation}
\frP := \gMan^{\infty}(-,\GL(V)),
\end{equation}
To show that $\frF \cong \frP$, we thus need to construct a bijection 
\begin{equation}
\Psi_{\cS}: \gMan^{\infty}(\cS,\GL(V)) \rightarrow \Aut(\frM(\cS)),
\end{equation}
and verify that it is natural in $\cS$. Let $\phi: \cS \rightarrow \GL(V)$ be a graded smooth map. Fix a total basis $(t_{\lambda})_{\lambda=1}^{n}$ for $V$. We thus have the induced coordinates $(\bby^{\kappa}{}_{\lambda})$ for $\GL(V)$ and the induced frame $(\Phi_{\lambda})_{\lambda=1}^{n}$ for $\frM(\cS)$. We declare
\begin{equation} \label{eq_PsicSmap}
[\Psi_{\cS}(\phi)](\Phi_{\lambda}) := \phi^{\ast}( \bby^{\kappa}{}_{\lambda}) \cdot \Phi_{\kappa}. 
\end{equation}
Write $\fF^{\kappa}{}_{\lambda} := \phi^{\ast}( \bby^{\kappa}{}_{\lambda})$. For each $s \in S$, the matrix $\ul{\fF}^{\kappa}{}_{\lambda}(s)$ is the transpose of the matrix of $\ul{\phi}(s) \in \GL(V_{\bullet})$. Consequently, $\det(\ul{\fF})$ is everywhere non-zero and thus $\Psi_{\cS}(\phi) \in \Aut(\frM(\cS))$. We invite the reader to verify that the definition does not depend on the choice of the total basis $(t_{\lambda})_{\lambda=1}^{n}$. One has to use (\ref{eq_coordinatesonGLVrelation}). Note the importance of the sign in the definition of $\Phi_{\lambda}$. The naturality in $\cS$ is easy to check. 

Conversely, if $F \in \Aut(\frM(\cS))$, one writes it as $F(\Phi_{\lambda}) = \fF^{\kappa}{}_{\lambda} \cdot \Phi_{\kappa}$, where $\fF^{\kappa}{}_{\lambda} \in \C^{\infty}_{\cS}(S)$ satisfy $|\fF^{\kappa}{}_{\lambda}| = |t_{\lambda}| - |t_{\kappa}|$. Let us first define a graded smooth map $\phi: \cS \rightarrow \dia{\gl(V)}$. Set
\begin{equation} \label{eq_phioncoordfunctions}
\phi^{\ast}(\bby^{\kappa}{}_{\lambda}) := \fF^{\kappa}{}_{\lambda},
\end{equation}
for each $\lambda,\kappa \in \{1, \dots, n\}$. Thanks to Proposition \ref{tvrz_onpullback}, this is enough to define $\phi$. It is not difficult to see that for each $s \in S$, the degree zero linear map $\ul{\phi}(s) \in \gl(V)_{0}$ is given by 
\begin{equation}
[\ul{\phi}(s)](t_{\lambda}) = \ul{\fF}^{\kappa}{}_{\lambda} t_{\kappa}. 
\end{equation}
This shows that $\ul{\phi}(s) \in \GL(V_{\bullet})$. One can thus view $\phi$ as a graded smooth map $\phi: \cS \rightarrow \GL(V)$ and set $\Psi_{\cS}^{-1}(F) := \phi$. It follows from the construction that this defines the inverse to $\Psi_{\cS}$. 
\end{proof}

To construct the inverse map $\iota: \GL(V) \rightarrow \GL(V)$, one can now utilize this new viewpoint. For each $\cS \in \gMan^{\infty}$, the set $\frF(\cS) = \Aut(\frM(\cS))$ has an obvious group structure. Let $\mathsf{m}'_{\cS}: \frF(\cS) \times \frF(\cS) \rightarrow \frF(\cS)$ denote the respective multiplication map. On the other hand, the graded smooth map $\mu: \GL(V) \times \GL(V) \rightarrow \GL(V)$ induces a set map $\sfm_{\cS}: \frP(\cS) \times \frP(\cS) \rightarrow \frP(\cS)$, where $\frP$ is the functor of points associated with $\GL(V)$. It turns out that those two maps are related by the above natural isomorphism. 

\begin{tvrz}
Using the notation of Proposition \ref{tvrz_GLVFOP} and the preceding paragraph, the map $\Psi_{\cS}: \frP(\cS) \rightarrow \frF(\cS)$ is equivariant with respect to $\sfm_{\cS}$ and $\sfm'_{\cS}$, that is it fits into the diagram 
\begin{equation} \label{eq_mSm'S}
\begin{tikzcd}
\frP(\cS) \times \frP(\cS) \arrow{d}{\sfm_{\cS}} \arrow{r}{\Psi_{\cS} \times \Psi_{\cS}} &[2em] \frF(\cS) \times \frF(\cS) \arrow{d}{\sfm'_{\cS}} \\
\frP(\cS) \arrow{r}{\Psi_{\cS}} & \frF(\cS)
\end{tikzcd}
\end{equation}
\end{tvrz}
\begin{proof}
For any $F,G \in \frF(\cS)$, we have $\sfm_{\cS'}(F,G) = F G$. For $\phi,\phi' \in \frP(\cS)$, one defines $\sfm_{\cS}(\phi,\phi') := \mu \circ (\phi,\phi')$. It thus suffices to argue that for any $F,G \in \frF(\cS)$, one has 
\begin{equation}
\Psi_{\cS}^{-1}(F G) = \mu \circ (\Psi_{\cS}^{-1}(F), \Psi_{\cS}^{-1}(G)). 
\end{equation}
Both sides are graded smooth maps from $\cS$ to $\GL(V)$. It thus suffices to compare the respective pullbacks of coordinate functions. Let $(t_{\lambda})_{\lambda=1}^{n}$ be a fixed total basis for $V$. Let $(\Phi_{\lambda})_{\lambda=1}^{n}$ be the corresponding frame for $\frM(\cS)$. One has 
\begin{equation}
(F G)(\Phi_{\lambda}) = F( \fG^{\nu}{}_{\lambda} \Phi_{\nu}) = \fG^{\nu}{}_{\lambda} \fF^{\kappa}{}_{\nu} \Phi_{\kappa}, 
\end{equation}
where $\fF$ and $\fG$ denote the matrices of $F$ and $G$, respectively. 
Consequently, from (\ref{eq_phioncoordfunctions}) one has 
\begin{equation}
[ \Psi_{\cS}^{-1}(FG)]^{\ast}( \bby^{\kappa}{}_{\lambda}) = \fG^{\nu}{}_{\lambda} \fF^{\kappa}{}_{\nu}. 
\end{equation}
On the other hand, one utilizes (\ref{eq_muastexplicit}) to write 
\begin{equation}
[ \mu \circ (\Psi_{\cS}^{-1}(F), \Psi_{\cS}^{-1}(G))]^{\ast}( \bby^{\kappa}{}_{\lambda}) =  (\Psi_{\cS}^{-1}(F), \Psi^{-1}_{\cS}(G))^{\ast}( \bbu^{\nu}{}_{\lambda} \bbz^{\kappa}{}_{\nu}) = \fG^{\nu}{}_{\lambda} \fF^{\kappa}{}_{\nu}. 
\end{equation}
We see that the both expressions are equal for all $\lambda,\kappa \in \{1,\dots,n\}$ and the proof is finished. 
\end{proof}
We can now proceed to the definition of the inverse. 

\begin{tvrz}
For each $\cS \in \gMan^{\infty}$, let $\sfi'_{\cS}: \frF(\cS) \rightarrow \frF(\cS)$ be the inverse with respect to $\sfm'_{\cS}$. Let us define $\sfi_{\cS}: \frP(\cS) \rightarrow \frP(\cS)$ to fit into the diagram
\begin{equation} \label{eq_iSi'S}
\begin{tikzcd}
\frP(\cS) \arrow[dashed]{d}{\sfi_{\cS}} \arrow{r}{\Psi_{\cS}} & \frF(\cS) \arrow{d}{\sfi'_{\cS}} \\
\frP(\cS) \arrow{r}{\Psi_{\cS}} & \frF(\cS)
\end{tikzcd}
\end{equation}
Then $\sfi := \{ \sfi_{\cS} \}_{\cS}$ forms a natural transformation $\sfi: \frP \rightarrow \frP$. Consequently, $\sfi$ is induced by a unique graded smooth map $\iota: \GL(V) \rightarrow \GL(V)$. 
\end{tvrz}
\begin{proof}
Since $\Psi_{\cS}$ are components of the natural isomorphism $\Psi: \frP \rightarrow \frF$, it suffices to verify that $\sfi' := \{ \sfi'_{\cS} \}_{\cS}$ defines a natural transformation $\sfi': \frF \rightarrow \frF$. But this follows from the uniqueness of inverses together with the fact that the multiplication $\sfm'_{\cS}: \frF(\cS) \times \frF(\cS) \rightarrow \frF(\cS)$ is natural in $\cS$ in the following sense: For any graded smooth map $\varphi: \cN \rightarrow \cS$, one has 
\begin{equation}
\sfm'_{\cN} \circ (\frF(\varphi) \times \frF(\varphi)) = \frF(\varphi) \circ \sfm'_{\cS}. 
\end{equation}
This follows immediately from the fact how $\frF(\varphi)$ was defined, see (\ref{eq_rfFfunctormap}). Equivalently, we can use (\ref{eq_mSm'S}) together with the similarly formulated naturality of $\sfm_{\cS}$. 

This proves that $\sfi: \frP \rightarrow \frP$ is a natural transformation. The fact that it is induced by a unique graded smooth map $\iota: \GL(V) \rightarrow \GL(V)$ corresponds to the fact that the Yoneda embedding is a full and faithful functor. Explicitly, it is obtained as 
\begin{equation} \label{eq_iota}
\iota = \sfi_{\GL(V)}( \1_{\GL(V)}). 
\end{equation}
It follows that for any $\cS \in \gMan^{\infty}$ and $\phi \in \frP(\cS)$, one has $\sfi_{\cS}(\phi) = \iota \circ \phi$. 
\end{proof}

We can now prove the main claim of this section.

\begin{theorem}
$(\GL(V), \mu, \iota, e)$ forms a graded Lie group called the \textbf{general linear group} of $V$. 
\end{theorem}
\begin{proof}
We shall utilize Proposition \ref{tvrz_gradedLGintermsofsets}. One thus has to argue that $(\frP(\cS),\sfm_{\cS},\sfi_{\cS},\sfe_{\cS})$ is a group for every $\cS \in \gMan^{\infty}$. Let $\sfe'_{\cS} := \Psi_{\cS} \circ \sfe_{\cS}: \{ \ast \} \rightarrow \frF(\cS)$. It is easy to see that $\sfe'_{\cS}(\ast) = \1_{\frM(\cS)}$. Thanks to the commutativity of (\ref{eq_mSm'S}) and (\ref{eq_iSi'S}), the claim is equivalent to showing that $(\frF(\cS), \sfm'_{\cS}, \sfi'_{\cS}, \sfe'_{\cS})$ forms a group. But that is obvious. This finishes the proof. 
\end{proof}

\begin{tvrz}
The graded Lie algebra associated with $\GL(V)$ can be canonically identified with $(\gl(V),[\cdot,\cdot])$, where $[\cdot,\cdot]$ is the graded commutator (\ref{eq_gcommutator}). 
\end{tvrz}
\begin{proof}
We know that there is a canonical isomorphism $\gl(V) \cong T_{e} (\dia{\gl(V)})$, see (\ref{eq_tangentVSidentification}). Since $\GL(V)$ is just a restriction of $\dia{\gl(V)}$ to the open subset $\GL(V_{\bullet})$, there is a canonical isomorphism
\begin{equation} \label{eq_glVTeGLViso}
\gl(V) \cong T_{e} \GL(V). 
\end{equation}
Let $A \in \gl(V)$. We can write $A = \fA_{\mu}{}^{\nu} \Delta_{\nu}{}^{\mu}$, where $\fA_{\mu}{}^{\nu} \in \R$ satisfy $|\fA_{\mu}{}^{\nu}| = |A| + |t_{\mu}| - |t_{\nu}|$. The corresponding tangent vector $x_{A} \in T_{e} \GL(V)$ under the isomorphism (\ref{eq_glVTeGLViso}) is given by
\begin{equation}
x_{A} = (-1)^{|t_{\nu}| - |t_{\mu}|} \fA_{\mu}{}^{\nu} \frac{\partial}{\partial \bby^{\nu}{}_{\mu}}|_{e}. 
\end{equation}
We want to calculate the corresponding left-invariant vector field $x_{A}^{L}$ on $\GL(V)$. Let $\ol{x}_{A}$ be the vector field extending $x_{A}$ in the simplest possible way, namely
\begin{equation} \label{eq_barxA}
\ol{x}_{A} = (-1)^{|t_{\nu}| - |t_{\mu}|} \fA_{\mu}{}^{\nu} \frac{\partial}{\partial \bby^{\nu}{}_{\mu}}.
\end{equation}
Clearly $\ol{x}_{A}|_{e} = x_{A}$. It follows from (\ref{eq_leftinvariantexplicit}) that $x_{A}^{L}$ is can be written as a composition
\begin{equation}
x_{A}^{L} = (\1_{\GL(V)}, e_{\GL(V)})^{\ast} \circ (1 \otimes \ol{x}_{A}) \circ \mu^{\ast}. 
\end{equation}
This can be evaluated easily using (\ref{eq_muastexplicit}) and (\ref{eq_barxA}) to find
\begin{equation} \label{eq_xAL}
x_{A}^{L} = (-1)^{|t_{\nu}| - |t_{\kappa}|} \fA_{\kappa}{}^{\nu} \bby^{\lambda}{}_{\nu} \frac{\partial}{\partial \bby^{\lambda}{}_{\kappa}}.
\end{equation}
Plugging this into the graded commutator, we obtain 
\begin{equation}
[x_{A}^{L},x_{B}^{L}] = (-1)^{|t_{\nu}| - |t_{\kappa}|} \big( (-1)^{|A|(|B| + |t_{\kappa}| - |t_{\mu}|)} \fB_{\kappa}{}^{\mu} \fA_{\mu}{}^{\nu} - (-1)^{|B|(|t_{\kappa}| - |t_{\mu}|)} \fA_{\kappa}{}^{\mu} \fB_{\mu}{}^{\nu} \big) \bby^{\lambda}{}_{\nu} \frac{\partial}{\partial \bby^{\lambda}{}_{\kappa}}.
\end{equation}
But the combination in the large parentheses is in fact precisely the matrix of the graded commutator $[A,B]$ with respect to the total basis $(t_{\lambda})_{\lambda=1}^{n}$, and it thus follows from (\ref{eq_xAL}) that 
\begin{equation}
[x^{L}_{A},x^{L}_{B}] = x_{[A,B]}^{L}. 
\end{equation}
This shows that the image of $[A,B]$ under the isomorphism (\ref{eq_glVTeGLViso}) is precisely $[x_{A},x_{B}]_{\g}$ defined by (\ref{eq_Liebracket}) for $\g := T_{e}\GL(V)$. This finishes the proof. 
\end{proof}
\section{Graded orthogonal group} \label{sec_GOVg}
Let $V \in \gVect$ and fix a degree $\ell$ metric $g$ on $V$. We intend to construct a graded Lie subgroup $(\gO(V,g),j)$ of $\GL(V)$. We expect its Lie algebra to be identified with the subspace $\ao(V,g) \subseteq \gl(V)$ consisting of graded linear maps skew-symmetric with respect to $g$. If $V$ is trivially graded, that is an ordinary finite-dimensional real vector space, and $\ell = 0$, we expect to obtain the ordinary orthogonal Lie group. The whole procedure will secretly mimic the standard construction and we invite the reader to find the analogies. 

\begin{enumerate}[(1)] 
\item First, recall from Definition \ref{def_symantisym} that we have a subspace $\Sym(V,g) \subseteq \gl(V)$. Let $i: \Sym(V,g) \rightarrow \gl(V)$ denote the inclusion. It follows from Proposition \ref{tvrz_subspacesubmanifold} that $( \dia{\Sym(V,g)}, \dia{i})$ forms a closed embedded submanifold of $\dia{\gl(V)}$. 

Let $k: \GL(V) \rightarrow \dia{\gl(V)}$ denote the embedding of the open submanifold $\GL(V)$. It is obviously transversal (see Definition 7.40 in \cite{Vysoky:2022gm}) to the map $\dia{i}$  and we can form the pullback 
\begin{equation} \label{eq_Symcross}
\begin{tikzcd}
\Sym^{\times}(V,g) \arrow[dashed]{r}{k'} \arrow[dashed]{d}{\dia{i}'}  & \dia{\Sym(V,g)} \arrow{d}{\dia{i}} \\
\GL(V) \arrow{r}{k} & \dia{\gl(V)} 
\end{tikzcd}
\end{equation}
such that $(\Sym^{\times}(V,g), \dia{i}')$ forms a closed embedded submanifold of $\GL(V)$. See Theorem 7.43 in \cite{Vysoky:2022gm} for details.  Note that its underlying subset is 
\begin{equation}
\Sym^{\times}_{0}(V,g) := \Sym(V,g)_{0} \cap \GL(V_{\bullet}),
\end{equation}
that is a set of degree zero invertible endomorphisms of $V$ symmetric with respect to $g$. 

\item Recall that we have a degree zero graded linear map $\tau: \gl(V) \rightarrow \gl(V)$ defined by (\ref{eq_taumap}). We can thus promote it to a graded smooth map $\dia{\tau}: \dia{\gl(V)} \rightarrow \dia{\gl(V)}$. Let us argue that it defines a graded smooth map $\tau^{\times}: \GL(V) \rightarrow \GL(V)$ fitting into the commutative diagram
\begin{equation}
\begin{tikzcd}
\GL(V) \arrow[dashed]{r}{\tau^{\times}} \arrow{d}{k} & \GL(V) \arrow{d}{k} \\
\dia{\gl(V)} \arrow{r}{\dia{\tau}} & \dia{\gl(V)} 
\end{tikzcd}
\end{equation}
Since $(\GL(V),k)$ is just an open submanifold of $\dia{\gl(V)}$, it suffices to check that $\ul{\dia{\tau}}(\GL(V_{\bullet})) \subseteq \GL(V_{\bullet})$. But this follows from the fact that for any $A \in \GL(V_{\bullet})$, one has $\ul{\dia{\tau}}(A) = \tau(A)$ and $\tau(A)^{-1} = \tau(A^{-1})$. The crucial property of $\tau^{\times}$ is the following one:
\begin{tvrz} \label{tvrz_tauantihom}
$\tau^{\times}: \GL(V) \rightarrow \GL(V)$ is an anti-homomorphism, that is it fits into the commutative diagram 
\begin{equation} \label{eq_tauantihom}
\begin{tikzcd}
\GL(V) \times \GL(V) \arrow{d}{\sigma} \arrow{r}{\mu} & \GL(V) \arrow{dd}{\tau^{\times}}\\
\GL(V) \times \GL(V) \arrow{d}{\tau^{\times} \times \tau^{\times}} & \\
\GL(V) \times \GL(V) \arrow{r}{\mu} & \GL(V)
\end{tikzcd},
\end{equation}
where $\sigma$ is the obvious ``flip map'' determined uniquely by 
\begin{equation}
p_{1} \circ \sigma = p_{2}, \; \; p_{2} \circ \sigma = p_{1},
\end{equation}
where $p_{1,2}: \GL(V) \times \GL(V) \rightarrow \GL(V)$ are the canonical projections. 
\end{tvrz}
We postpone the proof of this statement to Subsection \ref{subsec_technical1}. 

\item Let us consider a graded smooth map $\varphi^{\times} := \mu \circ (\tau^{\times}, \1_{\GL(V)}): \GL(V) \rightarrow \GL(V)$. 

\begin{tvrz} \label{tvrz_varphi}
There is a unique graded smooth map $\varphi: \GL(V) \rightarrow \Sym^{\times}(V,g)$ fitting into the commutative diagram
\begin{equation} \label{eq_varphidiagram}
\begin{tikzcd}
& \Sym^{\times}(V,g) \arrow{d}{\dia{i}'} \\
\GL(V) \arrow[dashed]{ur}{\varphi} \arrow{r}{\varphi^{\times}} & \GL(V) 
\end{tikzcd}
\end{equation} 
\end{tvrz}

The proof is statement is postponed to Subsection \ref{subset_technical2}. However, note that the corresponding underlying map $\ul{\varphi}: \GL(V_{\bullet}) \rightarrow \Sym^{\times}_{0}(V,g)$ has the form 
\begin{equation} \label{eq_varphiunderlying}
\ul{\varphi}(A) = (g^{-1} A^{T} g) A. 
\end{equation}

\item Let $e: \{ \ast \} \rightarrow \GL(V)$ be the unit. It corresponds to the choice of the point $e = \1_{V} \in \GL(V_{\bullet})$. Since $\1_{V} \in \Sym_{0}^{\times}(V,g)$, there exists a unique graded smooth map $e^{\times}: \{ \ast \} \rightarrow \Sym^{\times}(V,g)$ fitting into the commutative diagram
\begin{equation}
\begin{tikzcd}
& \Sym^{\times}(V,g) \arrow{d}{\dia{i}'} \\
\{ \ast \} \arrow{r}{e} \arrow[dashed]{ur}{e^{\times}} & \GL(V)
\end{tikzcd}.
\end{equation}
It corresponds to a choice of the point $e^{\times} = \1_{V} \in \Sym^{\times}_{0}(V,g)$.  This point is particularly important for the graded smooth map $\varphi$.

\begin{tvrz} \label{tvrz_regularvalue}
The map $\varphi: \GL(V) \rightarrow \Sym^{\times}(V,g)$ is transversal to $e^{\times}$. In other words, the corresponding point $\1_{V} \in \Sym^{\times}_{0}(V,g)$ is a regular value of $\varphi$. 
\end{tvrz}

We will prove this statement in Subsection \ref{subsec_technical3}. 

\item Since $\1_{V} \in \Sym^{\times}_{0}(V,g)$ is a regular value of $\varphi$, we can construct the corresponding regular level set submanifold. This will be our \textbf{graded orthogonal group} $\gO(V,g)$. In other words, it is the graded manifold defined by the pullback diagram
\begin{equation} \label{eq_O}
\begin{tikzcd}
\gO(V,g) \arrow[dashed]{d}{j} \arrow[dashed]{r} & \{ \ast \} \arrow{d}{e^{\times}} \\
\GL(V) \arrow{r}{\varphi} & \Sym^{\times}(V,g) 
\end{tikzcd}
\end{equation}
This makes $(\gO(V,g),j)$ into a closed embedded submanifold of $\GL(V)$. See Proposition 7.48 in \cite{Vysoky:2022gm}. Note that its underlying submanifold is the level set of $\ul{\varphi}: \GL(V_{\bullet}) \rightarrow \Sym^{\times}_{0}(V,g)$ corresponding to $\1_{V} \in \Sym^{\times}(V,g)$, let us denote it as $\gO(V_{\bullet},g)$. It follows from (\ref{eq_varphiunderlying}) that 
\begin{equation} \label{eq_gOVbulletg}
\gO(V_{\bullet},g) = \{ A \in \GL(V_{\bullet}) \mid A^{T} g A = g \}.
\end{equation}

The main statement of this section is the following one. See Subsection \ref{subsec_technical4} for the proof. 
\begin{theorem} \label{thm_O}
$(\gO(V,g),j)$ is a graded Lie subgroup of $\GL(V)$. 

Moreover, its associated graded Lie algebra can be naturally identified with the subalgebra $\ao(V,g) \subseteq \gl(V)$, see Proposition \ref{tvrz_Symortdecom} and Proposition \ref{tvrz_subalgebra}. 
\end{theorem}
\end{enumerate}

We have intentionally postponed the proofs to make the construction of $\gO(V,g)$ more streamlined. An uninterested reader can thus safely skip the following section. 
\section{Technical details} \label{sec_technical}
\subsection{Part I: Anti-homomorphism} \label{subsec_technical1}
\begin{proof}[Proof of Proposition \ref{tvrz_tauantihom}]
Recall that $\tau: \gl(V) \rightarrow \gl(V)$ is defined by (\ref{eq_taumap}). Let $\beta$ be the bilinear map (\ref{eq_betaglVmulti}) and let $\beta': \gl(V) \otimes_{\R} \gl(V) \rightarrow \gl(V)$ be the induced graded linear map. Suppose that $\sigma'$ is the canonical flip map
\begin{equation}
\sigma'(A \otimes B) := (-1)^{|A||B|} B \otimes A. 
\end{equation}
Let us start by proving the commutativity of the following diagram in $\gVect$:
\begin{equation} \label{eq_tauantihom1}
\begin{tikzcd}
\gl(V) \otimes_{\R} \gl(V) \arrow{d}{\sigma'} \arrow{r}{\beta'} & \gl(V) \arrow{dd}{\tau}\\
\gl(V) \otimes_{\R} \gl(V) \arrow{d}{\tau \otimes \tau}& \\
\gl(V) \otimes_{\R} \gl(V) \arrow{r}{\beta'} & \gl(V)
\end{tikzcd}.
\end{equation}
This is equivalent to the equation 
\begin{equation} \label{eq_tauoriginalantihom}
\tau(AB) = (-1)^{|A||B|} \tau(B) \tau(A),
\end{equation}
for all $A,B \in \gl(V)$. But this is easily verified using (\ref{eq_comptranspose}) and (\ref{eq_taumap}). One can now apply the $\diamond$ functor from Proposition \ref{tvrz_diamant} onto (\ref{eq_tauantihom1}) and consider the diagram
\begin{equation}
\begin{tikzcd}
\dia{\gl(V)} \times \dia{\gl(V)} \arrow{d}{\sigma_{0}} \arrow[bend left=25]{rr}{\dia{\beta}} \arrow{r}{\dia{\alpha}} & \dia{(\gl(V) \otimes_{\R} \gl(V))} \arrow{d}{\dia{\sigma}'} \arrow{r}{\dia{\beta}'} & \dia{\gl(V)} \arrow{dd}{\dia{\tau}}\\
\dia{\gl(V)} \times \dia{\gl(V)} \arrow{d}{\dia{\tau} \times \dia{\tau}} \arrow{r}{\dia{\alpha}} & \dia{(\gl(V) \otimes_{\R} \gl(V))} \arrow{d}{\dia{(\tau \otimes \tau)}}& \\
\dia{\gl(V)} \times \dia{\gl(V)} \arrow[bend right=25] {rr}{\dia{\beta}} \arrow{r}{\dia{\alpha}} & \dia{(\gl(V) \otimes_{\R} \gl(V))} \arrow{r}{\dia{\beta}'} & \dia{\gl(V)}
\end{tikzcd}.
\end{equation}
Here $\dia{\alpha}$ is the map from Proposition \ref{tvrz_diaalpha}, $\dia{\beta}$ is the graded smooth map induced by $\beta$ as in Corollary \ref{cor_bilineargrad}, and $\sigma_{0}$ is just the graded smooth map flipping the two components of the product graded manifold. Besides the main rectangle, all side squares and triangles also commute. One can see this using the definitions and the coordinate expressions (\ref{eq_diaApullback}) and (\ref{eq_alphapullback}). Consequently, the diagram composed from the outer edges also commutes. Finally, let us consider the diagram 
\begin{equation}
\begin{tikzcd}
& & \GL(V) \arrow{d}{k} \arrow[bend left=35]{dddd}{\tau^{\times}} \\
\GL(V) \times \GL(V) \arrow[bend left=10]{rru}{\mu} \arrow{d}{\sigma} \arrow{r}{k \times k} & \dia{\gl(V)} \times \dia{\gl(V)} \arrow{d}{\sigma_{0}} \arrow{r}{\dia{\beta}} & \dia{\gl(V)} \arrow{dd}{\dia{\tau}} \\
\GL(V) \times \GL(V) \arrow{d}{\tau^{\times} \times \tau^{\times}} \arrow{r}{k \times k} & \dia{\gl(V)} \times \dia{\gl(V)} \arrow{d}{\dia{\tau} \times \dia{\tau}}& \\
\GL(V) \times \GL(V) \arrow[bend right=10]{rrd}{\mu} \arrow{r}{k \times k} & \dia{\gl(V)} \times \dia{\gl(V)} \arrow{r}{\dia{\beta}}  & \dia{\gl(V)} \\
& & \GL(V) \arrow{u}{k}
\end{tikzcd}.
\end{equation}
Recall that $k: \GL(V) \rightarrow \dia{\gl(V)}$ is an open embedding. The main rectangle is commutative by the previous paragraph and all the side squares or triangles commute by definitions of the involved maps, see e.g. (\ref{eq_muviaksandbeta}). But this means that both paths through the diagram (\ref{eq_tauantihom}) composed with the monomorphic arrow $k$ are equal. Hence they themselves must be equal and the proof is finished. 
\end{proof}

For future reference, let us recall the following:
\begin{lemma} \label{lem_antihom}
Let $(\G,\mu,\iota,e)$ and $(\H,\mu',\iota',e')$ be graded Lie groups. Suppose that a graded smooth map $\varphi: \G \rightarrow \H$ is their anti-homomorphism, that is 
\begin{equation}
\begin{tikzcd}
\G \times \G \arrow{d}{\sigma} \arrow{r}{\mu} & \G \arrow{dd}{\varphi}\\
\G \times \G \arrow{d}{\varphi \times \varphi} & \\
\H \times \H \arrow{r}{\mu'} & \H 
\end{tikzcd}
\end{equation}
commutes. Then the following diagrams commute automatically:
\begin{equation} \label{eq_morphismsother}
\begin{tikzcd}
\G \arrow{r}{\varphi} \arrow{d}{\iota} & \H \arrow{d}{\iota'} \\
\G \arrow{r}{\varphi} & \H
\end{tikzcd}, \; \; 
\begin{tikzcd}
\{ \ast \} \arrow{rd}{e'} \arrow{r}{e} & \G \arrow{d}{\varphi} \\
& \H 
\end{tikzcd}.
\end{equation}
\end{lemma}
\begin{proof}
This is true for any group object in any suitable category. The proof consists of recalling the proof for ordinary groups and rewriting it as compositions of morphisms. 
\end{proof}
\subsection{Part II: Lifting the map} \label{subset_technical2}
We will need the following technical statement:
\begin{lemma} \label{lem_pullbackPsubspace}
Let $V \in \gVect$. Suppose there is a subspace $P \subseteq V$ and a surjective $A \in \Lin(V,W)$ such that $P = \ker(A)$. Let $i: P \rightarrow V$ be the inclusion. Then the diagram
\begin{equation}
\begin{tikzcd}
\dia{P} \arrow{r} \arrow{d}{\dia{i}} & \dia{\{0\}} \arrow{d}{\dia{0}} \\
\dia{V} \arrow{r}{\dia{A}} & \dia{W}
\end{tikzcd}
\end{equation}
is a pullback in $\gMan^{\infty}$. In other words, the submanifold $(\dia{P},\dia{i})$ is equivalent to the regular level set submanifold of $\dia{A}$ corresponding to the regular value $0 \in W_{0}$. 
\end{lemma}
\begin{proof}
The diagram commutes as it is just the $\diamond$ functor applied to the obvious diagram in $\gVect$. One only has to check the universal property of pullback. Hence suppose there is $\M \in \gMan^{\infty}$ and a graded smooth map $\varphi: \M \rightarrow \dia{V}$ such that
\begin{equation} \label{eq_commutativitytopullbackkerA}
\begin{tikzcd}
\M \arrow{r} \arrow{d}{\varphi} & \dia{\{0 \}} \arrow{d}{\dia{0}} \\
\dia{V} \arrow{r}{\dia{A}} & \dia{W}
\end{tikzcd}
\end{equation}
commutes. We must show that there is a unique graded smooth map $\hat{\varphi}: \M \rightarrow \dia{P}$ such that $\varphi = \dia{i} \circ \hat{\varphi}$. We can choose a total basis $(t_{1},\dots,t_{m},s_{1},\dots,s_{k})$ for $V$ and $(s'_{1},\dots,s'_{k})$ for $W$, such that $A(t_{\lambda}) = 0$ and $A(s_{\sigma}) = s'_{\sigma}$ for all $\lambda \in \{1,\dots,m\}$ and $\sigma \in \{1, \dots, k\}$. Let $(\bbz^{1},\dots,\bbz^{m},\bbu^{1},\dots,\bbu^{k})$ and $(\bbu'^{1},\dots,\bbu'^{k})$ be the induced coordinates on $\dia{V}$ and $\dia{W}$, respectively.

Then $P = \ker(A)$ is spanned by $(t_{\lambda})_{\lambda=1}^{m}$ there are induced coordinates $(\bbz'^{1},\dots,\bbz'^{m})$ on $\dia{P}$. The explicit expression for the pullback by $\dia{i}$ is then
\begin{equation}
\dia{i}^{\ast}( \bbz^{\lambda}) = \bbz'^{\lambda}, \; \; \dia{i}^{\ast}( \bbu^{\sigma}) = 0, 
\end{equation}
for all $\lambda \in \{1,\dots,m\}$ and $\sigma \in \{1, \dots, k \}$. The explicit expression for the pullback by $\dia{A}$ is 
\begin{equation}
\dia{A}^{\ast}( \bbu'^{\sigma}) = \bbu^{\sigma}, 
\end{equation}
for all $\sigma \in \{1, \dots, k \}$. Define $\hat{\varphi}: \M \rightarrow \dia{P}$ by declaring 
\begin{equation} \label{eq_hatvarphipullbackkerA}
\hat{\varphi}^{\ast}(\bbz'^{\lambda}) := \varphi^{\ast}(\bbz^{\lambda}),
\end{equation}
for all $\lambda \in \{1, \dots, m \}$. Thanks to Proposition \ref{tvrz_onpullback}, this is enough to determine $\hat{\varphi}$ uniquely. Finally, note that the assumed commutativity of (\ref{eq_commutativitytopullbackkerA}) implies $\varphi^{\ast}( \bbu^{\sigma}) = 0$ for all $\sigma \in \{1, \dots, k\}$. It is now trivial to check that indeed $\varphi = \dia{i} \circ \hat{\varphi}$ and the formula (\ref{eq_hatvarphipullbackkerA}) is the only one making this work. 
\end{proof}

\begin{proof}[Proof of Proposition \ref{tvrz_varphi}] Since the graded manifold $\Sym^{\times}(V,g)$ was constructed as a pullback, we will first look for a graded smooth map $\chi: \GL(V) \rightarrow \dia{\Sym(V,g)}$ fitting into the commutative diagram
\begin{equation} \label{eq_varphimap1}
\begin{tikzcd}
\GL(V) \arrow{d}{\varphi^{\times}} \arrow[dashed]{r}{\chi} & \dia{\Sym(V,g)} \arrow{d}{\dia{i}} \\
\GL(V) \arrow{r}{k} & \dia{\gl(V)} 
\end{tikzcd}
\end{equation}
Should we find such a map, the universal property of the pullback (\ref{eq_Symcross}) will give us a unique graded smooth map $\varphi: \GL(V) \rightarrow \Sym^{\times}(V,g)$ fitting into the diagram (\ref{eq_varphidiagram}), thus finishing the proof. 

Now, recall that by Proposition \ref{tvrz_Symortdecom}, we have a decomposition
\begin{equation}
\gl(V) = \Sym(V,g) \oplus \ao(V,g). 
\end{equation}
We can thus construct the surjective projector $\varpi: \gl(V) \rightarrow \ao(V,g)$ such that $\Sym(V,g) = \ker(\varpi)$. Using Lemma \ref{lem_pullbackPsubspace}, we see that the diagram
\begin{equation}
\begin{tikzcd}
\dia{\Sym(V,g)} \arrow{d}{\dia{i}} \arrow{r} & \dia{ \{0\}} \arrow{d}{\dia{0}} \\
\dia{\gl(V)} \arrow{r}{\dia{\varpi}} & \dia{ \ao(V,g)}
\end{tikzcd}
\end{equation}
is a pullback in $\gMan^{\infty}$. We construct $\chi: \GL(V) \rightarrow \dia{\Sym(V,g)}$ by employing the universal property of \emph{this} diagram. To do so, we will prove that $k \circ \varphi^{\times}$ fits into the commutative diagram 
\begin{equation} \label{eq_varphimap2}
\begin{tikzcd}
\GL(V) \arrow[dashed]{d}{k \circ \varphi^{\times}} \arrow{r} & \dia{\{0\}} \arrow{d}{\dia{0}} \\
\dia{\gl(V)} \arrow{r}{\dia{\varpi}} & \dia{\ao(V,g)}
\end{tikzcd}
\end{equation}
If we succeed, we will indeed obtain a unique graded smooth map $\chi: \GL(V) \rightarrow \Sym(V,g)$ fitting into (\ref{eq_varphimap1}). Observe that we can write 
\begin{equation}
k \circ \varphi^{\times} = k \circ \mu \circ (\tau^{\times}, \1_{\GL(V)}) = \dia{\beta} \circ (k \times k) \circ (\tau^{\times}, \1_{\GL(V)}) = \dia{\beta} \circ (\dia{\tau},\1_{\gl(V)}) \circ k.
\end{equation}
Let us write $\chi_{0} := \dia{\beta} \circ (\dia{\tau}, \1_{\gl(V)})$. The uniqueness of arrows into the terminal object $\dia{\{0\}}$ implies that to prove the commutativity of (\ref{eq_varphimap2}), it suffices to prove the commutativity of
\begin{equation} \label{eq_varphidulezity}
\begin{tikzcd}
\dia{\gl(V)} \arrow{r} \arrow{d}{\chi_{0}} & \dia{ \{0\}} \arrow{d}{\dia{0}} \\
\dia{\gl(V)} \arrow{r}{\dia{\varpi}} & \dia{ \ao(V,g)} 
\end{tikzcd}.
\end{equation}
Finally, we can compose both paths along this diagram with the embedding $\dia{l}: \dia{\ao(V,g)} \rightarrow \dia{\gl(V)}$ induced by the inclusion $l: \ao(V,g) \rightarrow \gl(V)$ and prove instead the commutativity of the diagram 
\begin{equation} \label{eq_vaprhifinal}
\begin{tikzcd}
\dia{\gl(V)} \arrow{r} \arrow{d}{\chi_{0}} & \dia{ \{0\}} \arrow{d}{\dia{0}} \\
\dia{\gl(V)} \arrow{r}{\dia{p}} & \dia{\gl(V)}
\end{tikzcd},
\end{equation}
where $p := l \circ \varpi: \gl(V) \rightarrow \gl(V)$ is the projector (\ref{eq_projectoronoVg}). To proceed, we need the coordinate expression for $\dia{\tau}^{\ast}$. To do so, let us introduce the matrices for $g$ and its inverse $g^{-1}$, namely write 
\begin{equation} \label{eq_gmatrices}
g(t_{\lambda}) = \fg_{\lambda \kappa} t^{\kappa}, \; \; g^{-1}(t^{\lambda}) = \fg^{\lambda \kappa} t_{\kappa},
\end{equation}
where $|\fg_{\lambda \kappa}| = |t_{\lambda}| + |t_{\kappa}| + \ell$ and $\fg^{\lambda \kappa} =  - |t_{\kappa}| - |t_{\lambda}| - \ell$. Now, let $A \in \gl(V)$ be any graded linear map. We write $A(t_{\lambda}) = \fA_{\lambda}{}^{\kappa} t_{\kappa}$, where $|\fA_{\lambda}{}^{\kappa}| = |t_{\lambda}| - |t_{\kappa}| + |A|$. Then
\begin{equation}
A^{T}(t^{\lambda}) = (-1)^{|t_{\lambda}|(|t_{\kappa}| - |t_{\lambda}|)} \fA_{\kappa}{}^{\lambda} t^{\kappa}. 
\end{equation}
Consequently, one has 
\begin{equation}
[\tau(A)](t_{\lambda}) = (-1)^{\ell(|A| - |t_{\lambda}| - |t_{\mu}| - 1) + |t_{\nu}|(|A| + |t_{\mu}| - 1) + |A||t_{\lambda}|} \fg_{\lambda \nu} \fA_{\mu}{}^{\nu} \fg^{\mu \kappa} t_{\kappa}. 
\end{equation}
Since the coefficients of $\tau(A)$ in the standard basis can be read out from its matrix, we have
\begin{equation} \label{eq_tauAmatrix}
\tau(A) = (-1)^{\ell(|A| - |t_{\lambda}| - |t_{\mu}| - 1) + |t_{\nu}|(|A| + |t_{\mu}| - 1) + |A||t_{\lambda}|} \fg_{\lambda \nu} \fA_{\mu}{}^{\nu} \fg^{\mu \kappa}  \Delta_{\kappa}{}^{\lambda} 
\end{equation}
By plugging in $A = \Delta_{\rho}{}^{\sigma}$ and noting that then $|A| = |t_{\rho}| - |t_{\sigma}|$ and $\fA_{\mu}{}^{\nu} = \delta^{\nu}_{\rho} \delta_{\mu}^{\sigma}$, one finds
\begin{equation} \label{eq_taucoordexpression}
\tau( \Delta_{\rho}{}^{\sigma}) = (-1)^{\ell(|t_{\rho}| - |t_{\lambda}| + 1) + |t_{\lambda}|(|t_{\rho}| - |t_{\sigma}|)} \fg_{\lambda \rho} \fg^{\sigma \kappa} \Delta_{\kappa}{}^{\lambda}. 
\end{equation}
But from this we can read out the matrix of $\tau: \gl(V) \rightarrow \gl(V)$ with respect to the standard basis and thus obtain the expression for $\dia{\tau}^{\ast}$ from (\ref{eq_diaApullback}). We find 
\begin{equation}
\begin{split}
\dia{\tau}^{\ast}( \bby^{\kappa}{}_{\lambda}) = & \  (-1)^{|t_{\lambda}| + |t_{\kappa}|(|t_{\rho}| - |t_{\sigma}| - 1) + \ell( |t_{\rho}| - |t_{\lambda}| + 1)} \fg_{\lambda \rho} \fg^{\sigma \kappa} \bby^{\rho}{}_{\sigma} \\
= & \ \fM^{\kappa}{}_{\lambda}{}^{\sigma}{}_{\rho} \bby^{\rho}{}_{\sigma},
\end{split}
\end{equation}
where we have defined the functions
\begin{equation} \label{eq_fMmaticehnusnablba}
\fM^{\kappa}{}_{\lambda}{}^{\sigma}{}_{\rho} := (-1)^{|t_{\lambda}| + |t_{\kappa}|(|t_{\rho}| - |t_{\sigma}| - 1) + \ell( |t_{\rho}| - |t_{\lambda}| + 1)} \fg_{\lambda \rho} \fg^{\sigma \kappa}. 
\end{equation}
The commutativity of (\ref{eq_vaprhifinal}) is equivalent to to proving that 
\begin{equation} \label{eq_chi0diapequation}
(\chi_{0}^{\ast} \circ \dia{p}^{\ast})(\bby^{\kappa}{}_{\lambda}) = 0,
\end{equation}
for all $\kappa,\lambda \in \{1, \dots, n\}$. By applying the first pullback, we get 
\begin{equation}
\dia{p}^{\ast}( \bby^{\kappa}{}_{\lambda}) = \frac{1}{2}\big( \bby^{\kappa}{}_{\lambda} - \fM^{\kappa}{}_{\lambda}{}^{\sigma}{}_{\rho} \bby^{\rho}{}_{\sigma} \big).
\end{equation}
Since $\chi_{0} = \dia{\beta} \circ (\dia{\tau},\1_{\gl(V)})$, the next step is to apply $\dia{\beta}^{\ast}$. This is given by the same formula as (\ref{eq_muastexplicit}). We again write $(\bbz^{\lambda}{}_{\kappa}, \bbu^{\lambda}{}_{\kappa})$ for coordinates on $\dia{\gl(V)} \times \dia{\gl(V)}$. One has 
\begin{equation} \label{eq_betaastdiapequation}
(\dia{\beta}^{\ast} \circ \dia{p}^{\ast})(\bby^{\kappa}{}_{\lambda}) = \frac{1}{2}\big( \bbu^{\nu}{}_{\lambda} \bbz^{\kappa}{}_{\nu} - \fM^{\kappa}{}_{\lambda}{}^{\sigma}{}_{\rho} \bbu^{\nu}{}_{\sigma} \bbz^{\rho}{}_{\nu} \big).
\end{equation}
Next, note that the pullback by $(\dia{\tau},\1_{\gl(V)}): \dia{\gl(V)} \rightarrow \dia{\gl(V)} \times \dia{\gl(V)}$ is given by 
\begin{align}
(\dia{\tau},\1_{\gl(V)})^{\ast}(\bbz^{\kappa}{}_{\lambda}) = & \ \fM^{\kappa}{}_{\lambda}{}^{\sigma}{}_{\rho} \bby^{\rho}{}_{\sigma}, \\
(\dia{\tau},\1_{\gl(V)})^{\ast}(\bbu^{\kappa}{}_{\lambda}) = & \  \bby^{\kappa}{}_{\lambda}.
\end{align}
First, one thus finds the expression
\begin{equation} \label{eq_theonewewanttobeat}
(\dia{\tau}, \1_{\gl(V)})^{\ast}(\bbu^{\nu}{}_{\lambda} \bbz^{\kappa}{}_{\nu}) = \bby^{\nu}{}_{\lambda} \fM^{\kappa}{}_{\nu}{}^{\sigma}{}_{\rho} \bby^{\rho}{}_{\sigma}. 
\end{equation}
It follows from (\ref{eq_betaastdiapequation}) that in order to prove (\ref{eq_chi0diapequation}), one has to show that this expression is equal to 
\begin{equation} \label{eq_thisisthestart}
(\dia{\tau},\1_{\gl(V)})^{\ast}( \fM^{\kappa}{}_{\lambda}{}^{\sigma}{}_{\rho} \bbu^{\nu}{}_{\sigma} \bbz^{\rho}{}_{\nu}) = \fM^{\kappa}{}_{\lambda}{}^{\sigma}{}_{\rho} \bby^{\nu}{}_{\sigma} \fM^{\rho}{}_{\nu}{}^{\alpha}{}_{\mu} \bby^{\mu}{}_{\alpha}. 
\end{equation}
At this moment, we must plug back the expression (\ref{eq_fMmaticehnusnablba}). Up to a sign, which we will now ignore, one obtains the expression of the form
\begin{equation}
\fg_{\lambda \rho} \fg^{\sigma \kappa} \bby^{\nu}{}_{\sigma} \fg_{\nu \mu} \fg^{\alpha \rho} \bby^{\mu}{}_{\alpha}. 
\end{equation}
One has to reshuffle the terms a bit. We aim to obtain the expression 
\begin{equation} \label{eq_almosttheone}
\fg_{\lambda \rho} \fg^{\alpha \rho} \bby^{\mu}{}_{\alpha}  \fg_{\nu \mu} \fg^{\sigma \kappa} \bby^{\nu}{}_{\sigma}.
\end{equation}
By doing so, one acquires an additional sign
\begin{equation} \label{eq_reshufflingsign}
(-1)^{(\ell + |t_{\rho}| + |t_{\mu}|)(|t_{\mu}| - |t_{\kappa}|) + (\ell + |t_{\nu}| + |t_{\mu}|)(\ell + |t_{\kappa}| + |t_{\nu}|)} 
\end{equation}
We are already getting close. At \textit{this moment and this moment only}, we shall utilize the fact that $g$ is symmetric. It follows from (\ref{eq_gradedsymmetrygflat}) that 
\begin{equation} \label{eq_gradedsymmetrymatrices}
\fg_{\nu \mu} = (-1)^{|t_{\nu}||t_{\mu}| + \ell} \fg_{\mu \nu}, \; \; \fg^{\alpha \rho} = (-1)^{|t_{\alpha}| + |t_{\rho}| + |t_{\alpha}||t_{\rho}|} \fg^{\rho \alpha}. 
\end{equation}
Using this to swap some indices in the expression (\ref{eq_almosttheone}), we get the term proportional to
\begin{equation} \label{eq_almostfinalexpression}
\fg_{\lambda \rho} \fg^{\rho \alpha} \bby^{\mu}{}_{\alpha}  \fg_{\mu \nu} \fg^{\sigma \kappa} \bby^{\nu}{}_{\sigma}.
\end{equation}
Note that we used the symmetry of $g$ \textit{exactly twice}. This will be important later. Except for the first two terms, this has the index structure in accord with (\ref{eq_theonewewanttobeat}). At this very point, we must examine the overall sign. We have to combine the ones coming from (\ref{eq_fMmaticehnusnablba}), the reshuffle sign (\ref{eq_reshufflingsign}) and the two signs from the index swap (\ref{eq_gradedsymmetrymatrices}). The resulting sign does not fully reflect the horrors endured during its calculation. One gets
\begin{equation} \label{eq_almostfinalsign}
(-1)^{|t_{\lambda}| + |t_{\alpha}| - |t_{\mu}| + |t_{\kappa}| (|t_{\nu}| - |t_{\sigma}| - 1) + \ell (|t_{\rho}| + |t_{\lambda}| + |t_{\nu}| - |t_{\mu}|)}
\end{equation}
Finally, the fact that $g^{-1}$ is inverse to $g$ provides the identity
\begin{equation} \label{eq_ginversegmatrix}
(-1)^{\ell (|t_{\lambda}| + |t_{\rho}| + 1)} \fg_{\lambda \rho} \fg^{\rho \alpha} = \delta_{\lambda}^{\alpha}. 
\end{equation}
Note that in order to utilize it, one has to be a bit careful - the index $\rho$ is not free in this formula. Luckily, we see that $|t_{\rho}|$ appears in (\ref{eq_almostfinalsign}) exactly as required. By using this to simplify the combination of (\ref{eq_almostfinalexpression}) and (\ref{eq_almostfinalsign}), one gets
\begin{equation}
(-1)^{|t_{\mu}| + |t_{\kappa}| (|t_{\nu}| - |t_{\sigma}| - 1) + \ell (|t_{\nu}| - |t_{\mu}| + 1)} \bby^{\mu}{}_{\lambda} \fg_{\mu \nu} \fg^{\sigma \kappa} \bby^{\nu}{}_{\sigma}. 
\end{equation}
It remains to relabel the dummy indices $\mu \rightarrow \nu, \nu \rightarrow \rho$ to get 
\begin{equation}
(-1)^{|t_{\nu}| + |t_{\kappa}|( |t_{\rho}| - |t_{\sigma}| - 1)  + \ell( |t_{\rho}| - |t_{\nu}| + 1)} \bby^{\nu}{}_{\lambda} \fg_{\nu \rho} \fg^{\sigma \kappa} \bby^{\rho}{}_{\sigma} = \bby^{\nu}{}_{\lambda} \fM^{\kappa}{}_{\nu}{}^{\sigma}{}_{\rho} \bby^{\rho}{}_{\sigma}.
\end{equation}
But this is precisely the expression (\ref{eq_theonewewanttobeat}), which was to be proved. 
\end{proof}
\subsection{Part III: Regular value} \label{subsec_technical3}
\begin{proof}[Proof of Proposition \ref{tvrz_regularvalue}]
We have to argue that for each $A \in \gO(V_{\bullet},g)$, the tangent map 
\begin{equation}
T_{A}\varphi: T_{A} \GL(V) \rightarrow T_{\1_{V}}(\Sym^{\times}(V,g))
\end{equation}
is surjective. Let us simplify this task. First, observe that during the construction of $\varphi$, we have proved the commutativity of the diagram (\ref{eq_varphidulezity}). This implies that there is a unique graded smooth map $\varphi_{0}: \dia{\gl(V)} \rightarrow \dia{\Sym(V,g)}$ fitting into the diagram
\begin{equation} \label{eq_varphi0}
\begin{tikzcd}
& \dia{\Sym(V,g)} \arrow{d}{\dia{i}} \\
\dia{\gl(V)} \arrow[dashed]{ur}{\varphi_{0}} \arrow{r}{\chi_{0}} & \dia{\gl(V)}
\end{tikzcd}
\end{equation}
But then $\dia{i} \circ (\varphi_{0} \circ k) = \chi_{0} \circ k = \dia{i} \circ \chi$, which implies that $\varphi_{0} \circ k = \chi$ by the uniqueness claim in the construction of $\chi: \GL(V) \rightarrow \dia{\Sym(V,g)}$. Finally, it follows from the construction of $\varphi$ that $\varphi_{0}$ fits into the commutative diagram
\begin{equation} \label{eq_varphivarphi0}
\begin{tikzcd} 
\GL(V) \arrow{r}{\varphi} \arrow{d}{k} & \Sym^{\times}(V,g) \arrow{d}{k'} \\
\dia{\gl(V)} \arrow{r}{\varphi_{0}} & \dia{\Sym(V,g)},
\end{tikzcd}
\end{equation}
where $k': \Sym^{\times}(V,g) \rightarrow \Sym(V,g)$ is defined by (\ref{eq_Symcross}). Note that this is an open embedding, so $\Sym^{\times}(V,g)$ can be viewed as a restriction of $\dia{\Sym}(V,g)$ to the open subset $\Sym(V,g)_{0} \cap \GL(V_{\bullet})$. In particular, it suffices to prove that the tangent map
\begin{equation}
T_{A}\varphi_{0}: T_{A}(\dia{\gl(V)}) \rightarrow T_{\1_{V}}(\dia{\Sym(V,g)})
\end{equation}
is surjective. Finally, it follows from (\ref{eq_varphi0}) that it suffices to prove that the tangent map
\begin{equation}
T_{A}\chi_{0}: T_{A}(\dia{\gl(V)}) \rightarrow T_{\1_{V}}(\dia{\gl(V)}) 
\end{equation}
is surjective when the codomain is restricted to the subspace of the tangent space corresponding to the submanifold $(\dia{\Sym(V,g)}, \dia{i})$. We utilize Proposition \ref{tvrz_tangentVSidentification} and the canonical identifications
\begin{equation}
\gl(V) \cong T_{A}( \dia{\gl(V)}), \; \; \gl(V) \cong T_{\1_{V}}( \dia{\gl(V)}).
\end{equation}
With respect to those identifications, $T_{A}\chi_{0}$ can be identified with a degree zero graded linear map $L_{A}: \gl(V) \rightarrow \gl(V)$, given for each $X \in \gl(V)$ by the formula
\begin{equation} \label{eq_LAmap}
L_{A}(X) := \tau(A) X + \tau(X) A. 
\end{equation}
Before actually showing this, observe that the subspace of the tangent space corresponding to $(\dia{\Sym(V,g)}, \dia{i})$ is identified with $\Sym(V,g) \subseteq \gl(V)$. To prove the claim about $T_{A}\chi_{0}$, one thus only has to show that $L_{A}: \gl(V) \rightarrow \Sym(V,g)$ is surjective. But for any $S \in \Sym(V,g)$, one has 
\begin{equation}
\begin{split}
L_{A}( \frac{1}{2}A S) = & \ \frac{1}{2} \tau(A) (AS) + \frac{1}{2} \tau(AS)A \\
= & \ \frac{1}{2} (\tau(A) A) S + \frac{1}{2} \tau(S) (\tau(A) A) \\
= & \ \frac{1}{2} S + \frac{1}{2} \tau(S) = S,
\end{split}
\end{equation}
where $\tau(A)A = \1_{V}$ follows from the assumption $A \in \gO(V_{\bullet},g)$, see (\ref{eq_gOVbulletg}). We have also utilized the property (\ref{eq_tauoriginalantihom}) and Proposition \ref{tvrz_taumap} implying that $\tau(S) = S$. 

To finish the proof, we have to prove that (\ref{eq_LAmap}) indeed corresponds to the tangent map $T_{A}\chi_{0}$. This is where the actual calculation starts. We will compare the matrix of $L_{A}$ with respect to the standard basis $\Delta_{\lambda}{}^{\kappa}$, and the matrix of $T_{A}\chi_{0}$ with respect to the corresponding basis
\begin{equation} \label{eq_bdeltabasis}
\bDelta_{\lambda}{}^{\kappa} = (-1)^{|t_{\lambda}| - |t_{\kappa}|} \left.\frac{\partial}{\partial \bby^{\lambda}{}_{\kappa}}\right\rvert_{A} 
\end{equation} 
for $T_{A}( \dia{\gl(V)})$ and the similarly defined (and denoted by the same symbols) basis for $T_{\1_{V}}( \dia{\gl(V)})$. 

At this moment, it is convenient to introduce a single symbol $\mind{\lambda}{\kappa}$ to denote the index corresponding to the basis vector $\Delta_{\kappa}{}^{\lambda} \in \gl(V)$. Consequently, let us define the matrix $\fT$ of $\tau$ by 
\begin{equation}
\tau( \Delta_{\rho}{}^{\sigma}) = \fT_{\mind{\sigma}{\rho}}{}^{\mind{\lambda}{\kappa}} \Delta_{\kappa}{}^{\lambda}. 
\end{equation}
Its exact form will not be important in this proof, but note that it follows from (\ref{eq_taucoordexpression}) that 
\begin{equation} \label{eq_fTmatrix}
\fT_{\mind{\sigma}{\rho}}{}^{\mind{\lambda}{\kappa}} = (-1)^{\ell(|t_{\rho}| - |t_{\lambda}| + 1) + |t_{\lambda}|(|t_{\rho}| - |t_{\sigma}|)} \fg_{\lambda \rho} \fg^{\sigma \kappa}. 
\end{equation}
This allows us to write 
\begin{equation} \label{eq_LAmatrix}
\begin{split}
L_{A}( \Delta_{\mu}{}^{\nu}) = & \ \tau(A) \Delta_{\mu}{}^{\nu} + \tau(\Delta_{\mu}{}^{\nu}) A \\
= & \ \big( \fA_{\sigma}{}^{\rho} \fT_{\mind{\sigma}{\rho}}{}^{\mind{\mu}{\kappa}} \delta_{\lambda}^{\nu} + (-1)^{(|t_{\kappa}| - |t_{\sigma}|)(|t_{\lambda}| - |t_{\sigma}|)} \fT_{\mind{\nu}{\mu}}{}^{\mind{\sigma}{\kappa}} \fA_{\lambda}{}^{\sigma} \big) \Delta_{\kappa}{}^{\lambda}, 
\end{split}
\end{equation}
where $A = \fA_{\sigma}{}^{\rho} \Delta_{\rho}{}^{\sigma}$. We compare this to the differential of $\chi_{0} := \dia{\beta} \circ (\dia{\tau}, \1_{\gl(V)})$. One has
\begin{equation}
\chi_{0}^{\ast}( \bby^{\kappa}{}_{\lambda}) = (-1)^{(|t_{\kappa}| - |t_{\alpha}|)(|t_{\rho}| - |t_{\sigma}| - 1)} \bby^{\alpha}{}_{\lambda} \fT_{\mind{\sigma}{\rho}}{}^{\mind{\alpha}{\kappa}} \bby^{\rho}{}_{\sigma}. 
\end{equation} 
This can be shown in the ways similar to those after (\ref{eq_chi0diapequation}), except that we do not insert (\ref{eq_fTmatrix}). It is convenient to reshuffle this a bit - namely interchange the positions of $\bby^{\alpha}{}_{\lambda}$ and $\bby^{\rho}{}_{\sigma}$. After counting the signs, one gets
\begin{equation} \label{eq_chi0explicitfinal}
\chi_{0}^{\ast}( \bby^{\kappa}{}_{\lambda}) = (-1)^{|t_{\rho}| - |t_{\sigma}| + |t_{\kappa}| + |t_{\alpha}|(|t_{\kappa}| - |t_{\lambda}|) + |t_{\kappa}||t_{\lambda}|} \bby^{\rho}{}_{\sigma} \fT_{\mind{\sigma}{\rho}}{}^{\mind{\alpha}{\kappa}} \bby^{\alpha}{}_{\lambda}. 
\end{equation}
At this moment, let us recall the well-known formula 
\begin{equation}
[T_{A}\chi_{0}]( \left.\frac{\partial}{\partial \bby^{\mu}{}_{\nu}}\right\rvert_{A}) = \frac{\partial( \chi_{0}^{\ast}(\bby^{\kappa}{}_{\lambda}))}{\partial \bby^{\mu}{}_{\nu}}(A)  \left.\frac{\partial}{\partial \bby^{\kappa}{}_{\lambda}}\right\rvert_{\1_{V}}.
\end{equation}
See e.g. Example 4.9 in \cite{Vysoky:2022gm}. In terms of the bases (\ref{eq_bdeltabasis}), one has 
\begin{equation} \label{eq_TAchi0vbdelta}
[T_{A}\chi_{0}](\bDelta_{\mu}{}^{\nu}) = (-1)^{|t_{\mu}| - |t_{\nu}| + |t_{\lambda}| - |t_{\kappa}|} \frac{\partial( \chi_{0}^{\ast}(\bby^{\kappa}{}_{\lambda}))}{\partial  \bby^{\mu}{}_{\nu}}(A) \bDelta_{\kappa}{}^{\lambda} 
\end{equation}
It turns out that the best course of action is to write the real number $\bby^{\nu}{}_{\lambda}(A)$ as if it is obtained by the action of the dual basis vector $\cD^{\nu}{}_{\lambda}$ on $A \in \gl(V)$, see (\ref{eq_dualstandard}). In other words, one writes
\begin{equation}
\bby^{\nu}{}_{\lambda}(A) = (-1)^{|t_{\lambda}| - |t_{\nu}|} \fA_{\lambda}{}^{\nu}. 
\end{equation}
Observe that this is actually a completely correct formula. It remains to plug (\ref{eq_chi0explicitfinal}) into (\ref{eq_TAchi0vbdelta}). After some sign cancellations and index relabeling, one arrives to 
\begin{equation}
[T_{A}\chi_{0}](\bDelta_{\mu}{}^{\nu}) = \big( \fA_{\sigma}{}^{\rho} \fT_{\mind{\sigma}{\rho}}{}^{\mind{\mu}{\kappa}} \delta_{\lambda}^{\nu} + (-1)^{(|t_{\kappa}| - |t_{\sigma}|)(|t_{\lambda}| - |t_{\sigma}|)} \fT_{\mind{\nu}{\mu}}{}^{\mind{\sigma}{\kappa}} \fA_{\lambda}{}^{\sigma} \big) \bDelta_{\kappa}{}^{\lambda}.
\end{equation}
This is completely the same expression as the one for $L_{A}$ in (\ref{eq_LAmatrix}). This finishes the proof. 
\end{proof}
\subsection{Part IV: It is a subgroup!} \label{subsec_technical4}
\begin{proof}[Proof of Theorem \ref{thm_O}] 
First, we must show that there exists a graded smooth map $\mu': \gO(V,g) \times \gO(V,g) \rightarrow \gO(V,g)$ fitting into the commutative diagram
\begin{equation} \label{eq_mu'diagram}
\begin{tikzcd}
\gO(V,g) \times \gO(V,g) \arrow[dashed]{r}{\mu'} \arrow{d}{j \times j} & \gO(V,g) \arrow{d}{j} \\
\GL(V) \times \GL(V) \arrow{r}{\mu} & \GL(V)
\end{tikzcd}
\end{equation}
Since $\gO(V,g)$ is constructed by the pullback (\ref{eq_O}), it suffices to prove that the diagram 
\begin{equation}
\begin{tikzcd}
\gO(V,g) \times \gO(V,g) \arrow{r} \arrow{d}{\mu \circ (j \times j)} & \{ \ast \} \arrow{d}{e^{\times}} \\
\GL(V) \arrow{r}{\varphi} & \Sym^{\times}(V,g)
\end{tikzcd}
\end{equation}
commutes. We can compose both paths along the diagram with the embedding $\dia{i}': \Sym^{\times}(V,g) \rightarrow \GL(V)$ and prove the commutativity of the diagram
\begin{equation} \label{eq_gOVgmult}
\begin{tikzcd}
\gO(V,g) \times \gO(V,g) \arrow{r} \arrow{d}{\mu \circ (j \times j)} & \{ \ast \} \arrow{d}{e} \\
\GL(V) \arrow{r}{\varphi^{\times}} & \GL(V)
\end{tikzcd}
\end{equation}
instead. Recall that $\varphi^{\times} = \mu \circ (\tau^{\times}, \1_{\GL(V)})$. In the following calculations, $p_{1}$ and $p_{2}$ will always denote canonical projections onto the first and the second component of a product graded manifold, respectively. We also write $\G = \GL(V)$ and $\H = \gO(V,g)$ to save some space. In the following, we will closely follow the usual proof of the implication $A,B \in \gO(V,g) \Rightarrow AB \in \gO(V,g)$. In each step, we thus include the corresponding ordinary manipulation in \cH{brown}. Therefore, we are proving the graded analogue of $\cH{\tau(AB)(AB) = \1_{V}}$. 
\begin{enumerate}[(1)]
\item Observe that one can write 
\begin{equation}
\begin{split}
\varphi^{\times} \circ \mu \circ (j \times j) = & \ \mu \circ (\tau^{\times},\1_{\G}) \circ \mu \circ (j \times j) \\
= & \ \mu \circ (\1_{\G} \times \mu) \circ (\tau^{\times} \circ \mu, (p_{1},p_{2})) \circ (j \times j).
\end{split}
\end{equation}
Now utilize the associativity (\ref{eq_associativity}) and write
\begin{equation}
\begin{split}
\mu \circ (\1_{\G} \times \mu) \circ (&\tau^{\times} \circ \mu, (p_{1},p_{2})) \circ (j \times j) = \\
= & \  \mu \circ (\mu \times \1_{\G}) \circ ((\tau^{\times} \circ \mu, p_{1}), p_{2}) \circ (j \times j) \\
 = & \ \mu \circ \big(\mu \circ (\tau^{\times} \circ \mu \circ (q_{1},q_{2}), q_{1}), q_{2}\big), 
\end{split}
\end{equation}
where we define $q_{1} := j \circ p_{1}$ and $q_{2} = j \circ p_{2}$. 
\[
\cH{\tau(AB)(AB) = (\tau(AB)A)B}. 
\]

\item At this moment, we recall the fact that $\tau^{\times}$ is an anti-homomorphism and utilize the commutativity of (\ref{eq_tauantihom}) to write 
\begin{equation}
\begin{split}
\mu \circ \big(\mu \circ (\tau^{\times} \circ \mu \circ (q_{1},q_{2}), q_{1}), q_{2}\big) = & \ \mu \circ \big(\mu \circ (\mu \circ (\tau^{\times} \times \tau^{\times}) \circ (q_{2},q_{1}),q_{1}), q_{2} \big) \\
= & \ \mu \circ \big( \mu \circ (\mu \times \1_{\G}) \circ ((\tau^{\times} \circ q_{2}, \tau^{\times} \circ q_{1}), q_{1}), q_{2} \big).
\end{split}
\end{equation}
\[
\cH{(\tau(AB)A)B = ( (\tau(B)\tau(A))A)B}.
\]
\item We can now use the associativity (\ref{eq_associativity}) again to get
\begin{equation}
\begin{split}
\mu \circ \big(& \mu \circ (\mu \times \1_{\G}) \circ ((\tau^{\times} \circ q_{2}, \tau^{\times} \circ q_{1}), q_{1}), q_{2} \big) = \\
= & \ \mu \circ \big( \mu \circ (\1_{\G} \times \mu) \circ (\tau^{\times} \circ q_{2}, (\tau^{\times} \circ q_{1}, q_{1})), q_{2} \big) \\
= & \ \mu \circ \big( \mu \circ ( \tau^{\times} \circ q_{2}, \mu \circ (\tau^{\times},\1_{\G}) \circ q_{1}), q_{2} \big) \\
= & \ \mu \circ \big( \mu \circ ( \tau^{\times} \circ q_{2}, \varphi^{\times} \circ q_{1}), q_{2} \big)
\end{split}
\end{equation}
\[
\cH{\tau( (\tau(B)\tau(A))A)B = ( \tau(B)(\tau(A)A))B}.
\]
\item At this moment, we use the commutativity of the diagram (\ref{eq_O}) composed with the embedding $\dia{i}'$, which gives $\varphi^{\times} \circ j = e_{\H} = e_{\G} \circ j$. Consequently, one can continue and write 
\begin{equation}
\begin{split}
\mu \circ \big( \mu \circ ( \tau^{\times} \circ q_{2}, \varphi^{\times} \circ q_{1}), q_{2} \big) = & \  \mu \circ \big( \mu \circ (\tau^{\times} \circ q_{2}, e_{\G} \circ q_{1}), q_{2}) \\
= & \ \mu \circ \big((\mu \circ (\tau^{\times} \circ q_{2}, e_{\G} \circ \tau^{\times} \circ q_{2}), q_{2}\big) \\
= & \ \mu \circ \big(\mu \circ (\1_{\G},e_{\G}) \circ \tau^{\times} \circ q_{2}, q_{2}\big), 
\end{split}
\end{equation}
where we have used the fact that $e_{\G} \circ q_{1} = e_{\G} \circ \tau^{\times} \circ q_{2}$. This is because arrows into terminal objects are unique. 
\[
\cH{( \tau(B)(\tau(A)A))B = (\tau(B)\1_{V})B}. 
\]
\item We now utilize the commutativity of the second unitarity diagram (\ref{eq_unitarity}) to get
\begin{equation}
\begin{split}
\mu \circ \big(\mu \circ (\1_{\G},e_{\G}) \circ \tau^{\times} \circ q_{2}, q_{2}\big) = & \ \mu \circ (\1_{\G} \circ \tau^{\times} \circ q_{2}, q_{2}) \\
= & \ \mu \circ (\tau^{\times},\1_{\G}) \circ q_{2} = \varphi^{\times} \circ q_{2}. 
\end{split}
\end{equation}
\[
\cH{(\tau(B)\1_{V})B = \tau(B)B}. 
\]
\item Finally, the commutativity of (\ref{eq_O}) composed with $\dia{i}'$ gives 
\begin{equation}
\varphi^{\times} \circ q_{2} = \varphi^{\times} \circ j \circ p_{2} = e_{\H} \circ p_{2} = e_{\H \times \H}. 
\end{equation}
\[
\cH{\tau(B)B = \1_{V}}. 
\]
Since $e_{\H \times \H}: \H \times \H \rightarrow \G$ is the composition of the terminal arrow $\H \times \H \rightarrow \{ \ast \}$ and the unit $e: \{\ast \} \rightarrow \G$, we have just proved (\ref{eq_gOVgmult}) commutative. 
\end{enumerate}

This finishes the construction of $\mu'$ fitting into (\ref{eq_mu'diagram}). Next, we must construct the inverse $\iota': \gO(V,g) \rightarrow \gO(V,g)$ fitting into the commutative diagram
\begin{equation} \label{eq_iprime}
\begin{tikzcd}
\gO(V,g) \arrow{d}{j} \arrow{r}{\iota'} & \gO(V,g) \arrow{d}{j} \\
\GL(V) \arrow{r}{\iota}& \GL(V)
\end{tikzcd}
\end{equation}
Similarly to the first part, this boils down to proving the commutativity of the diagram
\begin{equation} \label{eq_gOVginverse}
\begin{tikzcd} 
\gO(V,g) \arrow{d}{\iota \circ j} \arrow{r} & \{ \ast \} \arrow{d}{e} \\
\GL(V) \arrow{r}{\varphi^{\times}} & \GL(V)
\end{tikzcd}
\end{equation}
We are proving the graded analogue of $\cH{\tau( A^{-1})A^{-1} = \1_{V}}$ for all $A \in \gO(V,g)$. We again write $\G := \GL(V)$ and $\H := \gO(V,g)$. 
\begin{enumerate}[(1)]
\item We use the commutativity of the unitarity diagram (\ref{eq_unitarity}) to write
\begin{equation}
\begin{split}
\varphi^{\times} \circ \iota \circ j = & \ \mu \circ (\tau^{\times},\1_{\G}) \circ \iota \circ j \\
= & \ \mu \circ (\1_{\G} \circ \tau^{\times} \circ \iota, \iota) \circ j \\
= & \ \mu \circ \big(\mu \circ (\1_{\G},e_{\G}) \circ \tau^{\times} \circ \iota, \iota \big) \circ j \\
= & \ \mu \circ \big( \mu \circ (\tau^{\times} \circ \iota \circ j, e_{\H}), \iota \circ j). 
\end{split}
\end{equation}
\[
\cH{\tau(A^{-1})A^{-1} = (\tau(A^{-1})\1_{V})A^{-1}}. 
\]
\item We use the fact that $\varphi^{\times} \circ j = e_{\H}$ following from (\ref{eq_gOVginverse}). We can thus write
\begin{equation}
\begin{split}
\mu \circ \big( \mu \circ (\tau^{\times} \circ \iota \circ j, e_{\H}), \iota \circ j) = & \ \mu \circ \big(\mu \circ( \tau^{\times} \circ \iota \circ j, \varphi^{\times} \circ j), \iota \circ j\big) \\
= & \ \mu \circ \big(\mu \circ (\tau^{\times} \circ \iota, \mu \circ (\tau^{\times},\1_{\G}), \iota\big) \circ j \\
= & \ \mu \circ \big(\mu \circ (\1_{\G} \times \mu) \circ (\tau^{\times} \circ \iota, (\tau^{\times},\1_{\G})), \iota \big) \circ j.
\end{split}
\end{equation}
\[
\cH{(\tau(A^{-1})\1_{V})A^{-1} = ( \tau(A^{-1})(\tau(A)A))A^{-1}}.
\]
\item Since $\tau^{\times}$ is an anti-homomorphism by Proposition \ref{tvrz_tauantihom}, we can use Lemma \ref{lem_antihom} to obtain
\begin{equation}
\mu \circ \big(\mu \circ (\1_{\G} \times \mu) \circ (\tau^{\times} \circ \iota, (\tau^{\times},\1_{\G})), \iota \big) \circ j = \mu \circ \big(\mu \circ (\1_{\G} \times \mu) \circ (\iota \circ \tau^{\times}, (\tau^{\times},\1_{\G})), \iota \big) \circ j.
\end{equation}
\[
\cH{(\tau(A^{-1})(\tau(A)A))A^{-1} = ( \tau(A)^{-1}(\tau(A)A))A^{-1}}. 
\]
\item Use the associativity (\ref{eq_associativity}) to get
\begin{equation}
\begin{split}
\mu \circ \big(&\mu \circ (\1_{\G} \times \mu) \circ (\iota \circ \tau^{\times}, (\tau^{\times},\1_{\G})), \iota \big) \circ j = \\
= & \ \mu \circ \big(\mu \circ (\mu \times \1_{\G}) \circ ((\iota \circ \tau^{\times}, \tau^{\times}),\1_{\G}), \iota \big) \circ j \\
= & \ \mu \circ \big(\mu \circ (\mu \circ (\iota,\1_{\G}) \circ \tau^{\times}, \1_{\G}), \iota \big) \circ j.
\end{split}
\end{equation}
\[
\cH{(\tau(A)^{-1}(\tau(A)A))A^{-1} = ((\tau(A)^{-1}\tau(A))A)A^{-1}}. 
\]
\item It remains to use the axioms (\ref{eq_unitarity}) and (\ref{eq_inversion}) to get
\begin{equation}
\begin{split}
\mu \circ \big(\mu \circ (\mu \circ (\iota,\1_{\G}) \circ \tau^{\times}, \1_{\G}), \iota \big) \circ j = & \  \mu \circ (\mu \circ (e_{\G} \circ \tau^{\times}, \1_{\G}), \iota) \circ j \\
= & \ \mu \circ ( \mu \circ (e_{\G}, \1_{\G}), \iota) \circ j \\
= & \ \mu \circ ( \1_{\G}, \iota) \circ j = e_{\G} \circ j = e_{\H}. 
\end{split}
\end{equation}
\[
\cH{((\tau(A)^{-1}\tau(A))A)A^{-1} = (\1_{V}A)A^{-1} = AA^{-1} = \1_{V}}. 
\]
But $e_{\H}$ is precisely the terminal arrow $\H \rightarrow \{\ast \}$ followed by the unit $e: \{ \ast \} \rightarrow \G$ and the commutativity of (\ref{eq_gOVginverse}) is proved. 
\end{enumerate}
This completes the construction of the graded smooth map $\iota'$ fitting into the diagram (\ref{eq_iprime}). 

To finish the proof, it remains to prove that the Lie algebra associated with $\gO(V,g)$ can be canonically identified with the subalgebra $\ao(V,g) \subseteq \gl(V)$. As $\gO(V,g)$ is defined as a regular level set submanifold with respect to $\varphi: \GL(V) \rightarrow \Sym^{\times}(V,g)$, its tangent space at $e' = \1_{V}$ can be identified with the kernel of the tangent map $T_{\1_{V}}\varphi$. Using the same arguments as under (\ref{eq_varphivarphi0}), it suffices to find the kernel of the map $T_{\1_{V}} \chi_{0}$. Under the identification $\gl(V) \cong T_{\1_{V}} (\dia{\gl(V)})$, it corresponds to the kernel of the map $L_{\1_{V}}$ defined by (\ref{eq_LAmap}). One has 
\begin{equation}
\ker(L_{\1_{V}}) = \{ X \in \gl(V) \mid X + \tau(X) = 0 \}. 
\end{equation}
Thanks to Proposition \ref{tvrz_taumap}, this is precisely the subspace $\ao(V,g)$. 
\end{proof}
\section{Functor of points perspective} \label{sec_OGFOP}
In Proposition \ref{tvrz_GLVFOP}, we have observed that the functor of points $\frP = \gMan^{\infty}(-,\GL(V))$ associated with $\GL(V)$ is naturally isomorphic to the functor $\frF$ assigning to each $\cS \in \gMan^{\infty}$ the set of degree zero $\C^{\infty}_{\cS}(S)$-module automorphisms $\Aut(\frM(\cS))$ of the free $\C^{\infty}_{\cS}(S)$-module $\frM(\cS) = \C^{\infty}_{\cS}(S) \otimes_{\R} V$. In this section, we will argue that the functor of points associated with $\gO(V,g)$ can be canonically identified with a certain ``functor of subsets'' of $\frF$.

Let us first show that a degree $\ell$ metric $g$ induces a $\C^{\infty}_{\cS}(S)$-bilinear form $\<\cdot,\cdot\>_{g}$ on $\frM(\cS)$. On generators, it is defined by the formula
\begin{equation}
\< f \otimes v, f' \otimes w \>_{g} := (-1)^{(|v|+\ell)|f'|} ff' g(v,w) \in \C^{\infty}_{\cS}(S),
 \end{equation}
where $f,f' \in \C^{\infty}_{\cS}(S)$ and $v,w \in V$, where we view $g(v,w) \in \R$ as a ``constant function'' of degree $|g(v,w)| = |v| + |w| + \ell$. For general elements $\psi,\psi' \in \frM(\cS)$ and $f \in \C^{\infty}_{\cS}(S)$, one finds the properties
\begin{enumerate}[(i)]
\item $|\<\psi,\psi'\>_{g}| = |\psi| + |\psi'| + \ell$;
\item $\< \psi,\psi'\>_{g} = (-1)^{(|\psi|+\ell)(|\psi'|+\ell)} \<\psi',\psi\>_{g}$; 
\item $\< f \psi, \psi'\>_{g} = f \<\psi,\psi'\>_{g}$, $\<\psi,f\psi'\>_{g} = (-1)^{|f|(|\psi|+\ell)} f \<\psi,\psi'\>_{g}$.
\end{enumerate}
Compare this to Definition 2.1 in \cite{vysoky2022graded}. It is now natural to consider a subset
\begin{equation} \label{eq_gOfrMcSg}
\gO(\frM(\cS),g) := \{ F \in \Aut(\frM(\cS)) \mid \< F(\psi),F(\psi')\>_{g} = \<\psi,\psi'\>_{g} \text{ for all } \psi,\psi' \in \frM(\cS) \}. 
\end{equation}
\begin{tvrz}
For each graded smooth map $\varphi: \cN \rightarrow \cS$, the induced set map $\frF(\varphi): \frF(\cS) \rightarrow \frF(\cN)$ maps the subset $\gO(\frM(\cS),g)$ into $\gO(\frM(\cN),g)$.

Consequently, $\frF'(\cS) := \gO(\frM(\cS),g)$ can be promoted to a functor $\frF': (\gMan^{\infty})^{\op} \rightarrow \Set$ such that the inclusions $\frI_{\cS}: \frF'(\cS) \rightarrow \frF(\cS)$ define a natural transformation $\frI: \frF' \rightarrow \frF$.  
\end{tvrz}
\begin{proof}
Let us use the notation introduced in Proposition \ref{tvrz_GLVFOP}. Observe that the matrix of the bilinear form $\<\cdot,\cdot\>_{g}$ with respect to the frame $\Phi_{\lambda} = (-1)^{|t_{\lambda}|} 1 \otimes t_{\lambda}$ is given by
\begin{equation} \label{eq_gformmatrix}
\< \Phi_{\lambda}, \Phi_{\kappa} \>_{g} = (-1)^{|t_{\lambda}| + |t_{\kappa}|} g(t_{\lambda},t_{\kappa}) = (-1)^{|t_{\lambda}| + |t_{\kappa}| + (|t_{\lambda}| + \ell)\ell} \fg_{\lambda \kappa},
\end{equation}
where we have used (\ref{eq_gflattog}) and (\ref{eq_gmatrices}) in the last step. Any $F \in \frF(\cS)$ can be decomposed as in (\ref{eq_Fframedecomposition}). Suppose $F \in \gO(\frM(\cS),g)$. The condition imposed on $F$ by (\ref{eq_gOfrMcSg}) can be written as 
\begin{equation} \label{eq_FOGcondition0}
(-1)^{|t_{\sigma}|(|t_{\kappa}| - 1)} \fF^{\rho}{}_{\lambda} \< \Phi_{\rho}, \Phi_{\sigma} \>_{g} \fF^{\sigma}{}_{\kappa} = \< \Phi_{\lambda},\Phi_{\kappa}\>_{g}. 
\end{equation}
for all $\lambda,\kappa \in \{1,\dots,n\}$. Inserting (\ref{eq_gformmatrix}), this turns into the condition
\begin{equation} \label{eq_FOGcondition}
(-1)^{|t_{\kappa}|(|t_{\sigma}|-1) + |t_{\rho}|(1+\ell)} \fF^{\rho}{}_{\lambda} \fg_{\rho \sigma} \fF^{\sigma}{}_{\kappa} = (-1)^{|t_{\lambda}|(1+\ell)} \fg_{\lambda \kappa},
\end{equation}
for all $\lambda,\kappa \in \{1, \dots, n\}$. We have only slightly rearranged the signs. By applying $\varphi^{\ast}$ and using the fact that $\fg_{\lambda \kappa}$ are constants, one finds
\begin{equation}
(-1)^{|t_{\kappa}|(|t_{\sigma}|-1) + |t_{\rho}|(1+\ell)} \varphi^{\ast}(\fF^{\rho}{}_{\lambda}) \fg_{\rho \sigma} \varphi^{\ast}(\fF^{\sigma}{}_{\kappa})= (-1)^{|t_{\lambda}|(1+\ell)} \fg_{\lambda \kappa}. \end{equation}
By looking at (\ref{eq_rfFfunctormap}), we conclude that $[\frF(\varphi)](F) \in \gO(\frM(\cN),g)$. This proves the first claim. We can now \textit{define} the arrow map of $\frF'$ to make $\frI = \{ \frI_{\cS} \}_{\cS}$ into a natural transformation. 
\end{proof}
Next, let $\frP' := \gMan^{\infty}(-,\gO(V,g))$ denote the functor of points associated with $\gO(V,g)$. The closed embedding $j: \gO(V,g) \rightarrow \GL(V)$ induces a natural transformation $\frJ: \frP' \rightarrow \frP$. For each $\cS \in \gMan^{\infty}$, its component $\frJ_{\cS}: \frP'(\cS) \rightarrow \frP(\cS)$ is given by composition with $j$. 

\begin{lemma}
For each $\cS$, the set map $\frJ_{\cS}$ is injective. 
\end{lemma}
\begin{proof}
One has $\frJ_{\cS}(\phi) = j \circ \phi$ for all $\phi \in \frP'(\cS)$. Since $j: \gO(V,g) \rightarrow \GL(V)$ is a closed embedding, it is a monomorphism in $\gMan^{\infty}$. This clearly renders $\frJ_{\cS}$ injective. 
\end{proof}

Recall that for each $\cS \in \gMan^{\infty}$, we have constructed a bijection $\Psi_{\cS}: \frP(\cS) \rightarrow \frF(\cS)$. Since $\frP'(\cS)$ is one-to-one with the subset $\im(\frJ_{\cS}) \subseteq \frP(\cS)$, we may ask what is the corresponding subset of $\frF(\cS)$. The expected answer is given by the following statement:
\begin{tvrz} \label{tvrz_FOPOG}
The bijection $\Psi_{\cS}$ maps the subset $\im(\frJ_{\cS})$ onto $\gO(\frM(\cS),g)$. Consequently, there is a canonical natural isomorphism $\Psi': \frP' \rightarrow \frF'$ fitting into the diagram
\begin{equation}
\begin{tikzcd}
\frP' \arrow[dashed]{r}{\Psi'} \arrow{d}{\frJ} & \frF' \arrow{d}{\frI} \\
\frP \arrow{r}{\Psi} & \frF
\end{tikzcd}.
\end{equation}
\end{tvrz}
\begin{proof}
The subset $\im(\frJ_{\cS})$ consists of graded smooth maps $\phi: \cS \rightarrow \GL(V)$ which factor through the closed embedded submanifold $(\gO(V,g),j)$, that is there exists a graded smooth map $\hat{\phi}: \cS \rightarrow \gO(V,g)$ satisfying $j \circ \hat{\phi} = \phi$. Since (\ref{eq_O}) is a pullback diagram, this is equivalent to $\phi$ fitting into the diagram
\begin{equation}
\begin{tikzcd}
\cS \arrow{r} \arrow[dashed]{d}{\phi} & \{ \ast \} \arrow{d}{e^{\times}} \\
\GL(V) \arrow{r}{\varphi} & \Sym^{\times}(V,g)
\end{tikzcd}.
\end{equation}
Since we can compose both paths along the diagram with the embedding $\dia{i}': \Sym^{\times}(V,g) \rightarrow \GL(V)$, this is equivalent to the commutativity of the diagram
\begin{equation} \label{eq_commutativityphi}
\begin{tikzcd}
\cS \arrow{r} \arrow[dashed]{d}{\phi} & \{ \ast \} \arrow{d}{e} \\
\GL(V) \arrow{r}{\varphi^{\times}} & \GL(V)
\end{tikzcd}
\end{equation}
We will now show that $\phi$ fits into this diagram, if and only if the associated automorphism $F := \Psi_{\cS}(\phi)$ of $\frM(\cS)$ satisfies the condition (\ref{eq_FOGcondition}).

We have $F(\Phi_{\lambda}) = \fF^{\kappa}{}_{\lambda} \Phi_{\kappa}$, where $\fF^{\kappa}{}_{\lambda} := \phi^{\ast}( \bby^{\kappa}{}_{\lambda})$, see (\ref{eq_PsicSmap}). Since the pullback by $\varphi^{\times}$ and the pullback by $\chi_{0}$ have the same coordinate expressions, we can use (\ref{eq_theonewewanttobeat}) to find 
\begin{equation}
(\varphi^{\times})^{\ast}(\bby^{\alpha}{}_{\lambda}) = (-1)^{|t_{\rho}| + |t_{\alpha}|(|t_{\sigma}| - |t_{\nu}| -1) + \ell(|t_{\sigma}| - |t_{\rho}| +1)} \bby^{\rho}{}_{\lambda} \fg_{\rho \sigma} \fg^{\nu \alpha} \bby^{\sigma}{}_{\nu}. 
\end{equation}
Pulling this further back by $\phi$ thus gives
\begin{equation}
(\varphi^{\times} \circ \phi)^{\ast}(\bby^{\alpha}{}_{\lambda}) = (-1)^{|t_{\rho}| + |t_{\alpha}|(|t_{\sigma}| - |t_{\nu}| -1) + \ell(|t_{\sigma}| - |t_{\rho}| +1)} \fF^{\rho}{}_{\lambda} \fg_{\rho \sigma} \fg^{\nu \alpha} \fF^{\sigma}{}_{\nu}. 
\end{equation}
On the other hand, let $e_{\cS}: \cS \rightarrow \GL(V)$ be the composition of $\cS \rightarrow \{ \ast \}$ with $e$. Then 
\begin{equation}
e_{\cS}^{\ast}( \bby^{\alpha}{}_{\lambda}) = \delta^{\alpha}_{\lambda}. 
\end{equation}
The commutativity of (\ref{eq_commutativityphi}) is therefore equivalent to the equation
\begin{equation} \label{eq_commutativityphicoords}
(-1)^{|t_{\rho}| + |t_{\alpha}|(|t_{\sigma}| - |t_{\nu}| -1) + \ell(|t_{\sigma}| - |t_{\rho}| +1)} \fF^{\rho}{}_{\lambda} \fg_{\rho \sigma} \fg^{\nu \alpha} \fF^{\sigma}{}_{\nu} = \delta^{\alpha}_{\lambda}
\end{equation}
being valid for all $\lambda,\alpha \in \{1,\dots,n\}$. Let us show that this is equivalent to the condition (\ref{eq_FOGcondition}). First, interchange the last two terms on the left-hand side to get
\begin{equation}
(-1)^{|t_{\rho}| + |t_{\alpha}| + |t_{\nu}|(|t_{\sigma}| - 1) + \ell(|t_{\nu}| - |t_{\rho}| + 1)} \fF^{\rho}{}_{\lambda} \fg_{\rho \sigma} \fF^{\sigma}{}_{\nu} \fg^{\nu \alpha} = \delta^{\alpha}_{\lambda} 
\end{equation}
Multiply both sides of the equation by $(-1)^{|t_{\alpha}|(1 + \ell)} \fg_{\alpha \kappa}$ (first without summation) to get:
\begin{equation}
\big( (-1)^{|t_{\nu}|(|t_{\sigma}|-1) + |t_{\rho}|(1 + \ell) }  \fF^{\rho}{}_{\lambda} \fg_{\rho \sigma} \fF^{\sigma}{}_{\nu} \big) (-1)^{\ell(|t_{\nu}| + |t_{\alpha}| - 1)} \fg^{\nu \alpha} \fg_{\alpha \kappa} = (-1)^{|t_{\alpha}|(1 + \ell)} \fg_{\alpha \kappa} \delta^{\alpha}_{\lambda}. 
\end{equation}
Summing over $\alpha \in \{1, \dots, n\}$ and using the analogue of (\ref{eq_ginversegmatrix}), we get 
\begin{equation}
(-1)^{|t_{\kappa}|(|t_{\sigma}|-1) + |t_{\rho}|(1 + \ell)}  \fF^{\rho}{}_{\lambda} \fg_{\rho \sigma} \fF^{\sigma}{}_{\kappa} = (-1)^{|t_{\lambda}|(1 + \ell)} \fg_{\lambda \kappa}. 
\end{equation}
But this is indeed (\ref{eq_FOGcondition}). Since we can go back to (\ref{eq_commutativityphicoords}) by multiplying both sides by $(-1)^{\ell |t_{\kappa}|} \fg^{\kappa \alpha}$ and summing over $\kappa \in \{1,\dots,n\}$, this finishes the proof. 
\end{proof}
\section{Graded symplectic group} \label{sec_gradedsymplectic}
It turns out that it is completely trivial to modify everything to include the graded skew-symmetric case. Everything is proved in \textit{completely} the same way. We will thus only recall definitions and make some statements. 

\begin{definice}
Let $V \in \gVect$. By a \textbf{skew-symmetric bilinear form of degree $\ell$}, we mean a graded bilinear map $\omega: V \times V \rightarrow \R$ satisfying
\begin{enumerate}[(i)]
\item $|\omega(v,w)| = |v| + |w| + \ell;$
\item $\omega(v,w) = - (-1)^{(|v|+\ell)(|w|+\ell)} \omega(w,v)$.
\end{enumerate}
We say that $\omega$ is a \textbf{degree $\ell$ symplectic form} on $V$, $\omega_{\flat} \in \ul{\Lin}(V,V^{\ast})$ defined by the analogue of (\ref{eq_gflattog}) is an isomorphism of graded vector spaces. 
\end{definice}
Being symmetric and skew-symmetric with respect to $\omega$ is completely the same as in Definition \ref{def_symantisym}. We just denote the two subspaces of $\gl(V)$ as $\Sym(V,\omega)$ and $\sp(V,\omega)$, instead. There is a direct sum decomposition $\gl(V) = \Sym(V,\omega) \oplus \sp(V,\omega)$ and $\sp(V,\omega)$ forms a graded Lie subalgebra of $(\gl(V),[\cdot,\cdot])$. They are $\pm 1$ eigenspaces of the map 
\begin{equation}
\tau(A) = (-1)^{\ell |A|} \omega^{-1} A^{T} \omega,
\end{equation}
which we denote by the same symbol. We can now repeat the steps of Section \ref{sec_GOVg}. 
\begin{enumerate}[(1)]
\item Construct a closed embedded submanifold $(\Sym^{\times}(V,\omega), \dia{i}')$ of $\GL(V)$.
\item Lift $\dia{\tau}$ to a unique graded smooth map $\tau^{\times}: \GL(V) \rightarrow \GL(V)$ and prove that it forms an anti-homomorphism of graded Lie groups.
\item Define $\varphi^{\times} = \mu \circ (\tau^{\times},\1_{\GL(V)})$ and show that it lifts to a graded smooth map $\varphi: \GL(V) \rightarrow \Sym^{\times}(V,\omega)$. At this second, the strange comment under (\ref{eq_almostfinalexpression}) shows that the proof of this statement goes through also in the skew-symmetric case. 
\item Lift the unit $e$ to a graded smooth map $e^{\times}: \{ \ast \} \rightarrow \Sym^{\times}(V,\omega)$ and show that $\1_{V} \in \Sym^{\times}_{0}(V,\omega)$ is a regular value of $\varphi$.
\item Define the \textbf{graded symplectic group} $\gSp(V,\omega)$ by the pullback diagram
\begin{equation}
\begin{tikzcd}
\gSp(V,\omega) \arrow[dashed]{r} \arrow[dashed]{d}{j} & \{ \ast \} \arrow{d}{e^{\times}}\\
\GL(V) \arrow{r}{\varphi} & \Sym^{\times}(V,\omega)
\end{tikzcd}
\end{equation}
Its underlying manifold is 
\begin{equation}
\gSp(V_{\bullet},\omega) = \{ A \in \GL(V_{\bullet}) \mid A^{T} \omega A = \omega \}.
\end{equation}
Finally, we obtain the following statement:
\begin{theorem}
$(\gSp(V,\omega), j)$ is a graded Lie subgroup of $\GL(V)$. Moreover, its associated graded Lie algebra can be naturally identified with the subalgebra $\sp(V,\omega) \subseteq \gl(V)$. 
\end{theorem}
\end{enumerate}
\section{Isomorphisms} \label{sec_isomorphisms}
Let $V,W \in \gVect$ and suppose there exists an isomorphism $M: V \rightarrow W$ of degree $|M|$. Then 
\begin{equation} \label{eq_eta}
\eta(A) := (-1)^{|A||M|} M A M^{-1}
\end{equation}
defines a degree zero graded linear map $\eta: \gl(V) \rightarrow \gl(W)$. 
\begin{tvrz} \label{tvrz_etainduced}
The induced graded smooth map $\dia{\eta}: \dia{\gl(V)} \rightarrow \dia{\gl(W)}$ induces a graded smooth map $\eta^{\times}: \GL(V) \rightarrow \GL(W)$ fitting into the diagram
\begin{equation}
\begin{tikzcd}
\GL(V) \arrow{d}{k} \arrow[dashed]{r}{\eta^{\times}} & \GL(W) \arrow{d}{k'} \\
\dia{\gl(V)} \arrow{r}{\dia{\eta}} & \dia{\gl(W)} 
\end{tikzcd}.
\end{equation}
This map is a graded diffeomorphism and it is an isomorphism of graded Lie groups. 
\end{tvrz}
\begin{proof}
The first claim follows from the fact that $\ul{\dia{\eta}}(A) = \eta(A)$ which clearly maps $\GL(V_{\bullet})$ into $\GL(W_{\bullet})$. $\eta^{\times}$ is a graded diffeomorphism since its inverse can be easily constructed in the same way using $M^{-1}: W \rightarrow V$. To prove that $\eta^{\times}$ is a graded Lie group morphism, we must show that the diagram
\begin{equation} \label{eq_eta'ismhomo}
\begin{tikzcd}
\GL(V) \times \GL(V) \arrow{d}{\eta^{\times} \times \eta^{\times}} \arrow{r}{\mu} & \GL(V) \arrow{d}{\eta^{\times}} \\
\GL(W) \times \GL(W) \arrow{r}{\mu'} & \GL(W)
\end{tikzcd}
\end{equation}
commutes. We denote the multiplication in $\GL(W)$ by $\mu'$. The proof of this claim is completely analogous to the one of Proposition \ref{tvrz_tauantihom}. 
\end{proof}
\begin{rem}
This observation can be viewed in a more abstract way. 

Let us write $\ul{\gVect}$ for the enriched category of graded vector spaces (over $\gVect$), that is its morphisms from $V$ to $W$ form a graded vector space $\ul{\Lin}(V,W)$ of graded linear maps of arbitrary degrees. Let $\ul{\gVect}^{\times}$ be its core groupoid (one only considers isomorphisms). Proposition \ref{tvrz_etainduced} can be then used to show that the assignment $V \mapsto \GL(V)$ defines a functor from $\ul{\gVect}^{\times}$ to the category of graded Lie groups. 
\end{rem}

\begin{example} \label{ex_degreeshift}
Let $V \in \gVect$ and let $V[m]$ be its degree shift by $m \in \Z$. In other words, one has $(V[m])_{j} := V_{j + m}$ for each $j \in \Z$. There is always a canonical graded linear map $\delta[m]: V[m] \rightarrow V$ of degree $m$, defined by $(\delta[m])_{j} := \1_{V_{j+m}} : (V[m])_{j} \rightarrow V_{j+m}$ for each $j \in \Z$. It is clearly an isomorphism of graded vector spaces of degree $m$. By the above proposition, there is an induced graded Lie group isomorphism
\begin{equation}
\eta^{\times}: \GL(V[m]) \rightarrow \GL(V). 
\end{equation}
This example shows that the general linear group is ``stable under degree shifts''. 
\end{example}

Suppose that there is a degree $\ell$ bilinear form $\beta$ on $V$ and a degree $\ell'$ bilinear form $\beta'$ on $W$. How to make a graded linear map $M: V \rightarrow W$ compatible with these? Consider the induced graded linear maps $\beta_{\flat}: V \rightarrow V^{\ast}$ and $\beta'_{\flat}: W \rightarrow W^{\ast}$ given by the analogue of (\ref{eq_gflattog}). Impose
\begin{equation} \label{eq_betabeta'OG}
M^{T} \beta'_{\flat} M = (-1)^{|M|\ell'} \beta_{\flat}.
\end{equation}
The sign is guessed according to the Koszul convention. Observe that the two sides of (\ref{eq_betabeta'OG}) do not a priori have the same degree. In order for the equation to make sense, we have to assume that 
\begin{equation} \label{eq_ellell'Mrelation}
\ell' + 2|M| = \ell. 
\end{equation}
In particular, note that $\ell$ and $\ell'$ must have the same \textit{parity}. In terms of $\beta$ and $\beta'$, the condition (\ref{eq_betabeta'OG}) takes the following form for all $v,w \in V$:
\begin{equation} \label{eq_betabeta'OG2}
\beta'(M(v),M(w)) = (-1)^{|M|(|v|+\ell+1)} \beta(v,w).
\end{equation}
We now examine the consistency of this relation with a (skew-)symmetry of the forms $\beta$ and $\beta'$. 
\begin{lemma}
Let $\beta$ and $\beta'$ be bilinear forms as in the above paragraph. Suppose (\ref{eq_ellell'Mrelation}) and (\ref{eq_betabeta'OG2}) are true. Then we get the following statements:
\begin{enumerate}[(i)]
\item Suppose that $\beta$ is symmetric. Then for even $|M|$ the form $\beta'$ is symmetric, and for odd $|M|$ the form $\beta'$ is skew-symmetric. 

In particular, if $\beta$ is a metric, then for even $|M|$ the form $\beta'$ is a metric, and for odd $|M|$ the form $\beta'$ is a symplectic form.
\item Suppose that $\beta$ is skew-symmetric. Then for even $|M|$ the form $\beta'$ is skew-symmetric, and for odd $|M|$ the form $\beta'$ is symmetric. 

In particular, if $\beta$ is a symplectic form, then for even $|M|$ the form $\beta'$ is a symplectic form, and for odd $|M|$ the form $\beta'$ is a metric. 
\end{enumerate}
\end{lemma}
\begin{proof}
Let us only show $(i)$. Suppose $\beta$ is symmetric. Then 
\begin{equation}
\begin{split}
\beta'(M(v),M(w)) = & \ (-1)^{|M|(|v|+\ell+1)} \beta(v,w) = (-1)^{|M|(|v|+\ell+1) + (|v|+\ell)(|w|+\ell)} \beta(w,v) \\
= & \ (-1)^{|M|(|v| + |w|) + (|v|+\ell)(|w|+\ell)} \beta'(M(w),M(v)) \\
= & \ (-1)^{|M| + (|M(v)| + \ell')(|M(w)|+\ell')} \beta'(M(w),M(v)),
\end{split}
\end{equation}
for all $v,w \in V$. Since $M$ is an isomorphism, the claim $(i)$ follows. 
\end{proof}
\begin{rem}
The feature of symmetry and skew-symmetry being interchanged under odd-degree isomorphisms is common in any graded geometry. Note that the sign $(-1)^{|M|}$ popping out in the above proof cannot be killed by any choice of signs in (\ref{eq_betabeta'OG2}). The sign hydra always wins. 
\end{rem}

\begin{tvrz} \label{tvrz_isomorphisms}
Let $g$ be a degree $\ell$ metric on $V$, $g'$ a degree $\ell'$ metric on $W$. Let $\omega$ be a degree $\ell$ symplectic form on $V$, $\omega'$ a degree $\ell'$ symplectic form on $W$. Suppose that an isomorphism $M: V \rightarrow W$ satisfies the degree constraint (\ref{eq_ellell'Mrelation}). 

\begin{enumerate}[(i)]
\item Let $|M|$ be even, and suppose that $\beta = g$ and $\beta' = g'$ satisfy (\ref{eq_betabeta'OG2}). Then $\eta^{\times}: \GL(V) \rightarrow \GL(W)$ constructed in Proposition \ref{tvrz_etainduced} induces a unique graded Lie group isomorphism $\eta': \gO(V,g) \rightarrow \gO(W,g')$ fitting into the commutative diagram
\begin{equation} \label{eq_etaprimediagram}
\begin{tikzcd}
\gO(V,g) \arrow[dashed]{r}{\eta'} \arrow{d}{j} & \gO(W,g') \arrow{d}{j'}\\
\GL(V) \arrow{r}{\eta^{\times}} & \GL(W)
\end{tikzcd}
\end{equation}
\item Let $|M|$ be even and suppose that $\beta = \omega$ and $\beta' = \omega'$ satisfy (\ref{eq_betabeta'OG2}). Then there is a unique isomorphism $\eta': \gSp(V,\omega) \rightarrow \gSp(W,\omega')$ fitting into the analogue of (\ref{eq_etaprimediagram}). 
\item Let $|M|$ be odd and suppose that $\beta = g$ and $\beta' = \omega'$ satisfy (\ref{eq_betabeta'OG2}). Then there is a unique isomorphism $\eta': \gO(V,g) \rightarrow \gSp(W,\omega')$ fitting into the analogue of (\ref{eq_etaprimediagram}).
\item Let $|M|$ be odd and suppose that $\beta = \omega$ and $\beta' = g'$ satisfy (\ref{eq_betabeta'OG2}). Then there is a unique isomorphism $\eta': \gSp(V,\omega) \rightarrow \gO(W,g')$ fitting into the analogue of (\ref{eq_etaprimediagram}). 
\end{enumerate}
\end{tvrz}
\begin{proof}
It suffices to prove $(i)$. Since we are constructing a map into the embedded submanifold $\gO(W,g')$ defined by the pullback diagram (\ref{eq_O}), it suffices to prove the commutativity of the diagram
\begin{equation}
\begin{tikzcd}
\gO(V,g) \arrow{d}{\eta^{\times} \circ j} \arrow{r} & \{ \ast \} \arrow{d}{e'^{\times}} \\
\GL(W) \arrow{r}{\varphi'} & \Sym^{\times}(W,g') 
\end{tikzcd},
\end{equation}
where all primed maps are the counterparts we have used to construct $\gO(V,g)$, except now in the setting of $\gO(W,g')$. Composing the diagram with the embedding of $\Sym^{\times}(W,g')$ into $\GL(W)$, this is equivalent to the commutativity of 
\begin{equation}
\begin{tikzcd} \label{eq_diagramtoconstructetaprime}
\gO(V,g) \arrow{d}{\eta^{\times} \circ j} \arrow{r} & \{ \ast \} \arrow{d}{e'} \\
\GL(W) \arrow{r}{\varphi'^{\times}} & \GL(W)
\end{tikzcd},
\end{equation}
where $e': \{\ast\} \rightarrow \GL(W)$ is the unit and $\varphi'^{\times} = \mu' \circ (\tau'^{\times},\1_{\GL(W)})$. To see this, first note that
\begin{equation} \label{eq_etaintertwinestaus}
\begin{tikzcd}
\GL(V) \arrow{d}{\tau^{\times}} \arrow{r}{\eta^{\times}} & \GL(W) \arrow{d}{\tau'^{\times}} \\
\GL(V) \arrow{r}{\eta^{\times}} & \GL(W)
\end{tikzcd}
\end{equation}
commutes. Indeed, observe that for any $A \in \gl(V)$, one has 
\begin{equation}
\begin{split}
(\eta \circ \tau)(A) = & \  \eta( (-1)^{\ell |A|} g^{-1} A^{T} g) = (-1)^{\ell|A|} M (g^{-1} A^{T} g) M^{-1} \\
= & \ (-1)^{\ell|A|} (Mg^{-1}M^{T}) ((M^{-1}){}^{T} A^{T} M^{T}) ((M^{T}){}^{-1} g M^{-1}) \\
= & \ (-1)^{\ell'|A|} g'^{-1} \eta(A)^{T} g' = (\tau' \circ \eta)(A).
\end{split}
\end{equation}
Note that $|M|$ is even and can be omitted from all formulas. Hence $\eta \circ \tau = \tau' \circ \eta$ and the commutativity (\ref{eq_etaintertwinestaus}) follows by similar means to those in the proof of Proposition \ref{tvrz_tauantihom}. We can now prove the commutativity of (\ref{eq_diagramtoconstructetaprime}). Indeed, one has 
\begin{equation}
\begin{split}
\varphi'^{\times} \circ (\eta^{\times} \circ j) = & \ \mu' \circ (\tau'^{\times}, \1_{\GL(W)}) \circ \eta^{\times} \circ j = \mu' \circ (\tau'^{\times} \circ \eta^{\times}, \eta^{\times}) \circ j  \\
= & \ \mu' \circ (\eta^{\times} \circ \tau^{\times}, \eta^{\times} ) \circ j = \mu' \circ (\eta^{\times} \times \eta^{\times}) \circ (\tau^{\times}, \1_{\GL(V)}) \circ j \\
= & \ \eta^{\times} \circ \mu \circ (\tau^{\times}, \1_{\GL(V)}) \circ j = \eta^{\times} \circ (\varphi^{\times} \circ j) \\
= & \ \eta^{\times} \circ e_{\gO(V,g)} = e'_{\gO(V,g)}. 
\end{split}
\end{equation}
We have used the commutativity of (\ref{eq_etaintertwinestaus}), the fact that (\ref{eq_eta'ismhomo}) commutes, and the definition of $\gO(V,g)$. In the last step, we have used the fact that $\eta^{\times}$ preserves units. Since $e'_{\gO(V,g)}$ is a composition of the terminal arrow $\gO(V,g) \rightarrow \{\ast \}$ with the unit $e'$, the commutativity of (\ref{eq_diagramtoconstructetaprime}) follows. By the universal property of pullbacks, there now exists a unique graded smooth map $\eta': \gO(V,g) \rightarrow \gO(W,g')$ fitting into the diagram (\ref{eq_eta'ismhomo}). 

The fact that $\eta'$ is a graded Lie group isomorphism follows from its definition and the fact that $\eta^{\times}$ is a graded Lie group isomorphism. The proof is very similar to the one of Proposition \ref{tvrz_subgroupisgroup} and we leave it as an easy exercise. The proof of $(ii)-(iv)$ is analogous. Let us only warn the reader that for odd $|M|$, odd things may happen. For example, one has $(M^{T})^{-1} = -(M^{-1})^{T}$, see (\ref{eq_comptranspose}). 
\end{proof}
\begin{rem}
One can consider a category $\mathbf{qgVect}$ of \textbf{quadratic graded vector spaces}. Its objects are pairs $(V,g)$ consisting of $V \in \gVect$ and a degree $\ell$ metric $g$ (for some $\ell \in \Z$). Its arrows are degree \textit{zero} isomorphisms satisfying (\ref{eq_ellell'Mrelation}) and (\ref{eq_betabeta'OG2}). Proposition \ref{tvrz_isomorphisms}-$(i)$ then proves that the assignment $(V,g) \mapsto \gO(V,g)$ defines a functor from $\mathbf{qgVect}$ into the category of graded Lie groups. The same remark applies also to the category of symplectic graded vector spaces and the assignment $(V,\omega) \mapsto \gSp(V,\omega)$. 
\end{rem}
\begin{example}
Let $\beta: V \times V \rightarrow \R$ be a degree $\ell$ bilinear form on $V \in \gVect$. Let $m \in \Z$ be arbitrary. Recall that there is a degree shifting operator $\delta[m]: V[m] \rightarrow V$, see Example \ref{ex_degreeshift}. Define a bilinear form $\beta[m]: V[m] \times V[m] \rightarrow \R$ by the formula
\begin{equation}
\beta[m]_{\flat} := (-1)^{m\ell} \delta[m]^{T} \beta_{\flat} \delta[m]. 
\end{equation}
Note that its degree is $\ell + 2m$. We call $\beta[m]$ the \textbf{degree shift of $\beta$ by $m \in \Z$}. 

\begin{enumerate}[(i)]
\item If $m$ is even, and $\beta = g$ is a degree $\ell$ metric, then $g[m]$ is a degree $\ell + 2m$ metric and $\delta[m]$ induces a graded Lie group isomorphism $\eta': \gO(V[m],g[m]) \rightarrow \gO(V,g)$. 
\item If $m$ is even, and $\beta = \omega$ is a degree $\ell$ symplectic form, then $\omega[m]$ is a degree $\ell + 2m$ symplectic form and $\delta[m]$ induces an isomorphism $\eta': \gSp(V[m],\omega[m]) \rightarrow \gSp(V,\omega)$.
\item If $m$ is odd, and $\beta = g$ is a degree $\ell$ metric, then $g[m]$ is a degree $\ell + 2m$ symplectic form and $\delta[m]$ induces an isomorphism $\eta': \gSp(V[m],g[m]) \rightarrow \gO(V,g)$.
\item If $m$ is odd, and $\beta = \omega$ is a degree $\ell$ symplectic form, then $\omega[m]$ is a degree $\ell + 2m$ metric and $\delta[m]$ induces an isomorphism $\eta': \gO(V[m],\omega[m]) \rightarrow \gSp(V,\omega)$. 
\end{enumerate}
\end{example}
\begin{example}
Let $\beta: V \times V \rightarrow \R$ be a degree $\ell$ bilinear form on $V$, such that $\beta_{\flat}: V \rightarrow V^{\ast}$ is an isomorphism. Recall that there is a canonical isomorphism $\chi: V \rightarrow (V^{\ast})^{\ast}$, given by
\begin{equation}
[\chi(v)](\xi) = (-1)^{|v||\xi|} \xi(v),
\end{equation}
for all $v \in V$ and $\xi \in V^{\ast}$. Let us define a degree $-\ell$ bilinear form $\beta'$ on $V^{\ast}$ by declaring 
\begin{equation}
\beta'_{\flat} := \chi \circ \beta_{\flat}^{-1}. 
\end{equation}
More explicitly, for all $\xi,\eta \in V^{\ast}$, one finds
\begin{equation}
\begin{split}
\beta'(\xi,\eta) = & \ (-1)^{-\ell(|\xi|-\ell)} [\beta'_{\flat}(\xi)](\eta) = (-1)^{-\ell(|\xi|-\ell)} [\chi( \beta_{\flat}^{-1}(\xi))](\eta) \\
= & \ (-1)^{(|\xi|-\ell)(|\eta|-\ell)} \eta( \beta_{\flat}^{-1}(\xi)). 
\end{split}
\end{equation}
If there is no risk of confusion, we write $\beta' = \beta^{-1}$. 

\begin{enumerate}[(i)]
\item Let us first assume that $\beta$ is symmetric, that is a degree $\ell$ metric on $V$. It follows from (\ref{eq_gradedsymmetrygflat}) easily that for all $\xi,\eta \in V^{\ast}$, one has 
\begin{equation}
\eta(\beta_{\flat}^{-1}(\xi)) = (-1)^{(|\xi|-\ell)(|\eta|-\ell) + \ell} \xi(\beta^{-1}_{\flat}(\eta)). 
\end{equation}
Consequently, one finds  that for all $\xi, \eta \in V^{\ast}$, one has 
\begin{equation}
\beta^{-1}(\xi,\eta) = (-1)^{(|\xi|+\ell)(|\eta| + \ell) + \ell} \beta^{-1}(\eta,\xi). 
\end{equation}
We have just shown that for even $\ell$, $\beta^{-1}$ is a metric, and for odd $\ell$, $\beta^{-1}$ is a \textit{symplectic form}. Moreover, in both cases, one finds
\begin{equation} \label{eq_betabetabeta}
\beta^{-1}(\beta_{\flat}(v), \beta_{\flat}(w)) = (-1)^{\ell} [\beta_{\flat}(v)](w) = (-1)^{\ell(|v| + \ell + 1)} \beta(v,w). 
\end{equation}
But this means that the two bilinear forms $\beta$ and $\beta^{-1}$, together with a degree $\ell$ isomorphism $\beta_{\flat}: V \rightarrow V^{\ast}$, fit into (\ref{eq_ellell'Mrelation}) and (\ref{eq_betabeta'OG2}). Proposition \ref{tvrz_isomorphisms} now implies that for even $\ell$, there is a graded Lie group isomorphism 
\begin{equation}
\eta': \gO(V,\beta) \rightarrow \gO(V^{\ast},\beta^{-1}),
\end{equation}
and for odd $\ell$, there is a graded Lie group isomorphism 
\begin{equation}
\eta': \gO(V,\beta) \rightarrow \gSp(V^{\ast},\beta^{-1}).
\end{equation}
\item If $\beta$ is skew-symmetric, that is a degree $\ell$ symplectic form on $V$, one uses the same calculation to prove that $\beta^{-1}$ is a symplectic form, and for odd $\ell$, $\beta^{-1}$ is a \textit{metric}. However due to the extra sign, the analogue of (\ref{eq_betabetabeta}) now takes the form
\begin{equation} \label{eq_betabetabeta2}
\beta^{-1}(\beta_{\flat}(v), \beta_{\flat}(w)) = - (-1)^{\ell(|v|+\ell+1)} \beta(v,w). 
\end{equation}
Proposition \ref{tvrz_isomorphisms} thus implies that for even $\ell$, there is a graded Lie group isomorphism
\begin{equation}
\eta': \gSp(V,\beta) \rightarrow \gSp(V^{\ast}, -\beta^{-1}),
\end{equation}
and for odd $\ell$, there is a graded Lie group isomorphism
\begin{equation}
\eta': \gSp(V,\beta) \rightarrow \gO(V^{\ast}, -\beta^{-1}). 
\end{equation}
\end{enumerate}
\end{example}
\begin{example}
The extra signs in the symplectic case of the previous example are in fact completely irrelevant. First, notice that the statements of Proposition \ref{tvrz_isomorphisms} remain true even if we replace the condition (\ref{eq_betabeta'OG2}) with the ``anti'' condition
\begin{equation} \label{eq_betabetaantiOG}
\beta'(M(v),M(w)) = - (-1)^{|M|(|v|+\ell+1)} \beta(v,w). 
\end{equation}
Indeed, this condition was actually used only to prove (\ref{eq_etaintertwinestaus}), and it is easy to see that the extra sign does not change anything. In particular, we can always view the identity $\1_{V}: V \rightarrow V$ as a degree zero isomorphism fitting into (\ref{eq_betabetaantiOG}) for the bilinear forms $\beta$ and $\beta' = -\beta$. It now follows from our observation that $\gO(V,\beta) \cong \gO(V,-\beta)$ and $\gSp(V,\beta) \cong \gSp(V,-\beta)$, respectively. 

Note that this is precisely the canonical isomorphism $\gO(p,q) \cong \gO(q,p)$ in ordinary geometry. 
\end{example}
\section{Examples and applications} \label{sec_examples}
\subsection{Standard form of a metric} \label{subsec_standard}
Let us start this section by examining the standard form of a degree $\ell$ metric on a \textit{non-zero} graded vector space $V$. As already noted in Remark \ref{rem_metric}-(iii), the graded dimension $\gdim(V) = (r_{j})_{j \in \Z}$ has to satisfy $r_{j} = r_{-(j+\ell)}$ for all $j \in \Z$. To better understand this condition, let 
\begin{equation} \label{eq_kepsilon} k := - \lfloor \frac{\ell}{2} \rfloor, \; \; \epsilon := \ell \modu{2}. \end{equation}
We can thus write $\ell = -2k + \epsilon$ and rewrite the above restriction of the graded dimension as
\begin{equation} \label{eq_grkconditionfinal}
r_{k + i} = r_{k - (i+\epsilon)},
\end{equation}
for all $i \in \Z$. In fact, it suffices to consider $i \in \N_{0}$. Moreover, there is $i_{\bullet} \in \N_{0}$, such that $r_{k+i_{\bullet}} \neq 0$ and $r_{k+i} = 0$ for all $i > i_{\bullet}$. If one visualizes $\gdim(V)$ as a function from integers to non-negative integers, this means that its graph is symmetric around the axis passing through $-\ell / 2$ and non-zero only in the interval $[k - (i_{\bullet} + \epsilon), k + i_{\bullet}]$. Finally, note that the total dimension $n$ of $V$ is 
\begin{equation}
n = (1+\epsilon)r_{k} + 2 \sum_{i = 1}^{i_{\bullet}} r_{k+i}. 
\end{equation} 
We will now show that one can always choose a convenient total basis for $V$. 
\begin{tvrz} \label{tvrz_standardformmetric}
Let $V \in \gVect$ be non-trivial, equipped with a degree $\ell$ metric $g$. Let $k \in \Z$ and $\epsilon \in \{0,1\}$ be defined as above. 
\begin{enumerate}[(i)]
\item If $\epsilon = 1$ and $k$ is arbitrary, or $\epsilon = 0$ and $k$ is odd, there exists a total basis $(t_{\alpha})_{\alpha=1}^{n/2} \cup (\ol{t}_{\beta})_{\beta=1}^{n/2}$ for $V$, such that 
\begin{equation}
|t_{\alpha}| \geq k, \; \; |\ol{t}_{\beta}| = -|t_{\beta}| + 2k - \epsilon,
\end{equation}
and the metric takes the form
\begin{equation} \label{eq_ginONbasisexpl}
g(t_{\alpha}, \ol{t}_{\beta}) = \delta_{\alpha \beta},
\end{equation}
with other combinations either trivial or obtained by the symmetry.
\item If $\epsilon = 0$ and $k$ is even, there exists a total basis $(t_{\alpha})_{\alpha=1}^{m/2} \cup (\ol{t}_{\beta})_{\beta=1}^{m/2} \cup (e_{a})_{a=1}^{r_{k}}$ for $V$, where
\begin{equation}
|t_{\alpha}| > k, \; |\ol{t}_{\beta}| = -|t_{\beta}| + 2k, \; \; |e_{a}| = k,
\end{equation} 
and the metric takes the form 
\begin{equation}
g(t_{\alpha},\ol{t}_{\beta}) = \delta_{\alpha \beta}, \; \; g(e_{a},e_{b}) = \eta_{ab},
\end{equation}
where $\eta_{ab}$ is a Minkowski metric of some signature $(p,q)$ with $p + q = r_{k}$, $m = n - r_{k}$, and other combinations are either trivial or obtained by the symmetry.
\end{enumerate}
\end{tvrz}
\begin{proof}
Let us prove $(i)$ first. Suppose that there exists a non-zero vector $t \in V$ satisfying $|t| > k$. Since $g_{\flat}(t) \neq 0$, there exist a vector $\ol{t} \in V$ satisfying $[g_{\flat}(t)](\ol{t}) = (-1)^{\epsilon(|t| + 1)}$. But then
\begin{equation}
g(t,\ol{t}) = (-1)^{\epsilon(|t|+1)} [g_{\flat}(t)](\ol{t}) = 1. 
\end{equation}
Note that $|\ol{t}| = -(|t| + \ell) = -|t| + 2k - \epsilon$. Since $|\ol{t}| < k - \epsilon < |t|$, the collection $\{t,\ol{t}\}$ is linearly independent, and one also finds that
\begin{equation}
g(t,t) = 0, \; \; g(\ol{t}, \ol{t}) = 0,
\end{equation}
for degree reasons. However, it may happen that there is no non-zero vector $t \in V$ satisfying $|t| > k$. Thanks to (\ref{eq_grkconditionfinal}) and since $V \neq 0$, there is a non-zero $t$ with $|t| = k$. For $\epsilon = 1$, we are safe as $|\ol{t}| = k - 1 < |t|$ and we can proceed as above. If $\epsilon = 0$ and $k$ is odd, we have $g_{\flat}(V_{k}) = (V_{k})^{\ast}$ and $g$ restricts to a non-degenerate bilinear form $g_{k}: V_{k} \times V_{k} \rightarrow \R$. For any $v,w \in V_{k}$, one has 
\begin{equation}
g_{k}(v,w) = -g_{k}(w,v),
\end{equation}
that is $g_{k}$ is a symplectic form on $V_{k}$. We can thus choose a vector $\ol{t} \in V_{k}$ so that 
\begin{equation}
g(t,\ol{t}) = g_{k}(t,\ol{t}) = 1. 
\end{equation}
The skew-symmetry ensures that $g(t,t) = 0$ and $g(\ol{t},\ol{t}) = 0$. Note that this immediately implies that the collection $\{t, \ol{t} \}$ is linearly independent. 

In either case, we can now consider $L = \R \{t, \ol{t} \}$ and it follows from the construction that $V = L \oplus L^{\perp}$. In particular, the restriction $g|_{L^{\perp}}$ is a degree $\ell$ metric on $L^{\perp}$. By replacing $(V,g)$ with $(L^{\perp}, g|_{L^{\perp}})$, we can now iterate the procedure. After finitely many steps, we run out of vectors and obtain a sequence $\{ t_{\alpha}, \ol{t}_{\alpha} \}_{\alpha =1}^{n/2}$ of ``conjugated pairs''. By construction, they form a total basis $(t_{\alpha})_{\alpha =1}^{n/2} \cup (\ol{t}_{\beta})_{\beta=1}^{n/2}$ having the required properties. This proves $(i)$. 

Let $\epsilon = 0$ and $k$ be even. Whenever we can find $|t| > k$, we proceed in the same way as above. Indeed, in this case we can choose $\ol{t} \in V$ as before, define $L = \R \{t,\ol{t}\}$ and proceed. If there are no such vectors, we observe that that $g_{\flat}(V_{k}) = (V_{k})^{\ast}$ and $g$ restricts to a \textit{metric} $g_{k}: V_{k} \times V_{k} \rightarrow \R$. We can thus choose $t \in V_{k}$ to satisfy $g(t,t) = g_{k}(t,t) = \pm 1$. In this case, choose $L = \R \{ t \}$. We get $V = L \oplus L^{\perp}$, replace $(V,g)$ with $(L^{\perp}, g|_{L^{\perp}})$ and repeat. In this way, we obtain a sequence of ``conjugated pairs'' $\{ t_{\alpha}, \ol{t}_{\alpha} \}_{\alpha=1}^{m/2}$, where $m = n - r_{k}$, and a sequence $\{ e_{a} \}_{a =1}^{r_{k}}$ satisfying $|e_{a}| = k$ and $|g(e_{a},e_{b})| = \delta_{ab}$. By possibly reordering these, we obtain a total basis $(t_{\alpha})_{\alpha=1}^{m/2} \cup (\ol{t}_{\beta})_{\beta=1}^{m/2} \cup (e_{a})_{a=1}^{r_{k}}$ having the required properties. 
\end{proof}

\begin{rem}
The above proposition shows that for $\ell \modu{4} \neq 0$, there is a unique isomorphism class of degree $\ell$ metrics on a given $V \in \gVect$. If $\ell \modu{4} = 0$, the isomorphism classes of degree $\ell$ metrics are uniquely classified by the signature $(p,q)$ of the metric $g_{k}: V_{k} \times V_{k} \rightarrow \R$. 
\end{rem}

\subsection{Quadratic graded vector bundles} \label{subsec_OVG}
Now, let us quickly recall the notion of a \textbf{graded vector bundle $\E$} over $\M$ of a graded rank $(r_{j})_{j \in \Z}$. It is defined by its sheaf of sections $\Gamma_{\E}$, which is assumed to be a locally freely and finitely generated sheaf of $\C^{\infty}_{\M}$-modules of a constant graded rank $(r_{j})_{j \in \Z}$. In other words, for each $m \in M$, there is $U \in \Op_{m}(M)$ and a \textbf{local frame} $(\Phi_{\lambda})_{\lambda=1}^{n}$ for $\E$ over $U$, where 
\begin{enumerate}[(i)]
\item $\Phi_{\lambda} \in \Gamma_{\E}(U)$, $r_{j} = \# \{ \lambda \in \{1,\dots,n\} \mid |\Phi_{\lambda}| = j \}$. In particular, $n = \sum_{j \in \Z} r_{j}$. 
\item For every $V \in \Op(U)$, every $\psi \in \Gamma_{\E}(V)$ can be written as $\psi = \psi^{\lambda} \cdot \Phi_{\lambda}|_{V}$ for unique functions $\psi^{\lambda} \in \C^{\infty}_{\M}(V)$ satisfying $|f^{\lambda}| + |\Phi_{\lambda}| = |\psi|$. 
\end{enumerate}
For detailed exposition and examples, see \cite{vsmolka2025threefold}. We say that an $\R$-bilinear map $\<\cdot,\cdot\>: \Gamma_{\E}(M) \times \Gamma_{\E}(M) \rightarrow \C^{\infty}_{\M}(M)$ is a \textbf{degree $\ell$ fiber-wise metric on $\E$}, if 
\begin{enumerate}[(i)]
\item $|\<\psi,\psi'\>| = |\psi| + |\psi'| + \ell$. 
\item $\<\psi,\psi'\> = (-1)^{(|\psi|+\ell)(|\psi'|+\ell)} \<\psi',\psi\>$.
\item $\<\psi, f\psi'\> = (-1)^{|f|(|\psi|+\ell)} f \<\psi,\psi'\>$. 
\item $g_{\E}: \Gamma_{\E}(M) \rightarrow \Gamma_{\E^{\ast}}(M)$, defined by $[g_{\E}(\psi)](\psi') = (-1)^{\ell(|\psi|+\ell)} \<\psi,\psi'\>$, is an isomorphism.
\end{enumerate}
All conditions are assumed to hold for all $\psi,\psi' \in \Gamma_{\E}(M)$ and $f \in \C^{\infty}_{\M}(M)$. We say that $(\E, \<\cdot,\cdot\>)$ is a \textbf{quadratic graded vector bundle of degree $\ell$}. This notion was introduced in \cite{vysoky2022graded}. It can be shown that $\< \cdot,\cdot\>$ restricts to a degree $\ell$ metric $\< \cdot,\cdot\>_{(m)}$ on each fiber $\E_{(m)}$. In particular, the graded rank $(r_{j})_{j \in \Z}$ of $\E$ has to satisfy the same constraint as discussed in Subsection \ref{subsec_standard}. We can thus define $k \in \Z$ and $\epsilon \in \{0,1\}$ as in (\ref{eq_kepsilon}) to find the condition (\ref{eq_grkconditionfinal}) for all $i \in \Z$. There is a following analogue of Proposition \ref{tvrz_standardformmetric}:

\begin{tvrz} \label{tvrz_standardforQGVB}
Let $(\E, \<\cdot,\cdot\>)$ be a non-trivial quadratic graded vector bundle of degree $\ell$. Let $k \in \Z$ and $\epsilon \in \{0,1\}$ be defined as above. Let $m \in M$ be arbitrary. 
\begin{enumerate}[(i)]
\item If $\epsilon = 1$ and $k$ is arbitrary, or $\epsilon = 0$ and $k$ is odd, there exists a local frame $(\Phi_{\alpha})_{\alpha=1}^{n/1} \cup \{ \ol{\Phi}_{\beta})_{\beta=1}^{n/2}$ for $\E$ over some $U \in \Op_{m}(M)$, such that 
\begin{equation}
|\Phi_{\alpha}| \geq k, \; \; |\ol{\Phi}_{\beta}| = -|\Phi_{\beta}| + 2k - \epsilon,
\end{equation}
and the fiber-wise metric takes the form
\begin{equation} \label{eq_ONframe1}
\<\Phi_{\alpha}, \ol{\Phi}_{\beta}\> = \delta_{\alpha \beta},
\end{equation}
with other combinations either trivial or obtained by the symmetry. 
\item If $\epsilon = 0$ and $k$ is even, there exists a local frame $(\Phi_{\alpha})_{\alpha=1}^{m/2} \cup (\ol{\Phi}_{\beta})_{\beta=1}^{m/2} \cup (\Xi_{a})_{a=1}^{r_{k}}$ for $\E$ over some $U \in \Op_{m}(M)$, where
\begin{equation}
|\Phi_{\alpha}| > k, \; |\ol{\Phi}_{\beta}| = -|\Phi_{\beta}| + 2k, \; \; |\Xi_{a}| = k,
\end{equation} 
and the fiber-wise metric takes the form 
\begin{equation}
\<\Phi_{\alpha},\ol{\Phi}_{\beta}\> = \delta_{\alpha \beta}, \; \; \< \Xi_{a}, \Xi_{b}\> = \eta_{ab},
\end{equation}
were $\eta_{ab}$ is a Minkowski metric of some signature $(p,q)$, $m = n - r_{k}$, and other combinations are either trivial or obtained by the symmetry. 
\end{enumerate}
\end{tvrz}
\begin{proof}
The proof is similar to the one of Proposition \ref{tvrz_standardformmetric}. One uses the fact that independent vectors of the fiber $\E_{(m)}$ can be extended to a local frame on some $U \in \Op_{m}(M)$. This is Proposition 3.12 in \cite{vsmolka2025threefold}. Note that in the case $\epsilon = 0$ and $k$ even, one has to use a square root of a function $f \in \C^{\infty}_{\M}(M)$ to normalize sections in $\Gamma_{\E}(U)$. This is always possible if $f(m) > 0$ for all $m \in U$. We leave this as an interesting graded geometry exercise. 
\end{proof}

Now, for example, suppose that $\epsilon  = 1$ and $k$ is arbitrary, or $\epsilon = 0$ and $k$ is odd. In other words, the option $(i)$ of the above proposition happens. We can thus find a local frame $(\Phi_{\alpha})_{\alpha=1}^{n/2} \cup (\ol{\Phi}_{\beta})_{\beta=1}^{n/2}$ for $\E$ over $U$, having the described properties. 

Let $K \in \gVect$ be the typical fiber of $\E$. We know that it has a total basis $(t_{\alpha})_{\alpha=1}^{n/2} \cup (\ol{t}_{\beta})_{\beta=1}^{n/2}$ satisfying $|t_{\alpha}| = |\Phi_{\alpha}|$ and $|\ol{t}_{\beta}| = |\ol{\Phi}_{\beta}|$ for all $\alpha,\beta \in \{1,\dots,n/2\}$, and 
\begin{equation} \label{eq_localtrivialization}
\Gamma_{\E}(U) \cong \C^{\infty}_{\M}(U) \otimes_{\R} K,
\end{equation} 
where $\Phi_{\alpha}$ are mapped to $1 \otimes t_{\alpha}$ and $\ol{\Phi}_{\beta}$ are mapped to $1 \otimes t_{\beta}$. The formula 
\begin{equation}
g(t_{\alpha},\ol{t}_{\beta}) := \delta_{\alpha \beta}
\end{equation}
then defines a degree $\ell$ metric $g$ on $K$, where other combinations are trivial or obtained by the graded symmetry. This follows from Proposition \ref{tvrz_standardformmetric}. In particular, the equation (\ref{eq_ONframe1}) shows that under identifications induced by the local trivialization, the fiber-wise metric $\<\cdot,\cdot\>$ corresponds to the $\C^{\infty}_{\M}(U)$-bilinear form $\<\cdot,\cdot\>_{g}$ defined in Section \ref{sec_OGFOP}. Consequently, we find that the group of vector bundle automorphisms of $\E_{U}$ preserving $\<\cdot,\cdot\>$ can be identified with the group
\begin{equation}
\frP'(\M|_{U}) = \gMan^{\infty}( \M|_{U}, \gO(K,g)). 
\end{equation}
This follows immediately from the above observations, together with Proposition \ref{tvrz_FOPOG}. Completely the same statement is obtained for $\epsilon = 0$ and $k$ even. In other words, we see that $\gO(K,g)$ can be viewed as a ``\textit{structure group}'' for quadratic graded vector bundles.  Similar relation can be found between the graded symplectic group and symplectic graded vector bundles. 
\subsection{Underlying Lie groups} \label{subsec_ULG}
In the construction of the graded Lie group $\gO(V,g)$, we have argued that its underlying Lie group is $\gO(V_{\bullet},g)$ defined as 
\begin{equation}
\gO(V_{\bullet},g) = \{ A \in \GL(V_{\bullet}) \mid A^{T} g A = g \},
\end{equation}
see (\ref{eq_gOVbulletg}). Using Proposition \ref{tvrz_standardformmetric}, we can now see more clearly what is going on. Recall that $A = (A_{j})_{j \in \Z}$ for $A_{j} \in \Lin(V_{j},V_{j})$, and $g = (g_{j})_{j \in \Z}$, for $g_{j} \in \Lin(V_{j}, (V_{-(j+\ell)})^{\ast})$. We thus require
\begin{equation} \label{eq_AOVgdegreewise}
(A_{-(j+\ell)})^{T} g_{j} A_{j} = g_{j},
\end{equation}
for all $j \in \Z$. Let $k \in \Z$ and $\epsilon \in \{0,1\}$, and $i_{\bullet} \in \N_{0}$ be defined as in Subsection \ref{subsec_standard}. (\ref{eq_AOVgdegreewise}) reads
\begin{equation} \label{eq_AOVgdegreewise2}
(A_{k - (i+\epsilon)})^{T} g_{k+i} A_{k+i} = g_{k+i},
\end{equation}
for all $i \in \Z$. We will now discuss the situation based on the parity of $\epsilon$ and $k$. 
\begin{enumerate}[(i)]
\item Let $\epsilon = 1$ and let $k$ be arbitrary. Let $(t_{\alpha})_{\alpha=1}^{n/2} \cup (\ol{t}_{\beta})_{\beta=1}^{n/2}$ be a basis obtained as in Proposition \ref{tvrz_standardformmetric}-$(i)$. Let us further order and relabel this basis so that 
\begin{equation} \label{eq_basissubdivision}
(t_{\alpha})_{\alpha=1}^{n/2} = \cup_{i=0}^{i_{\bullet}} (t^{(i)}_{\alpha_{i}})_{\alpha_{i}=1}^{r_{k+i}}, \; \; (\ol{t}_{\beta})_{\beta=1}^{n/2} = \cup_{i=0}^{i_{\bullet}} (\ol{t}{}^{(i)}_{\beta_{i}})_{\beta_{i}=1}^{r_{k+i}},
\end{equation} 
where $(t_{\alpha_{i}}^{(i)})_{\alpha_{i}=1}^{r_{k+i}}$ is a basis for $V_{k+i}$ and $(\ol{t}^{(i)}_{\beta_{i}})_{\beta_{i}=1}^{r_{k+i}}$ is a basis for $V_{k - (i+1)}$, for each $i \in \{0,\dots,i_{\bullet}\}$. Finally, for each $i \in \{0,\dots,i_{\bullet}\}$, write 
\begin{equation} \label{eq_Amatrices}
A_{k+i}(t^{(i)}_{\alpha_{i}}) = [\bbA_{(i)}]{}^{\beta_{i}}{}_{\alpha_{i}} t^{(i)}_{\beta_{i}}, \; \; A_{k-(i+1)}(\ol{t}^{(i)}_{\alpha_{i}}) = [\ol{\bbA}_{(i)}]^{\beta_{i}}{}_{\alpha_{i}} \ol{t}^{(i)}_{\beta_{i}},
\end{equation} 
thus defining matrices $\bbA_{(i)}, \ol{\bbA}_{(i)} \in \GL(r_{k+i}, \R)$. It follows from (\ref{eq_gflattog}) and (\ref{eq_ginONbasisexpl}) that 
\begin{equation}
g_{k+i}( t^{(i)}_{\alpha_{i}}) = (-1)^{k+i+1} \ol{t}^{\alpha_{i}}_{(i)}, \; \; g_{k-(i+1)}(\ol{t}^{(i)}_{\alpha_{i}}) = (-1)^{k+i} t^{\alpha_{i}}_{(i)},
\end{equation}
for each $i \in \{0,\dots,i_{\bullet}\}$ and all $\alpha_{i} \in \{1,\dots,r_{k+i}\}$. Now, by plugging everything into (\ref{eq_AOVgdegreewise2}), one arrives to two matrix equations, namely 
\begin{equation}
\ol{\bbA}{}^{T}_{(i)} \bbA_{(i)} = \1_{r_{k+i}}, \; \; \bbA^{T}_{(i)} \ol{\bbA}_{(i)} = \1_{r_{k+i}},
\end{equation}
for all $i \in \{0,\dots,i_{\bullet}\}$. Both can be solved at once by setting $\ol{\bbA}_{(i)} = \bbA_{(i)}^{-T}$. In particular, we have just proved the following statement:
\begin{tvrz}
The map $A \mapsto (\bbA_{(i)})_{i=0}^{i_{\bullet}}$ defines a Lie group isomorphism 
\begin{equation}
\gO(V_{\bullet},g) \cong \prod_{i=0}^{i_{\bullet}} \GL(r_{k+i}, \R).
\end{equation}
\end{tvrz}
\item Let $\epsilon = 0$ and let $k$ be odd. Choose a basis $(t_{\alpha})_{\alpha=1}^{n/2} \cup (\ol{t}_{\beta})_{\beta=1}^{n/2}$ as in Proposition \ref{tvrz_standardformmetric}-$(i)$. This time, we can order and relabel it so that 
\begin{align}
(t_{\alpha})_{\alpha=1}^{n/2} = & \ (t^{(0)}_{\alpha_{0}})_{\alpha_{0}=1}^{r_{k}/2} \cup \big( \cup_{i=1}^{i_{\bullet}} (t^{(i)}_{\alpha_{i}})_{\alpha_{i}=1}^{r_{k+i}} \big), \\
(\ol{t}_{\beta})_{\beta=1}^{n/2} = & \  (\ol{t}^{(0)}_{\beta_{0}})_{\beta_{0}=1}^{r_{k}/2} \cup \big( \cup_{i=1}^{i_{\bullet}} (\ol{t}^{(i)}_{\beta_{i}})_{\beta_{i}=1}^{r_{k+i}} \big),
\end{align}
where $(t_{\alpha_{i}}^{(i)})_{\alpha_{i}}^{r_{k+i}}$ is a basis for $V_{k+i}$ and $(\ol{t}_{\beta_{i}}^{(i)})_{\beta_{i}=1}^{r_{k+i}}$ is a basis for $V_{-(k+i)}$ for every $i \in \{1,\dots,i_{\bullet}\}$. Moreover, $(t_{\alpha_{0}}^{(0)})_{\alpha_{0}=1}^{r_{k}/2} \cup (\ol{t}_{\beta_{0}}^{(0)})_{\beta_{0}=1}^{r_{k}/2}$ forms a basis for $V_{k}$. 

Let us introduce the matrices $\bbA_{(i)}, \ol{\bbA}_{(i)} \in \GL(r_{k+i},\R)$ using the same formula (\ref{eq_Amatrices}) for $i \in \{1,\dots,i_{\bullet}\}$. Finally, let $\bbA_{(0)} \in \GL(r_{k},\R)$ be the matrix of $A_{k}$ in the basis $(t_{\alpha_{0}}^{(0)})_{\alpha_{0}=1}^{r_{k}/2} \cup (\ol{t}_{\beta_{0}})_{\beta_{0}=1}^{r_{k}/2}$. It is now straightforward to arrive to the matrix equations 
\begin{equation}
\ol{\bbA}{}^{T}_{(i)} \bbA_{(i)} = \1_{r_{k+i}}, \; \; \bbA^{T}_{(i)} \ol{\bbA}_{(i)} = \1_{r_{k+i}},
\end{equation}
for each $i \in \{1,\dots,i_{\bullet}\}$. Finally, observe that the matrix of $g_{k}$ in the basis $(t_{\alpha_{0}}^{(0)})_{\alpha_{0}=1}^{r_{k}/2} \cup (\ol{t}^{(0)}_{\beta_{0}})_{\beta_{0}=1}^{r_{k}/2}$ is just the standard symplectic form $\Omega$ on $\R^{r_{k}}$ and the matrix $\bbA_{(0)}$ has to satisfy
\begin{equation}
\bbA_{(0)}^{T} \Omega \bbA_{(0)} = \Omega,
\end{equation}
that is $\bbA_{(0)} \in \gSp(r_{k})$. We have just proved the following statement:  
\begin{tvrz}
The map $A \mapsto (\bbA_{(i)})_{i=0}^{i_{\bullet}}$ defines a Lie group isomorphism 
\begin{equation}
\gO(V_{\bullet},g) \cong \gSp(r_{k}) \times \prod_{i=1}^{i_{\bullet}} \GL(r_{k+i}, \R).
\end{equation}
\end{tvrz}
\item Let $\epsilon = 0$ and let $k$ be even. The discussion is very similar to part (ii), except that one uses Proposition \ref{tvrz_standardformmetric}-$(ii)$. In particular, there is a basis $(e_{a})_{a=1}^{r_{k}}$ for $V_{k}$, such that 
\begin{equation}
g(e_{a},e_{b}) = \eta_{ab},
\end{equation}
where $\eta$ is the Minkowski metric of some signature $(p,q)$. Everything proceeds as above, except that the matrix of $g_{k}$ in the basis $(e_{a})_{a=1}^{r_{k}}$ is $\eta$ and the matrix $\bbA_{(0)}$ with respect to this basis has to satisfy
\begin{equation}
\bbA_{(0)}^{T}  \eta \bbA_{(0)} = \eta,
\end{equation}
that is $\bbA_{(0)} \in \gO(p,q)$. Consequently, this proves the following statement:
\begin{tvrz}
The map $A \mapsto (\bbA_{(i)})_{i=0}^{i_{\bullet}}$ defines a Lie group isomorphism 
\begin{equation}
\gO(V_{\bullet},g) \cong \gO(p,q) \times \prod_{i=1}^{i_{\bullet}} \GL(r_{k+i}, \R),
\end{equation}
where $(p,q)$ is the signature of the restriction $g_{k}: V_{k} \times V_{k} \rightarrow \R$. 
\end{tvrz}
\end{enumerate}
\subsection{Representations} \label{subsec_repre}
The existence of a graded Lie group $\GL(V)$ allows one to define representations of graded Lie groups in a straightforward manner. First, let us recall the concept of Lie group actions. Let $(\G,\mu,\iota,e)$ be a graded Lie group. A graded smooth map $\theta: \G \times \M \rightarrow \M$ is called a \textbf{left action} of $\G$ on $\M \in \gMan^{\infty}$, if the following diagrams commute:
\begin{equation} \label{eq_actiondiagrams}
\begin{tikzcd}
\G \times (\G \times \M) \arrow{d} \arrow{r}{\1_{\G} \times \theta} & \G \times \M \arrow{dd}{\theta}\\
(\G \times \G) \times \M \arrow{d}{\mu \times \1_{\M}} & \\
\G \times \M \arrow{r}{\theta} & \M 
\end{tikzcd}, \; \; 
\begin{tikzcd}
\M \arrow{r}{(e_{\M},\1_{\M})} \arrow{rd}{\1_{\M}} &[2em] \G \times \M \arrow{d}{\theta} \\
& \M 
\end{tikzcd}.
\end{equation}
It is reasonable to expect that a general linear group $\GL(V)$ acts on the graded manifold $\dia{V}$ associated with the graded vector space $V$ we have started with. This is indeed the case. See Proposition 2.17 in \cite{Smolka2023} for the original definition. 
\begin{tvrz} \label{tvrz_standardaction}
There is a canonical left action $\theta_{V}$ of $\GL(V)$ on $\dia{V}$.  
\end{tvrz}
\begin{proof}
Let $\vartheta: \gl(V) \times V \rightarrow V$ be a canonical degree zero bilinear map $\vartheta(A,v) = A(v)$. Using (\ref{cor_bilineargrad}), it can be promoted to a graded smooth map $\dia{\vartheta}: \dia{\gl(V)} \times \dia{V} \rightarrow \dia{V}$. Let
\begin{equation}
\theta_{V} := \dia{\vartheta} \circ (k \times \1_{\dia{V}}),
\end{equation} 
where $k: \GL(V) \rightarrow \dia{\gl(V)}$ is the open embedding. To prove that $\theta_{V}$ fits into (\ref{eq_actiondiagrams}), the easiest way is to use the coordinates $(\bbx^{\lambda})_{\lambda=1}^{n}$ on $\dia{V}$ and $(\bby^{\lambda}{}_{\kappa})$ on $\GL(V)$ induced by a a choice of a total basis $(t_{\lambda})_{\lambda=1}^{n}$ for $V$. Using (\ref{eq_diabetaastexplicit}), it is not difficult to see that
\begin{equation} \label{eq_thetaVexplicit}
\theta_{V}^{\ast}(\bbx^{\lambda}) = \bbx^{\kappa} \bby^{\lambda}{}_{\kappa},
\end{equation}
where on the right-hand side, we view $\bby^{\lambda}{}_{\kappa}$ and $\bbx^{\kappa}$ as coordinates on $\GL(V) \times \dia{V}$. Using (\ref{eq_muastexplicit}), we thus find the coordinate expression 
\begin{equation}
[\theta_{V} \circ(\mu \times \1_{\M})]^{\ast}(\bbx^{\lambda}) = (\mu \times \1_{\M})^{\ast}( \bbx^{\kappa} \bby^{\lambda}{}_{\kappa}) = \bbx^{\kappa} \bbu^{\nu}{}_{\kappa} \bbz^{\lambda}{}_{\nu},
\end{equation}
where $\bbz$'s and $\bbu$'s are the coordinates on the first and the second copy of $\GL(V)$ in the product $(\GL(V) \times \GL(V)) \times \dia{V}$, respectively. On the other hand, one has 
\begin{equation}
[\theta_{V} \circ (\1_{\G} \times \theta_{V})]^{\ast}(\bbx^{\lambda}) = (\1_{\GL(V)} \times \theta_{V})^{\ast}( \bbx^{\kappa} \bby^{\lambda}{}_{\kappa}) = (\bbx^{\nu} \bbu^{\kappa}{}_{\nu}) \bbz^{\lambda}{}_{\kappa} = \bbx^{\kappa} \bbu^{\nu}{}_{\kappa} \bbz^{\lambda}{}_{\nu}. 
\end{equation}
where $\bbz$'s and $\bbu$'s are the coordinates on the first and the second copy of $\GL(V)$ in the product $\GL(V) \times (\GL(V) \times \dia{V})$. Both expressions are mapped to each other by the pullback of the canonical graded diffeomorphism, so the leftmost diagram in (\ref{eq_actiondiagrams}) commutes. Moreover, one has 
\begin{equation}
[\theta \circ (e_{\dia{V}},\1_{\dia{V}})]^{\ast}(\bbx^{\lambda}) = (e_{\dia{V}}, \1_{\dia{V}})^{\ast}( \bbx^{\kappa} \bby^{\lambda}{}_{\kappa}) = \bbx^{\kappa} \delta^{\lambda}_{\kappa} = \bbx^{\lambda} = (\1_{\dia{V}})^{\ast}(\bbx^{\lambda}). 
\end{equation}
This proves that the rightmost diagram in (\ref{eq_actiondiagrams}) commutes and the proof is complete. 
\end{proof}

\begin{definice}
A \textbf{representation of a graded Lie group $\G$ on a graded vector space $V$} is a graded Lie group homomorphism $\rho: \G \rightarrow \GL(V)$. 
\end{definice}

Similarly to the ordinary geometry, every representation induces a left action. 

\begin{tvrz}
Let $\rho: \G \rightarrow \GL(V)$ be a representation of $\G$ on $V$. Then $\theta := \theta_{V} \circ (\rho \times \1_{\dia{V}})$ is a left action of $\G$ on $\dia{V}$. Here $\theta_{V}$ is the canonical left action from Proposition \ref{tvrz_standardaction}. 
\end{tvrz}
\begin{proof}
Let us denote the group operations on $\G$ as $(\mu,\iota,e)$ and those on $\GL(V)$ as $(\mu',\iota',e')$. Let us consider a commutative diagram 
\begin{equation}
\begin{tikzcd}
\G \times (\G \times \dia{V}) \arrow{dd} \arrow{rd}{\rho \times (\rho \times \1_{\dia{V}})} \arrow{rr}{\1_{\G} \times \theta} &[2em] &[2em] \G \times \dia{V} \arrow{d}{\rho \times \1_{\dia{V}}} \arrow[bend left=460]{ddd}{\theta} \\
& \GL(V) \times (\GL(V) \times \dia{V}) \arrow{d} \arrow{r}{\1_{\GL(V)} \times \theta_{V}} & \GL(V) \times \dia{V} \arrow{dd}{\theta_{V}} \\
(\G \times \G) \times \dia{V} \arrow{r}{(\rho \times \rho) \times \1_{\dia{V}}} \arrow{dd}{\mu \times \1_{\dia{V}}} & (\GL(V) \times \GL(V)) \times \dia{V} \arrow{d}{\mu' \times \1_{\dia{V}}} & \\
& \GL(V) \times \dia{V} \arrow{r}{\theta_{V}} & \dia{V} \\
\G \times \dia{V} \arrow{ur}{\rho \times \1_{\dia{V}}} \arrow[bend right=10]{urr}{\theta} & &
\end{tikzcd}.
\end{equation}
Now, the inner diagram commutes since $\theta_{V}$ is a left action of $\GL(V)$ on $\dia{V}$. All triangles commute by definition of $\theta$. The upper-left square contains canonical diffeomorphisms and commutes trivially. The bottom-left square commutes since $\rho$ is assumed to be a graded Lie group homomorphism. Finally, the commutativity of the upper square can be checked easily by composing it with the projections onto $\GL(V)$ and $\dia{V}$, respectively, and using the definition of $\theta$. This proves the commutativity of the first diagram in (\ref{eq_actiondiagrams}). The commutativity of the second diagram is checked in a similar fashion.
\end{proof}
We can now ask which actions of $\G$ on a graded manifold $\dia{V}$ are induced by representations as in the above proposition. First, observe that for each $\lambda \in \R$, there is a degree zero graded linear map $v \mapsto \lambda v$, thus inducing a graded smooth map $H_{\lambda}: \dia{V} \rightarrow \dia{V}$, called the \textbf{homothety map} on $\dia{V}$ corresponding to $\lambda \in \R$. 

\begin{definice}
Let $\theta: \G \times \dia{V} \rightarrow \dia{V}$ be a left action of $\G$ on $\dia{V}$. We say that it is \textbf{linear}, if the diagram
\begin{equation} \label{eq_linearaction}
\begin{tikzcd}
\G \times \dia{V} \arrow{d}{\1_{\G} \times H_{\lambda}} \arrow{r}{\theta} & \dia{V} \arrow{d}{H_{\lambda}} \\
\G \times \dia{V} \arrow{r}{\theta} & \dia{V}
\end{tikzcd}
\end{equation}
commutes for each $\lambda \in \R$. 
\end{definice}
It turns out that this is precisely the sufficient and necessary condition to establish the link between actions and representations of graded Lie groups. 
\begin{tvrz}
Let $\theta: \G \times \dia{V} \rightarrow \dia{V}$ be a left action of $\G$ on $\dia{V}$. 

Then $\theta$ is linear, if and only if there exists a representation $\rho: \G \rightarrow \GL(V)$, such that $\theta = \theta_{V} \circ (\rho \times \1_{\dia{V}})$. Moreover, if such $\rho$ exists, it is unique. 
\end{tvrz}
\begin{proof}
Let us first argue that the left action $\theta_{V}$ is linear. Indeed, let us consider the diagram
\begin{equation}
\begin{tikzcd}
\GL(V) \times \dia{V} \arrow{d}{\1_{\GL(V)} \times H_{\lambda}} \arrow{r}{k \times \1_{\dia{V}}} &[2em] \dia{\gl(V)} \times \dia{V} \arrow{r}{\dia{\vartheta}} \arrow{d}{\1_{\gl(V)} \times H_{\lambda}} & \dia{V} \arrow{d}{H_{\lambda}} \\
\GL(V) \times \dia{V} \arrow{r}{k \times \1_{\dia{V}}}  & \dia{\gl(V)} \times \dia{V} \arrow{r}{\dia{\vartheta}} & \dia{V}
\end{tikzcd},
\end{equation}
for any given $\lambda \in \R$. See the proof of Proposition \ref{tvrz_standardaction} for the notation. The leftmost square commutes trivially, the rightmost square commutes since $\vartheta(A,\lambda v) = \lambda \vartheta(A,v)$ for all $A \in \gl(V)$ and $v \in V$. The horizontal arrows compose to $\theta_{V}$, proving it linear. Now, if $\theta = \theta_{V} \circ (\rho \times \1_{\dia{V}})$ for some representation $\rho: \G \rightarrow \GL(V)$, for each $\lambda \in \R$ we can consider a diagram
\begin{equation}
\begin{tikzcd}
\G \times \dia{V} \arrow{d}{\1_{\G} \times H_{\lambda}} \arrow{r}{\rho \times \1_{\dia{V}}} &[2em] \GL(V) \times \dia{V} \arrow{d}{\1_{\GL(V)} \times H_{\lambda}} \arrow{r}{\theta_{V}} & \dia{V} \arrow{d}{H_{\lambda}} \arrow{d}{H_{\lambda}}\\
\G \times \dia{V} \arrow{r}{\rho \times \1_{\dia{V}}} & \GL(V) \times \dia{V} \arrow{r}{\theta_{V}} & \dia{V}
\end{tikzcd}.
\end{equation}
The leftmost square commutes trivially, the second square commutes thanks to the linearity of $\theta_{V}$. The horizontal arrows compose to $\theta$, proving it linear. 

Conversely, suppose that $\theta: \G \times \dia{V} \rightarrow \dia{V}$ is linear. Let us consider a map $F_{\theta}: \G \times \dia{V} \rightarrow \G \times \dia{V}$ defined by $F_{\theta} := (p_{\G}, \theta)$, where $p_{\G}: \G \times \dia{V} \rightarrow \G$ is the canonical projection. Since $\theta$ is linear, the map $F_{\theta}$ fits into the commutative diagram
\begin{equation}
\begin{tikzcd}
\G \times \dia{V} \arrow{d}{\1_{\G} \times H_{\lambda}} \arrow{r}{F_{\theta}} & \G \times \dia{V} \arrow{d}{\1_{\G} \times H_{\lambda}} \\
\G \times \dia{V} \arrow{r}{F_{\theta}} & \G \times \dia{V}
\end{tikzcd}
\end{equation}
We can view $\G \times \dia{V}$ as a trivial graded vector bundle over $\G$. The commutativity of the diagram then shows that $F_{\theta}$ is a graded vector bundle morphism over $\1_{\G}$. See \S 5 of \cite{vsmolka2025threefold} for details. In fact, it is an isomorphism with the inverse given by 
\begin{equation}
F_{\theta}^{-1} = (p_{\G}, \theta \circ (\iota \times \1_{\dia{V}})),
\end{equation}
where $\iota: \G \rightarrow \G$ is the group inverse. Consequently, by Theorem 5.1 in \cite{vsmolka2025threefold}, $F_{\theta}$ corresponds to a unique degree zero $\C^{\infty}_{\G}(G)$-linear automorphism $\hat{F}_{\theta}$ of the graded module 
\begin{equation}
\Gamma_{\G \times \dia{V}}(G) = \C^{\infty}_{\G}(G) \otimes_{\R} V \equiv \frM(\G)
\end{equation}
see just above Proposition \ref{tvrz_GLVFOP} for the notation. Using the statements in the same proposition, it thus corresponds to a unique element in the set
\begin{equation}
\frP(\G) = \gMan^{\infty}(\G, \GL(V)). 
\end{equation}
This element is our representation $\rho: \G \rightarrow \GL(V)$. To prove that $\theta = \theta_{V} \circ (\rho \times \1_{\dia{V}})$, one has to dig a bit deeper into the coordinate expressions. Fix a total basis $(t_{\lambda})_{\lambda=1}^{n}$ for $V$. Let $\Phi_{\lambda} = (-1)^{|t_{\lambda}|} 1 \otimes t_{\lambda}$ be the corresponding frame for $\frM(\G)$ we have used in the proof of Proposition \ref{tvrz_GLVFOP}. We have
\begin{equation}
\hat{F}_{\theta}(\Phi_{\lambda}) = \fF^{\kappa}{}_{\lambda} \Phi_{\kappa},
\end{equation}
for unique functions $\fF^{\kappa}{}_{\lambda} \in \C^{\infty}_{\G}(G)$. They are related to $\rho: \G \rightarrow \GL(V)$ by the formula $\rho^{\ast}( \bby^{\lambda}{}_{\kappa}) = \fF^{\lambda}{}_{\kappa}$. On the other hand, it follows that $F_{\theta}: \G \times \dia{V} \rightarrow \G \times \dia{V}$ satisfies
\begin{equation}
F_{\theta}^{\ast}(\bbx^{\lambda}) = (-1)^{|t_{\kappa}|(|t_{\lambda}|-1)} p^{\ast}_{\G}( \fF^{\lambda}{}_{\kappa}) \bbx^{\kappa},
\end{equation}
where $\bbx^{\lambda}$ are viewed as functions on $\G \times \dia{V}$. But from this, we can read out the expression for $\theta$, namely 
\begin{equation} \label{eq_thetaincoordinateslinear}
\theta^{\ast}(\bbx^{\lambda}) = (-1)^{|t_{\kappa}|(|t_{\lambda}|-1)} p^{\ast}_{\G}( \fF^{\lambda}{}_{\kappa}) \bbx^{\kappa},
\end{equation}
where on the left-hand side, $\bbx^{\lambda}$ are now viewed as coordinate functions on $\dia{V}$. But observe that
\begin{equation}
\begin{split}
(\theta_{V} \circ (\rho \times \1_{\dia{V}}))^{\ast}(\bbx^{\lambda}) = & \ (\rho \times \1_{\dia{V}})^{\ast}( \bbx^{\kappa} \bby^{\lambda}{}_{\kappa}) = \bbx^{\kappa} p_{\G}^{\ast}(\fF^{\lambda}{}_{\kappa}) \\
= & \ (-1)^{|t_{\kappa}|(|t_{\lambda}| - 1)} p^{\ast}_{\G}(\fF^{\lambda}{}_{\kappa}) \bbx^{\kappa}.
\end{split}
\end{equation}
Since $\lambda \in \{1,\dots,n\}$ was arbitrary, it follows from Proposition \ref{tvrz_onpullback} that $\theta = \theta_{V} \circ (\rho \times \1_{\dia{V}})$. 

To show that $\rho$ is unique such map, let us utilize the following observation: Let $\phi: \M \rightarrow \GL(V)$ be an arbitrary smooth map. Then it is uniquely determined by the composition 
\begin{equation}
\chi := \theta_{V} \circ (\phi \times \1_{\dia{V}}): \M \times \dia{V} \rightarrow \dia{V}.
\end{equation}
Indeed, choose a basis $(t_{\lambda})_{\lambda=1}^{n}$ for $V$ to obtain the induced coordinates $(\bbx^{\lambda})_{\lambda=1}^{n}$ and $(\bby^{\lambda}{}_{\kappa})$ for $\dia{V}$ and $\GL(V)$, respectively. For each $\lambda \in \{1,\dots,n\}$, one finds 
\begin{equation}
\chi^{\ast}(\bbx^{\lambda}) = (\phi \times \1_{\dia{V}})^{\ast}( \bbx^{\kappa} \bby^{\lambda}{}_{\kappa}) = \bbx^{\kappa} p_{\M}^{\ast}( \phi^{\ast}(\bby^{\lambda}{}_{\kappa})),
\end{equation}
where we have used (\ref{eq_thetaVexplicit}) and $p_{\M}: \M \times \dia{V} \rightarrow \M$ is the projection. Let $0_{\M}: \M \rightarrow \M \times \dia{V}$ be the zero section. It satisfies $p_{\M} \circ 0_{\M} = \1_{\M}$. Observe that one can now write
\begin{equation}
\phi^{\ast}(\bby^{\lambda}{}_{\kappa}) = 0_{\M}^{\ast}( \frac{\partial}{\partial \bbx^{\kappa}} \chi^{\ast}(\bbx^{\lambda})). 
\end{equation}
This shows that $\phi$ is uniquely determined by $\chi$. Since $\theta = \theta_{V} \circ (\rho \times \1_{\dia{V}})$, this observation shows that $\rho$ is uniquely determined by $\theta$. It remains to prove that $\rho$ is a representation, that is a commutativity of the diagram
\begin{equation}
\begin{tikzcd}
\G \times \G \arrow{r}{\rho \times \rho} \arrow{d}{\mu} & \GL(V) \times \GL(V) \arrow{d}{\mu'}\\
\G \arrow{r}{\rho} & \GL(V)
\end{tikzcd}
\end{equation}
Instead of doing this directly, we can use the above observation and prove 
\begin{equation}
\theta_{V} \circ ( (\rho \circ \mu) \times \1_{\dia{V}}) = \theta_{V} \circ ( (\mu' \circ (\rho \times \rho)) \times \1_{\dia{V}})
\end{equation}
instead. But this equation can be verified easily using fact that $\theta$ and $\theta_{V}$ are left actions of $\G$ and $\GL(V)$ of $\dia{V}$, respectively. We leave this as an exercise. 
\end{proof}
\begin{example}
The representation of $\GL(V)$ corresponding to the linear action $\theta_{V}$ is obviously the identity map $\1_{\GL(V)}: \GL(V) \rightarrow \GL(V)$. 

Now, the embedding $j: \gO(V,g) \rightarrow \GL(V)$ constructed by (\ref{eq_O}) can be viewed as a representation of $\gO(V,g)$ on $\dia{V}$. The corresponding \textbf{standard action of $\gO(V,g)$ on $\dia{V}$} is thus $\theta = \theta_{V} \circ (j \times \1_{\dia{V}})$. 
\end{example}
\begin{example}
Let $\G$ be a graded Lie group, and let $\g$ be its graded Lie algebra. The map which assigns to each $x \in \g$ the corresponding left-invariant vector field $x^{L} \in \X_{\G}(G)$ can be viewed as an isomorphism of graded vector bundles $F_{L}: \G \times \dia{\g} \rightarrow T\G$, defined on ``constant sections'' as $F_{L}(1 \otimes x) := x^{L}$. There is a similarly defined isomorphism $F_{R}: \G \times \dia{\g} \rightarrow T\G$ using right-invariant vector fields. Let
\begin{equation}
\theta := p_{\dia{\g}} \circ ( F_{R}^{-1} \circ F_{L}): \G \times \dia{\g} \rightarrow \dia{\g}, 
\end{equation}
where $p_{\dia{g}}$ is the projection. We claim that $\theta$ is a left action of $\G$ on $\dia{\g}$. By construction, it is linear. We write $\Ad: \G \rightarrow \GL(\g)$ for the corresponding \textbf{adjoint representation} of $\G$. Note that the definition of $\theta$ is just a generalization of the well-known formula 
\begin{equation}
\Ad_{g}(x) = [T_{g} R_{g^{-1}}]( (T_{e} L_{g})(x)).
\end{equation}
The proof of the fact that $\theta$ is a left action of $\G$ on $\dia{\g}$ is straightforward, but requires some work. We plan to provide more details in the forthcoming self-contained paper. 
\end{example}
\bibliography{bib}
 \end{document}